\documentclass{article}
\usepackage[utf8]{inputenc}
\usepackage{fullpage}
\usepackage{fancyhdr}
\usepackage{amsmath,amssymb,amsthm}
\usepackage{tikz}
\usepackage{wrapfig} %
\usepackage{graphicx} %
\usepackage{tikz-cd}
\usepackage[margin=1in]{geometry} %
\usepackage{tipa} %
\usepackage{mathtools} %
\usepackage{caption}
\usepackage{thmtools, thm-restate} %
\usepackage{bbm} %
\usepackage[capitalize]{cleveref}
\usepackage{float} %

\usepackage{mathpartir}

\newcommand{\yields}{\vdash}

\newcommand{\jtype}{\,\,\mathsf{type}}

\newcommand{\jeq}{\equiv}

\newcommand{\ind}[3]{\mathsf{let} \; {#2} := {#1} \, \mathsf{in} \, {#3}}

\usetikzlibrary{decorations.pathreplacing} %

\graphicspath{graphics/}

\usepackage[backend=bibtex]{biblatex}
\DeclareMathOperator{\Aa}{\mathcal{A}}

\DeclareMathOperator{\Ca}{\mathcal{C}}
\DeclareMathOperator{\Da}{\mathcal{D}}
\DeclareMathOperator{\Ea}{\mathcal{E}}

\DeclareMathOperator{\Ha}{\mathcal{H}}

\DeclareMathOperator{\La}{\mathcal{L}}

\DeclareMathOperator{\Sa}{\mathcal{S}}

\DeclareMathOperator{\Za}{\mathcal{Z}}

\DeclareMathOperator{\Cb}{\mathbb{C}}
\DeclareMathOperator{\Db}{\mathbb{D}}
\DeclareMathOperator{\Eb}{\mathbb{E}}

\DeclareMathOperator{\Nb}{\mathbb{N}}

\DeclareMathOperator{\Rb}{\mathbb{R}}
\DeclareMathOperator{\Sb}{\mathbb{S}}

\DeclareMathOperator{\Zb}{\mathbb{Z}}

\newcommand{\lie}[1]{\mathbbm{#1}}
\newcommand{\B}{\textbf{B}}

\DeclareMathOperator{\im}{\mathsf{im}}

\DeclareMathOperator{\op}{^\text{op}}
\DeclareMathOperator{\inv}{^{-1}}
\DeclareMathOperator{\id}{\mathsf{id}}
\renewcommand{\ker}{\mathsf{ker}}

\DeclareMathOperator{\Type}{\textbf{Type}}

\DeclareMathOperator{\refl}{\mathsf{refl}}
\DeclareMathOperator{\fib}{\mathsf{fib}}

\DeclareMathOperator{\fst}{\term{fst}}
\DeclareMathOperator{\snd}{\term{snd}}

\DeclareMathOperator{\ap}{\term{ap}}

\DeclareMathOperator{\shape}{\text{\textesh}}

\newcommand{\sslash}{\mathbin{/\mkern-6mu/}} %
\newcommand{\inftycover}[1]{\overset{\infty}{#1}}

\DeclareMathOperator{\Set}{\textbf{Set}}

\DeclareMathOperator{\Id}{\textbf{Id}}

\newcommand{\dprod}[1]{({#1}) \to }           %
\newcommand{\dsum}[1]{({#1}) \times }       %
\newcommand{\type}[1]{\mathsf{{#1}}}
\newcommand{\term}[1]{\mathsf{{#1}}}

\newcommand{\trunc}[1]{\left\lVert#1\right\rVert}       %

\newcommand{\pt}{\term{pt}}

\newcommand{\xto}[1]{\xrightarrow{#1}}

\newcommand{\pto}{\,\cdot\kern-.1em{\to}\,}
\newcommand{\pxto}[1]{\,\cdot\kern-.1em{\xto{#1}}\,}

\makeatletter
\providecommand*{\xmapstofill@}{%
  \arrowfill@{\mapstochar\relbar}\relbar\rightarrow
}
\providecommand*{\xmapsto}[2][]{%
  \ext@arrow 0395\xmapstofill@{#1}{#2}%
}

\makeatletter
\def\slashedarrowfill@#1#2#3#4#5{%
  $\m@th\thickmuskip0mu\medmuskip\thickmuskip\thinmuskip\thickmuskip
   \relax#5#1\mkern-7mu%
   \cleaders\hbox{$#5\mkern-2mu#2\mkern-2mu$}\hfill
   \mathclap{#3}\mathclap{#2}%
   \cleaders\hbox{$#5\mkern-2mu#2\mkern-2mu$}\hfill
   \mkern-7mu#4$%
}
\def\rightslashedarrowfill@{%
  \slashedarrowfill@\relbar\relbar\mapstochar\rightarrow}
\newcommand\xslashedrightarrow[2][]{%
  \ext@arrow 0055{\rightslashedarrowfill@}{#1}{#2}}
\makeatother

\tikzset{ vert/.style={anchor=south, rotate=90, inner sep=.5mm} }

\newtheorem{thm}{Theorem}[subsection]

\theoremstyle{definition}
\newtheorem{defn}[thm]{Definition}

\newtheorem{ex}[thm]{Example}
\newtheorem{axiom}{Axiom}
\newtheorem{assumption}{Assumption}

\newtheorem{rmk}[thm]{Remark}
\newtheorem{warning}[thm]{Warning}

\newtheorem*{acknowledgements}{Acknowledgements}

\newtheorem{lem}[thm]{Lemma}
\newtheorem{notation}{Notation}

\newtheorem{conjecture}[thm]{Conjecture}
\newtheorem{cor}[thm]{Corollary}
\newtheorem{prop}[thm]{Proposition}
\newtheorem{con}[thm]{Construction}

\makeatletter
\let\c@equation\c@thm
\makeatother
\numberwithin{equation}{subsection}



\title{Modal Fracture of Higher Groups}
\author{David Jaz Myers}

\begin{document}
\maketitle

\begin{abstract}

In this paper, we examine the modal aspects of higher groups in
Shulman's Cohesive Homotopy Type Theory. We show that every higher group sits
within a modal fracture hexagon which renders it into its discrete,
infinitesimal, and contractible components. This gives an unstable and synthetic
construction of Schreiber's differential cohomology hexagon. As an example of this modal fracture
hexagon, we recover the character diagram characterizing
ordinary differential cohomology by its relation to its underlying integral
cohomology and differential form data, although there is a subtle obstruction to
generalizing the usual hexagon to higher types. Assuming the existence of a long
exact sequence of differential form classifiers, we construct the classifiers for circle $k$-gerbes
with connection and describe their modal fracture hexagon. 
  
\end{abstract}

\tableofcontents

\begin{acknowledgements}
  I would like to thank Emily Riehl, Mitchell Riley, and Grigorios Giotopolous for their careful reading and helpful comments
  during the drafting of this paper. I also appreciate the reviewer's comments and suggestion to include an introduction to homotopy type theory in the paper. The author is grateful for financial
  support from the ARO under MURI Grant W911NF-20-1-0082 and by Tamkeen under the NYU Abu Dhabi Research Institute
  Grant CG008.
\end{acknowledgements}

\section{Introduction}

There are many situations where cohomology is useful but we need more than
just the information of cohomology classes and their relations in cohomology --- we need the information of
specific cocycles which give rise to those classes and cochains which witness these
relations. A striking example of this is \emph{ordinary differential cohomology}. To give a home for
calculations done in \cite{Chern-Simons:Characteristic.Forms}, Cheeger and Simons
\cite{Cheeger-Simons:Differential.Characters} gave a series of lectures in 1973 defining and studying
\emph{differential characters}, which equip classes in ordinary integral
cohomology with explicit differential form representatives. Slightly earlier, Deligne \cite{Deligne:Cohomology} had put forward a
cohomology theory in the complex analytic setting which would go on to be called
Deligne cohomology. It was later realized that when placed in the differential
geometric setting, Deligne cohomology gave a presentation of the theory of
differential characters. This combined theory has become known as ordinary
  differential cohomology.

The ordinary differential cohomology $D_k(X)$ of a manifold $X$ is characterized by its relationship to
the ordinary cohomology of $X$ and the differential forms on $X$ by a diagram known as
the \emph{differential cohomology hexagon} or the \emph{character diagram} \cite{Simons-Sullivan:Axiomatic.Ordinary.Differential.Cohomology}:
\begin{equation}\label{eqn:character.diagram}
\begin{tikzcd}
	& {\Lambda^k(X)/\text{im}(d)} && {\Lambda^{k + 1}_{\text{cl}}(X)} \\
	{H^k(X; \Rb)} && {D_k(X)} && {H^{k + 1}(X; \Rb)} \\
	& {H^{k}(X; U(1))} & {} & {H^{k + 1}(X ; \Zb)}
	\arrow[from=2-1, to=1-2]
	\arrow[from=2-1, to=3-2]
	\arrow[from=1-2, to=2-3]
	\arrow[from=3-2, to=2-3]
	\arrow["d", from=1-2, to=1-4]
	\arrow["\beta"', from=3-2, to=3-4]
	\arrow[from=1-4, to=2-5]
	\arrow[from=2-3, to=3-4]
	\arrow[from=2-3, to=1-4]
	\arrow[from=3-4, to=2-5]
\end{tikzcd}
\end{equation}
In this diagram, the top and bottom sequences are long exact, and the diagonal
sequences are exact in the middle. The bottom sequence is the Bockstein sequence
associated to the universal cover short exact sequence
$$0 \to \Zb \to \Rb \to U(1) \to 0$$
while the top sequence is given by de Rham's theorem representing real
cohomology classes by differential forms.

This sort of diagram is characteristic of differential cohomology theories in
general. Bunke, Nikolaus, and V{\:o}kel \cite{Bunke-Nikolaus-Volke:Differential.Cohomology}
interpret differential cohomology theories as sheaves on the site of smooth
manifolds and construct differential cohomology hexagons very generally in this setting:
\[
  \begin{tikzcd}
    & \Aa(\hat{E})^k(X)\ar[rr] \ar[dr] & & \Za(\hat{E})^k(X)\ar[dr] & \\
    H^{k-1}(X; Z(\hat{E})) \ar[ur] \ar[dr]& & \hat{E}^k(X) \ar[ur] \ar[dr] & & H^k(X; Z(\hat{E})) \\
    & H^k(X; S(\hat{E})) \ar[rr]\ar[ur] & & H^k(X; U(E)) \ar[ur]
  \end{tikzcd}
\]
Here, $\hat{E}$ is the differential cohomology theory, $U(\hat{E})$ is its
underlying topological cohomology, $\Za(\hat{E})$ are the differential cycles,
$S$ is the secondary cohomology theory given by flat classes, and $\Aa$
classifies differential deformations (this summary discussion is lifted from
\cite{Bunke-Nikolaus-Volke:Differential.Cohomology}). Here, as with ordinary differential cohomology, the top
and bottom sequences are exact, and the diagonal sequences are exact in the middle.

The arguments of Bunke, Nikolaus, and V{\:o}kel are abstract and \emph{modal} in
character --- they only make essential use of an adjunction between an
idempotent (co)monads.\footnote{The term ``modal'' is taken from the classic
  alethic modalities of possibility and necessity, which are idempotent
  (co)monads on suitable categories of predicates.} This is emphasized by Schreiber in his book
\cite{Schreiber:Differential.Cohomology}, where he constructs similar diagrams
in the setting of an adjoint triple 
\begin{equation}\label{eqn:adjoint.diagram}
\begin{tikzcd}
	\Ea \\
	\\
	\Ha
	\arrow[""{name=0, anchor=center, inner sep=0}, "\Delta"{description}, hook, from=3-1, to=1-1]
	\arrow[""{name=1, anchor=center, inner sep=0}, "{\Pi_{\infty}}"', bend right =
  50, from=1-1, to=3-1]
	\arrow[""{name=2, anchor=center, inner sep=0}, "\Gamma", bend left = 50, from=1-1, to=3-1]
	\arrow["\dashv"{description}, Rightarrow, draw=none, from=1, to=0]
	\arrow["\dashv"{description}, Rightarrow, draw=none, from=0, to=2]
\end{tikzcd}
\end{equation}

in which the middle functor is fully faithful and the leftmost adjoint
$\Pi_{\infty}$ preserves
products. In the case that $\Gamma$ is the
global sections functor of an $\infty$-topos $\Ea$ landing in the $\infty$-topos
of homotopy types $\Ha$, this structure makes $\Ea$
into a \emph{strongly $\infty$-connected} $\infty$-topos. When $\Ea$
is the $\infty$-topos of sheaves of homotopy types on manifolds, the leftmost adjoint $\Pi_{\infty}$
is given by localizing at the sheaf of real-valued functions; for a
representable, this recovers the homotopy type or fundamental $\infty$-groupoid
of the manifold. Schreiber shows in Proposition 4.1.17 of
\cite{Schreiber:Differential.Cohomology} that any such adjoint triple gives
rise to differential cohomology hexagons, specializing to those of Bunke,
Nikolaus, and V{\:o}kel in the case that $\Ea$ is the $\infty$-topos of sheaves of homotopy types on
smooth manifolds.

This abstract re-reading of the differential cohomology hexagon shows that
there is nothing specifically ``differential'' about them, and that they arise
in situations where there is no differential calculus to be found. Schreiber
emphasizes this point in an $n$Lab article \cite{nlab:differential_cohesion_and_idelic_structure} where he refigures these
hexagons as \emph{modal fracture squares}. To recover more traditional fracture theorems, Schreiber considers the
case where $\Ea = A\text{-Mod}$ is the $\infty$-category of modules over an
$\Eb^2$-ring $A$; $\Gamma = \Gamma_I$ is the reflection into $I$-nilpotent
modules (with $I \subseteq \pi_0 A$ a finitely generated ideal) constructed by
Lurie in \cite[Notation 4.1.13]{Lurie:DAGXII}, and $\Pi_{\infty} M = M_I^{\wedge}$ is
the
$I$-completion constructed in \cite[Notation 4.2.3]{Lurie:DAGXII}. The subcategories of $I$-nilpotent and $I$-complete modules are distinct but equivalent, allowing us to see them as a single
$\infty$-category $\Ha$.

In this paper, we will construct the \emph{modal fracture hexagon}
associated to a higher group --- a group object in the $\infty$-topos of
stacks on a given site --- \emph{synthetically}, working in an
appropriately modal homotopy type theory. We will then specialize to ordinary
differential cohomology on the site of smooth manifolds, constructing the
classifying stacks for $U(1)$-bundle gerbes with connection from a few basic
assumptions, computing their modal fracture hexagon, and relating it to the
traditional character diagram.

We hope that the simplicity and directness of the
arguments we give here will serve as ample advertisement for the use of modal
homotopy type theory as a language for doing higher differential geometry. To
that end, we begin by introducing modal homotopy type theory for those unfamiliar.
Those already familiar may want to skip to \cref{sec:modal.fracture}.

\subsection{Convenient categories of smooth spaces}\label{sec:convenient}

The category of smooth manifolds is somewhat deficient from a categorical point of view. Manifolds are not closed under taking pullback (fiber product), for example, and there is no mapping space functor right adjoint to the categorical product. The category of manifolds also lacks many colimit constructions such as coequalizers. This is due to the \emph{niceness} of manifolds; they are spaces which are locally diffeomorphic to $\Rb^n$. It is a truism of category theory that categories of nice objects tend to be ill-behaved, while categories of perhaps ill-behaved objects are free to be nice. If we expand our notion of smooth space to include any ``space'' that can be probed by arbitrary smooth functions (and not just diffeomorphisms), we will arrive at a much more convenient category of smooth spaces. 

This is the idea behind Chen's notion of \emph{differentiable spaces} \cite{Chen:diffeo} and Souriau's more general \emph{diffeological spaces} \cite{Souriau:diffeo} (see \cite{BH:smooth.spaces} for an overview and comparison). The definition of diffeological spaces is illustrative, so let's see it.

\begin{defn}
  A \emph{diffeological space} is a set $X$ together with, for every open subset $U \subseteq \Rb^n$ of Euclidean space (for any $n$), a subset $X(U) \subseteq \type{Fun}(U, X)$ of functions from $U$ to $X$ whose elements are called ``plots". These sets are required to satisfy the following axioms:
  \begin{enumerate}
    \item Plots are closed under precomposition by smooth functions. That is, if $\phi \in X(U)$ is a plot and $f : U' \to U$ is a smooth function, then the composite $\phi \circ f$ is in $X(U')$.
    \item Plots glue over open covers. That is, suppose that the family $(U_i)_{i \in I}$ covers $U$ with inclusions $f_i : U_i \hookrightarrow U$, and let $\phi : U \to X$. Whenever $\phi \circ f_i \in X(U_i)$ for all $i$, then $\phi \in X(U)$.
    \item Every point is a plot. That is, every function $\Rb^0 \to X$ is a plot.
  \end{enumerate}
\end{defn}

The notion of ``plot" expands the notion of chart or local coordinates to possibly degenerate ``coordinates" given by any smooth parameterization of elements of $X$. The category of diffeological spaces has all pullbacks and mapping spaces (it is furthermore \emph{locally cartesian closed}), and it has all colimits. Nicely, the category of smooth manifolds fully faithfully embeds into the category of diffeological spaces by defining the plots $M(U)$ to be the smooth functions $U \to M$. However, this category still has a few difficiences; it is not \emph{balanced}, for example --- there are monic epimorphisms which are not isomorphisms.

The school of \emph{synthetic differential geometry} developed by Lawvere, Dubuc, Penon, Bunge, and many others (see \cref{rmk:future.work.sdg} for a brief overview) make a further leap: it is not necessary to assume we have the set $X$ of points given a priori, since we can recover it as the set of zero dimensional plots $X(\Rb^0)$. This leads to a definition of \emph{smooth sets}\footnote{This term is due to Schreiber, and appears in, for example, \cite{khavkine2017synthetic}, as well as in \cite{Schreiber:Differential.Cohomology}.} as sheaves on the site of open subsets of Euclidean spaces and smooth maps:
\begin{defn}
Let $\Omega_{\type{Euc}}$ be the category of open subsets of Euclidean spaces and smooth maps between them. A \emph{smooth set} is a functor $X : \Omega_{\type{Euc}}\op \to \Set$ assigning to each open subset $U \subseteq \Rb^n$ a set $X(U)$ of ``plots" together with a right action $\varphi \mapsto \varphi \cdot f : X(U) \to X(U')$ by smooth functions $f : U' \to U$, subject to the following \emph{sheaf conditions}:
\begin{enumerate}
  \item Suppose that the family $(U_i)_{i \in I}$ covers $U$ with inclusions $f_i : U_i \hookrightarrow U$. Let $f_{ij}^i : U_i \cap U_j \hookrightarrow U_i$ and similarly $f_{ij}^j : U_i \cap U_j \to U_j$ be the inclusions. For every family $\varphi_i \in X(U_i)$ of plots such that $\varphi_i \cdot f_{ij}^i = \varphi_j \cdot f_{ij}^i$, there exists a unique $\varphi \in X(U)$ so that $\varphi_i = \varphi \cdot f_i$. 
\end{enumerate}
A map of smooth sets is a natural transformation: a family of maps $g_U : X(U) \to Y(U)$ such that $g_U(\varphi) \cdot f_i = g_{U'}(\varphi \cdot f_i)$ for all $\varphi \in X(U)$ and smooth $f : U' \to U$.\footnote{In fact, since every open $U \subseteq \Rb^n$ may be covered by open balls, and since open balls are diffeomorphic to $\Rb^n$, Schreiber observes in \cite{Schreiber:Differential.Cohomology} that it suffices to define the set of plots $X(\Rb^n)$ for each $n$. Similary, but in the other direction, we may define a smooth set by plots $X(M)$ from any smooth manifold; the gluing condition assures us that this definition may be reduced to the one given above, since manifolds are locally diffeomorphic to open subsets $U$ of Euclidean spaces.}
\end{defn}

The leap from diffeological spaces to smooth sets is guided by \emph{representability}. More notions become representable in the category of smooth sets than are in diffeological spaces. This includes our earlier examples of fiber products and mapping spaces, both of which are defined by the functors they represent. But it also includes more alien examples of smooth spaces. For example, there is a smooth set $\Lambda^n$ which represents $n$-forms: maps $M \to \Lambda^n$ from a manifold correspond bijectively with $n$-forms on $M$. 
The definition of $\Lambda^n$ is forced by the Yoneda lemma: plots $\Lambda^n(U)$ are in bijection with maps $U \to \Lambda^n$, which we want to be in bijection with $n$-forms on $U$. Precomposing a map $U \to \Lambda^n$ by an inclusion $V \hookrightarrow U$ should induce the pullack on forms.
Therefore, we must take $\Lambda^n(U)$ to be the set of $n$-forms on $U$ and define the functorial action $\omega \cdot f := f^{\ast}\omega$ by pullback. Note that $\Lambda^n$ is far from being a diffeological space; there is exactly one element of $\Lambda^n(\Rb^0)$ (the unique $n$-form on the point), but of course many $n$-forms on an open $U$ of dimension at least $n$. These sheaves will play an important role in our synthetic definition of differential cohomology in \cref{sec:diff.coh}. 

Smooth sets have been used to given a synthetic account of PDEs in \cite{SK:pde}, and provide a good setting for making precise the common manipulations performed by physicists on infinite dimensional spaces of fields; see \cite{GH:field.theory} for an extensive introduction to smooth sets for this use. 

It is a somewhat under-appreciated fact that toposes of sheaves of sets --- the category of smooth sets, for example --- are themselves models of intuitionistic set theory (see, for example, \cite{MM:sheaves.in.geometry.and.logic}). That is, arguments given in ordinary set theory which make no use of the axiom of choice or the law of excluded middle\footnote{This avoidance of the axiom of choice and excluded middle is not on philosophical grounds, but rather topological grounds. The law of excluded middle states that any proposition is either true or false; seems obvious enough, but note that this may be equivalently expressed as saying that any subset $U \subseteq X$ admits a characteristic function $\chi : X \to \{0, 1\}$ (defined by $\chi(x) = 1$ if $x \in U$ and $0$ otherwise, which covers every case by the law of excluded middle: either $x \in U$ or $x \not\in U$). But since all functions between smooth sets are themselves smooth, such a characteristic function would need to be smooth and therefore constant on connected components of $X$ --- in other words, a general manifold is not the \emph{disjoint} union of any of its subsets and its complement. Failure of excluded middle therefore detects the \emph{connectedness} of objects. Similarly, the axiom of choice says that any surjection admits a section; but it is not true that any cover of a manifold admits a (global) smooth section.} may be interpreted as statements about sheaves of sets in any topos, and in particular as statements about smooth sets. Since the category of smooth manifolds embeds fully faithfully into the category of smooth sets, we end up with a set theory where ``every function is smooth" --- without having to check any conditions about it. 
This observation underlies the development of \emph{synthetic differential geometry}, which develops differential geometric notions in the topos of smooth sets (or, most commonly, a slightly richer topos in which the tangent bundle functor is representable by an infinitesimal line, a ``walking tangent vector"). See Bell's \emph{Introduction to smooth infintesimal analysis} \cite{Bell:smooth.infinitesimal.analysis} for an introduction to synthetic differential geometry from the (more naive) set theoretic, as opposed to sheaf theoretic, point of view.

Smooth sets are not, however, the end of the search for a truly convenient category of smooth spaces. Just as the definition of diffeological spaces excluded smooth sets like $\Lambda^n$ which are easily definable as representing a functor on manifolds, the definition of smooth sets excludes other objects which are almost sheaves, but for which the equation in the gluing condition is too strong a constraint. For example, consider the assignement $U \mapsto \{\mbox{Hermitian line bundles on $U$}\}$; this is almost a sheaf, except that it satisfies the following weaker version of the gluing condition:
\begin{itemize}
  \item Suppose that the family $(U_i)_{i \in I}$ covers $U$, with inclusions $f_i : U_i \hookrightarrow U$. Let $f_{ij}^i : U_i \cap U_j \hookrightarrow U_i$ and similarly $f_{ij}^j : U_i \cap U_j \to U_j$ be the restricted inclusions. Now suppose that $\La_i$ is a Hermitian line bundle on $U_i$ and that we have isomorphisms $\rho_{ij} : (f_{ij}^i)^{\ast}\La_i  \cong  (f_{ij}^j)^{\ast}\La_j$ which satisfy the cocycle condition $\rho_{jk} \circ \rho_{ij} = \rho_{ik}$ on triple intersections $U_i \cap U_j \cap U_k$; then there exists a line bundle $\La$ on $U$ with isomorphisms $\lambda_i : f_i^{\ast}\La \cong \La_i$ so that $\rho_{ij} = (f_{ij}^j)^{\ast}\lambda_j \circ (f_{ij}^i)^{\ast}\lambda_i^{-1}$, and the $\La$ and $\lambda_i$ are the unique such extension, up to unique isomorphism. 
\end{itemize}
The assignment $U \mapsto \{\mbox{Hermitian line bundles on $U$}\}$ is therefore not a sheaf of sets, but rather a sheaf of groupoids --- a \emph{stack}. This stack in particular is $\B U(1)$, the classifying stack for Hermitian line bundles; the groupoid of maps $M \to \B U(1)$ of stacks from (the Yoneda image of) a manifold $M$ to $\B U(1)$ is equivalent to the groupoid of Hermitian line bundles on $M$.

In general, to capture not only the classifying stacks for Hermitian line bundles but also those for bundle gerbes and bundle gerbes with connection (which represent classes in ordinary differential cohomology), we are led to work with stacks of higher groupoids on the site of smooth manifolds. But the gluing conditions for $n$-stacks become more complicated as $n$ grows. Taking for granted Grothendieck's \emph{homotopy hypothesis}, the limiting case of $\infty$-groupoids should coincide with the homotopy types of topological spaces as presented, for example, by Kan complexes. It is therefore more convenient to encode the gluing condition as a model structure on the category simplicial presheaves, which is the approach taken by Schreiber in his extensive work in this area \cite{Schreiber:Differential.Cohomology}, or to work entirely with $\infty$-categories as is done by Bunke, Nikolaus, and V{\"o}lkl \cite{Bunke-Nikolaus-Volke:Differential.Cohomology}.

Both \cite{Schreiber:Differential.Cohomology} and \cite{Bunke-Nikolaus-Volke:Differential.Cohomology} emphasize sheaves of $\infty$-groupoids on manifolds as the appropriate setting for differential cohomology theories. Remarkably, just as toposes of sheaves of sets are models of intuitionistic set theory, $\infty$-toposes of sheaves of $\infty$-groupoids are models of a novel foundation of mathematics: \emph{homotopy type theory}. In this paper, we will demonstrate how concise and conceptual arguments in homotopy type theory can be when applied to sheaves of $\infty$-groupoids by constructing the classifying stacks for ordinary differential cohomology from first principles in \cref{sec:diff.coh}.

\subsection{A brief introduction to homotopy type theory.}

Homotopy type theory (HoTT, or sometimes Univalent Foundations \cite{HoTTBook}) is a logical system which lets us work directly with higher stacks as if they
were defined by their elements, just as we work with ordinary sets. To be more precise,
HoTT is a formal language which admits a \emph{categorical semantics} in
$\infty$-toposes\footnote{This is an difficult theorem whose proof was begun by
  Voevodsky, Lumsdaine, and Kapulkin \cite{LK:hott.semantics} (by giving semantics in the
  $\infty$-topos of homotopy types) and completed by Shulman \cite{Shulman:hott.semantics}} --- that is, expressions in homotopy type theory may be
interpreted as constructions in $\infty$-categories of higher stacks. For comparison, ZFC
(Zermelo-Fraenkel set theory with Choice) may be interpreted in toposes of
sheaves on extremally disconnected compact Hausdorff spaces (among others), ZF (without choice)
may be interpreted in sheaves on Stone spaces and toposes of actions of totally
disconnected groups (among others), and intuitionistic ZF (without the law of
excluded middle or choice) may be interpreted in any Grothendieck topos using
the Kripke-Joyal semantics (see, e.g. \cite{MM:sheaves.in.geometry.and.logic}).
From the point of view of the ``internal'' set theory of a topos of sheaves, a sheaf of rings is just a ring, a sheaf of modules is just a module, and a coherent sheaf is a finitely presented module. In this way, the internal language lets us reduce complicated arguments about sheaves to familiar arguments about sets; see \cite{Blechschmidt:thesis} for many examples of how this technique can be used in algebraic geometry.

Generally speaking, while intuitionistic set theory gives us a language for working with
sheaves of sets as if they were just ordinary sets, homotopy type theory gives us a
language for working with higher stacks as if they were just\ldots what exactly?

The answer to this elliptic question is that we will treat higher stacks as if they were
just \emph{types of mathematical objects}. To understand what that means, we
must consider HoTT on its own terms as
a naive and elementary foundation for mathematical practice, just as we do for
(naive) set theory. This does not mean that we have to change how we do all our mathematics; it just means that when using homotopy type theory, we should take it on its own terms and not always try to reduce it to statements about stacks. Things look a lot easier from this point of view, at least when it has been digested.

As a foundation,
HoTT is based on the following elementary principles:
\begin{itemize}
\item In order for a variable $x$ to make sense, we must know what type of
  mathematical object
  $x$ is meant to be. Is $x$ a natural number? A real number? A smooth manifold?
  A tangent vector to a point $p$ in a manifold $M$? These sorts of questions
  are not to be proved or disproved; when we use the variable $x$, we already
  know what type of thing $x$ is meant to be. We never say ``Let $x$, and if $x$
  is a smooth manifold\ldots''; instead, we say ``Let $x$ be a smooth
  manifold\ldots''.

  We write $x : X$ to say that object $x$ has type $X$, or that $x$ is an
  element of type $X$.\footnote{We reserve the set membership symbol $x \in S$ for
    the \emph{proposition} that $x$ lies in the subtype $S$ of the type $X$ of
    $x$. This proposition doesn't make sense unless we know $x : X$ already; we
    never ask whether a smooth manifold lies in the set of even numbers, for example.} For example, $3 : \Nb$
  says that $3$ is a natural number --- a \emph{set} like $\Nb$ is just a
  particular sort of type (see \cref{defn:truncation}).

  \item A type $T$ of mathematical objects (e.g. \emph{natural number},
    \emph{complex number}, \emph{smooth manifold}, \emph{vector field on
      $M$}\ldots) is itself a mathematical object --- it is a \emph{type}, $T :
    \Type$. This applies equally well to the type $\Type$ of types, so that (naively) we
    have $\Type : \Type$.\footnote{Just as naive set theory is inconsistent,
      this naive type theory is inconsistent. Most approaches to proving the
      inconsistency of $\Type : \Type$ begin by constructing ordinals (as in
      \cite{Girard:paradox}) or a model of set theory (as in
      \cite{Miquel:thesis}), and then showing that $\Type : \Type$ enables a
      familiar paradox in these settings (the Burali-Forti paradox in the former
      case, and Russell's paradox in the latter).

      We avoid this contradiction by
    using a hierarchy of \emph{type universes} $\Type_{\ell}$, which, similar to
    Grothendieck universes in set theory, are types
    of types closed under all operations of type theory. We assume that
    $\Type_{\ell} : \Type_{\ell + 1}$, and thereby avoid paradox. As a brief
    remark on terminology, note that just as set theory reserves the word
    ``class'' for sets that are too large to be consistent, type theory reserves
  the word ``kind'' for types that are too large to be consistent; we will not
  pay much attention to these details in this paper. Suffice it to say that type
theory is equiconsistent with set theory, and can suffer from the same
self-referential paradoxes.}

\item\label{identification.desiderata} For any two mathematical objects $a$ and $b$ of the same type $A$, we know
  what it means to \emph{identify} them as elements of that type. For example:
  \begin{itemize}
\item To identify a real vector space $V$ with $\Rb^n$, we may give an $n$-element
 ordered basis of $V$. More generally, two identify two real vector spaces $V$ and $W$,
  we give a linear isomorphism between them.
\item To identify two smooth manifolds $M$ and $N$, we give a diffeomorphism
  between them.
  \item To identify two categories $\Ca$ and $\Da$, we give an equivalence
    between them.
    \item To identify two numbers, we simply prove them 
      equal. 
      \item To identify two subsets of a given set, we show that they contain the same elements.
      \item To identify two types, we put their elements into a one-to-one
        correspondence --- we give an \emph{equivalence} between them.
  \end{itemize}
  Clearly, an identification of $a$ with $b$ as elements of $A$ is
  itself a mathematical object, and so there is a type $a =_A b$ of such
  identifications. We use the ``equals'' symbol $=$ because
  \emph{identification} is a generalization of equality to more general
  mathematical objects where equality is too strict a notion; in particular, for
  things like numbers or points in a manifold, identification is equality. For sets with structure, 
  such vector spaces, groups, Lie algebras and so on, identification is isomorphism. 
\end{itemize}

These principles are meant to be very simple and fundamental. To do
the same exercise for naive set theory, one might say that it is defined by the
principles that mathematical objects may be organized into collections ---
\emph{sets} --- on the basis of a shared property; that collections themselves
are mathematical objects and so can be collected on the basis of
their shared properties; and that two collections are equal when they contain 
the same members.

The notion of \emph{identification} of mathematical objects is
not formalized in set theory. Nevertheless, we have a solid heuristic
understanding of how to identify mathematical objects. HoTT takes the notion of
identification as primitive, and formal systems for HoTT give methods for
computing the type $a =_A b$ of identifications between $a$ and $b$ as
elements of type $A$ --- in all the above cases and more, we get the answer we
expect.\footnote{I would go further and say that the notion of identification we
get through homotopy type theory is \emph{always} the right one, and that if we
are getting a notion of identification we don't like, we are really considering
$a$ and $b$ as elements of the wrong type. Changing the type $A$ will change the
the notion of identification.} 

There are a number of formal systems that realize the above principles --- such as Book HoTT
\cite{HoTTBook}, cubical type theory \cite{CCHM:cubical}, and the in-development
higher observational type theory \cite{AS:HOTT}. These differ in their
implementation of the above principles, but share enough features for us to work
informally in a formalization-agnostic manner in much the way that ordinary
mathematics relates to formal set theories. Let's see these common features
now.\footnote{For a formal presentation, see the appendix of \cite{HoTTBook}.}

Following Martin-L{\"o}f \cite{ML:meaning.of.logical.connectives}, type theory
concerns \emph{judgements} that something is a type, and that something is an element of a given type.
We write the judgement that $D(x, y, z)$ is a type when $x$ is of type $A$, $y$ is of type $B(x)$, and $z$ is of type $C(x, y)$ as follows:
\begin{mathpar}
  x : A, y : B(x), z : C(x, y) \yields D(x, y, z) \jtype
\end{mathpar}
The symbol $\yields$ is called a ``turnstile" and originates with Frege \cite{Frege:turnstile}; it is used in logic to mean ``entails" or ``yields".
We write the judgement that $d(x, y, z)$ is an element of type $D(x, y, z)$ under the same assumptions in this way:
\begin{mathpar}
  x : A, y : B(x), z : C(x, y) \yields d(x, y, z) : D(x, y, z) 
\end{mathpar}
The list $x : A, y : B(x), z : C(x, y)$ of typed free variables is known as the \emph{context} of the element $d(x, y, z)$
and its type $D(x, y, z)$.  Note that the types of the variables appearing later in the context may themselves depend on variables appearing earlier
in the context. 

For a concrete example of such judgements, the value $g(x, y) : \Rb$ of a metric at two tangent vectors $x$ and $y$ would be expressed as follows: 
\begin{mathpar}
M : \type{Manifold}, p : M, x : T_p M, y : T_p M \yields g(x, y) : \Rb.
\end{mathpar}
The power of type theory comes from this explicit treatment of free variables. 

When interpreting type theory in a topos $\Ea$, we interpret each context $\Gamma$ as an object of $\Ea$. The empty context, with no free variables, is interpreted as the terminal object of $\Ea$. We interpret a typing judgement $\Gamma \yields D \jtype$ as a map $\pi : (\Gamma,D) \to \Gamma$ into the interpretation of $\Gamma$ (which, when interpreting in a model category, should be a \emph{fibration} over the interpretation of $\Gamma$), 
and we interpret an element $\Gamma \yields d : D$ as a section of the interpretation of $D$. The context $\Gamma, x : D$ extended by a free variable of the type $D$ is interpreted as the domain $(\Gamma,D)$ of the interpretation of the type judgement $\Gamma \yields D \jtype$. Substitution of elements for variables is interpreted by pullback. In other words, if our type theory is interpreted into the $\infty$-topos $\Ea$, then statements in context $\Gamma$ are interpreted in the slice topos $\Ea_{/\Gamma}$ over the interpretation of $\Gamma$; we base-change between slice toposes by substituting in elements for variables. For a comprehensive introduction to the semantics of homotopy type theory, see \cite{Riehl:Semantics.of.HoTT}.

\begin{table}[H]\label{table:type.interpret}
\begin{tabular}{|c||c|}
  \hline
\begin{minipage}{7cm}
{\bf Type theory} (syntax)
\end{minipage}
&
\begin{minipage}{7cm}
{\bf Higher topos theory} (semantics)
\end{minipage}
\\
\hline
\hline
\begin{minipage}{7cm}
  \[\Gamma \yields D \jtype \]
  \end{minipage}
  &
  \begin{minipage}{7cm}
    \[
    \begin{tikzcd}
     (\Gamma,D) \ar[d] \\ \Gamma 
    \end{tikzcd}
    \]
  \end{minipage}
  \\
  \hline
  \begin{minipage}{7cm}
    \[\Gamma \yields d : D \]
    \end{minipage}
    &
    \begin{minipage}{7cm}
      \[
      \begin{tikzcd}
        (\Gamma,D) \ar[d] \ar[d, dashed, bend right, leftarrow, "d"']\\ \Gamma 
      \end{tikzcd}
      \]
    \end{minipage}
    \\
    \hline
    \begin{minipage}{7cm}
      \begin{mathpar}
\inferrule{\mbox{Given }\Delta \yields a : A \mbox{  and  }x : A \yields d(x) : D(x)}
{\mbox{We deduce    } \Delta \yields d(a) : D(a)}
      \end{mathpar}
      \end{minipage}
      & 
      \begin{minipage}{7cm}
        \[
        \begin{tikzcd}
          a^{\ast}(A,D) \ar[d] \ar[d, dashed, bend right, leftarrow, "a^{\ast}d"'] \ar[r] 
          \arrow["\lrcorner"{anchor=center, pos=0.125}, draw=none, from=1-1, to=2-2]
           & (A,D) \ar[d] \ar[d, bend right, leftarrow, "d"']\\ \Delta \ar[r, "a"] & A
        \end{tikzcd}
        \]
      \end{minipage}
      \\
      \hline
\end{tabular}
\end{table}

Formally, a type theory in the Martin-L{\"o}f tradition \cite{ML:type.theory} (which homotopy type theories are) is a system for deducing new typing and element judgements from others.
We begin by presuming a few basic types, of which the \emph{type universes} $\Type_{\ell}$ are the most central, and then give rules for constructing new types from existing types. 
A type universe $\Type_{\ell}$ is a type whose elements are types. They are defined by the simple rule that if $A : \Type_{\ell}$ then $A$ is a type\footnote{This is known as a ``Russell-style" universe. The somewhat more fiddly but also more exacting ``Tarski-style'' universes have a type $\type{El}(A)$ of elements of $A$ for any $A : \Type_{\ell}$. These details are unimportant for our development here.}; in other words, we have the judgement $A : \Type_{\ell} \yields A \jtype$. By the prescription above, we interpret this judgement as a map $\pi : \Type_{\ell\ast} \to \Type_{\ell}$ in an $\infty$-topos; for $\Type_{\ell}$ to be a type universe means that $\pi$ is an \emph{object classifier} in the sense of \cite[\S\~6.1.6]{Lurie:HTT} --- every map (with suitably small fibers) arises as a pullback of $\pi$ up to a contractible space of choices. We generally assume a tower of universes $\Type_{\ell} : \Type_{\ell + 1}$ (avoiding a paradoxical $\Type : \Type$); however, in this paper we will elide the universe level $\ell$, since it can be readily understood from context. From now on, we will just write $\Type$ to be the type universe. From now on, we will just use $\Gamma \yields A : \Type$ in the same way we have used the typing judgement $\Gamma \yields A \jtype$.

There are, generally speaking, four sorts of rules for constructing and using new types: the \emph{formation} rule, which gives the conditions under which we may construct the type; the \emph{introduction} rule, which gives the conditions for constructing elements of the type; the \emph{elimination} rule, which gives the conditions for using elements of the type to construct elements of other types; and the \emph{computation} rules, which define how the elimination rules compute \emph{definitionally} on introduced elements, and sometimes vice-versa. We use the symbol $\jeq$ to mean \emph{definitional} equality (as opposed to the symbol $=$ we will use for identification). The two most important type constructions are \emph{pair types} and \emph{function types}, which we will present simultaneously:
\begin{itemize}
  \item (Formation) Suppose that $B(x)$ is a type for a free variable $x : A$ (that is, $x : A \yields B(x) : \Type $); then we may form both the type of pairs $\dsum{x : A} B(x)$ and the type of functions $\dprod{x : A} B(x)$. In type theoretic notation:
  \begin{mathpar}
\inferrule{\Gamma,\, x : A \yields B(x) : \Type}{\Gamma \yields \dsum{x : A} B(x) : \Type} \and \inferrule{\Gamma,\, x : A \yields B(x) : \Type}{\Gamma \yields \dprod{x : A} B(x) : \Type}
  \end{mathpar}

  The horizontal bar here is a meta-theoretic implication: if we have made the judgement on the top of the bar, then we may make the judgement on the bottom of the bar.
  \item (Introduction) Given an element $a : A$ and an element $b : B(a)$, we may introduce the pair $(a, b) : \dsum{x : A} B(x)$. On the other hand, if $b(x) : B(x)$ for a free variable $x : A$, then we may introduce the function $x \mapsto b(x) : \dprod{x : A} B(x)$.
   \begin{mathpar}
    \inferrule{\Gamma \yields a : A \and \Gamma \yields b : B(a)}{\Gamma \yields (a, b) : \dsum{x : A} B(x)} \and \inferrule{\Gamma,\, x : A \yields b(x) : B(x)}{\Gamma \yields x \mapsto b(x) : \dprod{x : A} B(x)}
      \end{mathpar}
  \item (Elimination) Given an element $p : \dsum{x : A} B(x)$, we have its first projection $\fst p : A$ and its second projection $\snd p : B(\fst p)$. On the other hand, given $f : \dprod{x: A} B(x)$ and an element $a : A$, we may apply $f$ to $a$ to get $f(a) : B(a)$. 
  \begin{mathpar}
    \inferrule{\Gamma \yields p : \dsum{x : A} B(x)}{\Gamma \yields \fst p : A \and \Gamma \yields \snd p : B(\fst p)} \and \inferrule{\Gamma \yields f : \dprod{x : A} B(x) \and \Gamma \yields a : A}{\Gamma \yields f(a) : B(a)}
      \end{mathpar}
  \item (Computation) Finally, we add definitional equalities (denoted by the symbol $\jeq$) which relate the introduction and elimination rules. In particular, we define $\fst(a, b) \jeq a$, $\snd(a, b) \jeq b$, and $p \jeq (\fst p, \snd p)$ for pairs, and $(x \mapsto b(x))(a) \jeq b(a)$ and $f \jeq (x \mapsto f(x))$ for functions.
\end{itemize}

Semantically, if $\Gamma \yields A : \Type$ is interpreted as $\pi_A : (\Gamma,A) \to \Gamma$, then the pair types and function types are interpreted as the left and right adjoints respectively to the pullback functor $\pi_A^* : \Ea_{/\Gamma} \to \Ea_{/(\Gamma,A)}$.

\begin{notation}
  We will use the Agda-inspired notation for dependent pair types (also known as
  dependent sum types) and dependent function types (also known as dependent
  product types). Sum and product notation are also common in the literature:
  \begin{align*}
    \dsum{a : A} B(a) &\equiv \sum_{a : A} B(a) \\
    \dprod{a : A} B(a) &\equiv \prod_{a : A} B(a).
  \end{align*}

  \end{notation}

There is one remaining type constructor which we need: the type $x =_A y$ of ways to identify $x$ with $y$ as elements of type $A$. Different formalizations of homotopy type theory approach the definition of identity types differently; in this paper, we will use Martin L{\"o}f's identity types together with Voevodsky's \emph{univalence axiom}, following \cite{HoTTBook}.

\begin{itemize}
  \item (Formation) If $a : A$ and $b : A$ are elements of the same type, then we may form the type $a =_{A} b$ of identifications of $a$ with $b$ as elements of type $A$. In type theoretic notation:
  \begin{mathpar}
    \inferrule{\Gamma \yields a : A \and \Gamma \yields b : A}{\Gamma \yields a =_{A} b : \Type}
  \end{mathpar}
  \item (Introduction) If $a : A$, then we may reflexively identify $a$ with $a$ via an element $\refl_a : a =_A a$. 
  \begin{mathpar}
    \inferrule{\Gamma \yields a : A}{\Gamma \yields \refl_a : a =_A a }
  \end{mathpar}
  When $A$ is a type of structured sets, we may think of $\refl_a$ as the identity isomorphism on $a : A$.
  \item (Elimination) The elimination rule for identity types is known as \emph{(based) path induction}. Given $a : A$, to construct an element $c(x, p) : C(x, p)$ of any type $C(x, p)$ depending on variables $x : A$ and $p : a =_A x$, it suffices to assume that $x$ is $a$ and $p$ is $\refl_a$ and to construct an element of type $C(a, \refl_a)$. In type theoretic notation:
  \begin{mathpar}
    \inferrule{\Gamma \yields a : A \and \Gamma, x : A, p : a =_A x \yields C(x, p) : \Type \and \Gamma \yields c : C(a, \refl_a)}{\Gamma, x : A, p : a =_A x \yields \ind{\refl_a}{\,p}{c} : C(x, p)}
  \end{mathpar}
  \item (Computation) The computation rule for identity types says that $(\ind{\refl_a}{\,\refl_a}{c}) \jeq c$; when taking $p : a =_A x$ to be $\refl_A : a =_A a$, the elimination rule does not produce anything new.
\end{itemize}

Since we have $x : A,\, y : A \yields x =_A y : \Type$, we interpret the identity type as a fibration $\pi_{\Id_A} : \Id_A \to A \times A$. Per the formulation of \cref{table:type.interpret}, we then interpret $x : A \yields \refl_x : x =_A x$ as a section of the pullback of $\pi_{\Id_A}$ along the diagonal $\Delta : A \to A \times A$:
\[
\begin{tikzcd}
	{\Delta^{\ast}\Id_{A}} & {\Id_A} \\
	A & {A \times A}
	\arrow[from=1-1, to=1-2]
	\arrow[from=1-1, to=2-1]
	\arrow["\lrcorner"{anchor=center, pos=0.125}, draw=none, from=1-1, to=2-2]
	\arrow["{\pi_{\Id_A}}", from=1-2, to=2-2]
	\arrow["\refl", bend left, dashed, from=2-1, to=1-1]
	\arrow["\Delta"', from=2-1, to=2-2]
\end{tikzcd}
\]
This is equivalent to giving a map $\refl : A \to \Id_A$ so that $\pi_{\Id_A} \circ \refl = \Delta$; the elimination rule for identity types lets us see this factorization of the diagonal as giving $\Id_A$ the structure of a \emph{(very good) path object} for $A$. Specifically, the elimination rule guarentees a lift in the following square:
\[
\begin{tikzcd}
	A & C \\
	{\Id_A} & {\Id_A}
	\arrow["c", from=1-1, to=1-2]
	\arrow["\refl"', from=1-1, to=2-1]
	\arrow[from=1-2, to=2-2]
	\arrow[dashed, from=2-1, to=1-2]
	\arrow[equals, from=2-1, to=2-2]
\end{tikzcd}
\]
which we may think of as witnessing the fact that $\refl : A \to \Id_A$ is an acyclic cofibration (see \cite{awodey-warren:identity.types} for a model theoretic account).

Remarkably, the above rules are enough to equip every type $A$ with the structure of an $\infty$-groupoid, with hom $\infty$-groupoids given by the identity types $a =_A b$ \cite{awodey-warren:identity.types}. As an example, we may define the composition of identifications (generalizing the transitivity of equality) to be the following function:
\[
  p \mapsto q \mapsto \ind{\refl_b}{\,q}{p} : (a = b) \to ((b = c) \to (a = c)).
\]
That is, to define the function $q \mapsto p \bullet q : (b = c) \to (a = c)$ given by precomposing with identification $p : a = b$, it suffices by path induction to assume that $q : b = c$ is of the form $\refl_b : b = b$, in which case $p \bullet \refl_b$ is defined to be $p$. 

While this notion of identity type is sufficient to endow every type with the structure of an $\infty$-groupoid, it is not on its own strong enough to prove the naive desiderata of \ref{identification.desiderata}. For this, we need a further axiom: the \emph{univalence axiom} which says that identifications $A =_{\Type} B$ between types in a type universe are equivalent to \emph{equivalences} $A \simeq B$. Let's turn to understanding equivalence of types now.

\begin{rmk}
  There is also a family of type definitions known as \emph{higher inductive types} \cite{LUMSDAINE_2019}. These include the natural numbers (which is just an ``inductive type", and inspires the name ``inductive''), as well as pushouts and colimits of sequences. In general, higher inductive types are types freely generated by some elements and identifications between those elements (and perhaps identifications between those identifications, and so on). We will make use of some constructions made possible by higher inductive types, such as \emph{localization}, but will have no need for the general concept. 
\end{rmk}

\subsection{Uniqueness, equivalence, and univalence}

In this section, we will describe Voevodsky's famous \emph{univalence axiom} which ensures that the elements of the Martin-L{\"o}f identity types described above behave as \emph{identifications} in the naive sense described in \ref{identification.desiderata}. The content of this section appears already in Voevodsky's \emph{A very short note on the homotopy $\lambda$-calculus} \cite{voevodsky:short.note}.

Voevodsky considered his most important and novel contribution in foundational matters to be the definition of \emph{uniqueness} and the more general \emph{$h$-levels} which depend on it (see the upcoming \cref{sec:hlevel})\footnote{``The most important new concept of the Univalent Foundations is the concept of h-level.'' \cite{voevodsky:foundations.part.3}.}.
To even express the univalence axiom requires this notion of uniqueness.
The definition of uniqueness in homotopy type theory is straightforward: a type $A$ has a unique element if we have an element $a : A$, and for any other element $x : A$ we have an identification $c(x) : a =_A x$ of $a$ with $x$. This definition can itself be expressed as a type.

\begin{defn}\label{defn:unique.element}
Let $A$ be a type. Define the type $\exists! A$ by
\[
  \exists! A \jeq \dsum{a : A} \big(\dprod{x : A} (a =_A x)\big).
\]
We say that $A$ has a unique element if we have an element $(a, c) : \exists! A$. 
\end{defn}

\begin{warning}
It is far more common in the homotopy type theory literature to say that the type $A$ is \emph{contractible} (with center of contraction $a$ and contraction $c$) when $(a, c) : \exists! A$, and to write $\type{isContr}(A)$ for $\exists! A$. This is because when interpreted in the Kan complex model, a type $A$ has a unique element in the sense of \cref{defn:unique.element} just when its interpretation is contractible (equivalent to the point) as a Kan complex. 

However, we will avoid this terminology here because when working with smooth stacks --- which fully faithfully include not only homotopy types (as constant stacks) but also manifolds (as $0$-truncated stacks) --- there is a competing notion which we want to call ``contractible'': being weakly homotopy equivalent to the point as measured by smooth paths from the manifold $\Rb$. In particular, the manifold $\Rb$, considered as a smooth stack, is contractible in the usual sense, but since identification in $\Rb$ is simply equality of real numbers, it does not have a unique element in the sense of \cref{defn:unique.element}; after all, $0$ does not equal $1$, though they are connected by a path (unique up to smooth deformation). More on this distinction in \cref{sec:intro.modal}. 
\end{warning}

  \cref{defn:unique.element} is our first example of encoding a \emph{proposition} as a type in its own right; we'll see more on propositions-as-types in \cref{sec:hlevel}. 

What is perhaps surprising about the definition of uniqueness in HoTT is that it does not mean `unique up to \emph{some} identification'\footnote{For this notion in HoTT, see \cref{sec:hlevel}.}; rather it means `unique up to \emph{unique} identification'. That is, given $(a, c) : \exists! A$, we may also construct an element of $\exists!(a =_A x)$ for any $x : A$. 

With this definition of unique existence, we can now define an \emph{equivalence} between types in a manner remarkably reminiscent of a bijection of sets: a function $f : A \to B$ is an equivalence when there is a unique inverse image of any element in $B$. We need to give the appropriate definition of ``inverse image'' in homotopy type theory first.

\begin{defn}\label{defn:fiber}
 Let $f : A \to B$ be a function. The \emph{fiber} of $f$ over $b : B$ is the type of pairs of an $a : A$ and an identification $p : b =_B f(a)$.
 \[
  \fib_f(b) \jeq \dsum{a : A}(b =_B f(a)).
 \]
\end{defn}

\begin{rmk}
  Where we avoid the term ``contractible'' for uniqueness because it could have conflicting interpretations in the $\infty$-topos of smooth stacks, here we use the term fiber both in the sense of ``homotopy fiber'' and in the sense of ``fiber bundle''. The fiber (\cref{defn:fiber}) of a map between homotopy types (considered as constant smooth stacks) is its homotopy fiber, while the fiber of a map between manifolds (for example, a fiber bundle) is its fiber or inverse image in the geometric sense.
\end{rmk}

\begin{defn}\label{defn:equivalence}
  Let $f : A \to B$ be a function. Then $f$ is an \emph{equivalence} if for all $b : B$, the fiber $\fib_f(b)$ has a unique element. We can express the proposition that $f$ is an equivalence as a type itself:
  \[
    \type{isEquiv}(f) \jeq \dprod{b : B} \exists!\fib_f(b).
  \]
  We may therefore define the type of equivalences $A \simeq B$ to be the type of pairs of a function $f : A \to B$ and a \emph{witness} $w : \type{isEquiv}(f)$ that $f$ is an equivalence:
  \[
    (A \simeq B) \jeq \dsum{f : A \to B} \type{isEquiv}(f).
  \]
\end{defn}

This our second example of encoding a proposition as a type. Note that to carve out a subtype described by a proposition --- e.g. the type of equivalences from the type of functions --- we take the type of pairs of an element together with a witness that the proposition is true of that element. Having this witness around as a bona-fide element lets us use the truth of the proposition quite concretely; for example, if $(f, w) : A \simeq B$, then the inverse to $f$ encoded by the witness $w$ is the function $b \mapsto \fst \fst w(b)$ which selects out the (unique) inverse image to $b$ guarenteed by $w$.

In order to express the univalence axiom --- still the goal of this section --- we need to know that identity functions $(a \mapsto a) : A \to A$ are equivalences. That is, we need to give a function $w : \dprod{a : A} \exists! \fib_{\id}(a)$; in other words, we need to show that the singleton types $\fib_{\id}(a) \jeq \dsum{x : A} (a =_A x)$ have a unique element. This is essentially the content of \emph{path induction} --- when defining a function out of $\dsum{x: A}(a =_A x)$, it suffices to define it on the element $(a, \refl_a)$. In particular, we may prove that $\id_A : A \to A$ is an equivalence with the witness 
\[
  w(a) \jeq \big((a, \refl_a), (x, p) \mapsto \ind{\refl_a}{\,p}{\refl_{(a, \refl_a)}}\big).
\]

With this in hand, we may then express the univalence axiom.
\begin{axiom}[Univalence Axiom]\label{axiom:univalence}
  For any types $A, B : \Type$, the function
  \[
    p \mapsto \ind{\refl_A}{\,p}{(\id_A, w)} : (A =_{\Type} B) \to (A \simeq B)
  \]
  is an equivalence. That is, we have an element $\term{ua}_{A,B} : \type{isEquiv}(p \mapsto \ind{\refl_A}{\,p}{(\id_A, w)})$.
\end{axiom}
We may think of the univalence axiom as an axiom of \emph{type extensionality}: two types may be identified by a one-to-one correspondence between their elements. This axiom has a number of important corollaries which together imply that the types of identifications $x =_A y$ of mathematical structures $x$ and $y$ are given by \emph{structure preserving equivalences} --- a thesis known in the HoTT literature as the \emph{structure identity principle} \cite{aczel:on.voevodskys.univalence.axiom,ahrens-north:univalent.foundations}. 

\subsection{Propositions, sets, and truncation} \label{sec:hlevel}

We have so far seen two examples of propositions as types: the proposition $\exists! A$ that there is a unique element of $A$ and the proposition $\type{isEquiv}(f)$ that a function $f$ is an equivalence. To prove these propositions was to give an element of the corresponding type. Unlike set theory, which takes place within the framework of first order logic, by considering propositions as types we can reason about types within type theory itself --- type theory is its own logic. 

Indeed, the type constructions of pairs and functions specialize to propositions when encoded as types. If $P$ and $Q$ are propositions considered as types, then to prove the proposition ``$P$ and $Q$'' we must prove both $P$ and prove $Q$; for this reason, we may represent the proposition ``$P$ and $Q$'' by the pair type $P \times Q$, since an element $(p, q) : P \times Q$ is just a pair of an element $p : P$ (proving $P$) and $q : Q$ (proving $Q$). Similarly, to prove ``$P$ implies $Q$'', we must prove $Q$ under the hypothesis that $P$ is true; for this reason, we may represent the proposition ``$P$ implies $Q$'' by the function type $P \to Q$. If $P(x)$ is a family of propositions depending on a variable $x : X$, then the proposition $\forall x: X.\, P(x)$ may be represented by the dependent function type $\dprod{x : X}P(x)$; to give an element of $\dprod{x : X}P(x)$ is to give for any $x : X$ an element of $P(x)$, proving the corresponding proposition. The terminal type $\top$ with a unique element plays the role of true, and the empty type $\emptyset$ plays the role of false\footnote{These are both definable as inductive types.}. We define the negation $\neg P \jeq P \to \emptyset$ to be the proposition that $P$ implies false.

We might expect that the existential quantifier $\exists x: X.\, P(x)$ may dually be given by the pair type $\dsum{x : X}P(x)$; after all, to give an element of $\dsum{x : X}P(x)$ is to give a pair $(x, p)$ where $x : X$ and $p : P(x)$, giving a specific witness to the existence of an element of $X$ satisfying $P$. But we just saw how the pair type $\dsum{x : X}P(x)$ can be used to carve out a subtype with the type of equivalences $(A \simeq B) \jeq \dsum{ f : A \to B} \type{isEquiv}(f)$. While giving an element of $(A \simeq B)$ does prove that $A$ and $B$ are equivalent, it does more of that: it fixes a particular equivalence. The difference between proving the existence of an element of a type and actually giving an element of that type becomes much more stark when we work in the context of some free variables. For example, consider the statement that a manifold $M$ is $n$-dimensional: for every point $p$, there exists an $n$-element basis of the tangent space $T_p M$. It is not generally the case that we can give an element $f(p) : \type{Frame}(T_pM)$ of the set of bases of $T_p M$ depending on the free variable $p: M$, since this would be interpreted in the $\infty$-topos of smooth stacks as a smooth section of the frame bundle, which would trivialize the tangent bundle. In HoTT, a proof of existence without a given witness is called \emph{mere} existence.

Not all types are propositions --- some types are sets, like $\Nb$ and $\Rb$ or the set of bases of a vector space, and some types are groupoids, like the type of real vector spaces itself. What determines when a type is to be considered a proposition, a set, or a groupoid? To prove a proposition considered as a type $P$, we give an element $p : P$; if the type $P$ has at most one element, then to give an element is only to prove the proposition and nothing more. We therefore define the propositions to be those types $P$ which have at most one element. 

\begin{rmk}
  It is sometimes said of homotopy type theory that the elements of a proposition $P$ considered as a type are its proofs. But if we consider the elements of $P$ up to identification, then the very definition of a proposition ensures that any two of them are uniquely identified. This appears to contradicts the intuition that the elements of $P$ are its proofs, since there may be many different ways to prove a proposition. The idea that the elements of $P$ are its proofs is only true if we consider those elements up to judgemental or definitional equality; there may be many different ways to define an element of a type, even if many differently defined elements may then be identified. In type theory, the mathematically relevant notion of sameness is identification and not definitional equality; for this reason, it's better to think of an element $p : P$ as \emph{the fact that $P$ is true}, as in the phrase ``by the fact that $f$ is an equivalence, we may transport the relevant structure across it''. In this phrase, we are implicitly using the element $p : \type{isEquiv}(f)$ --- the fact that $f$ is an equivalence --- to perform a construction.
\end{rmk}

We may compare two elements of a set for equality, but there is no extra data involved in an equality --- if two elements are equal, then they are \emph{just equal}. Identification in a set should therefore be a \emph{proposition}. Similarly, groupoids are very often composed of sets equipped with some structure; identifying two elements of such a groupoid involves giving an isomorphism between these structured sets. Therefore, there will be a set of identifications in a groupoid. This gives us a general inductive definition of the \emph{$h$-level} or \emph{truncation level} of a type.

\begin{defn}\label{defn:truncation}
  We define what it means for a type $A$ to be $n$\emph{-truncated} by induction on $n$, starting at $n = -2$:
  \begin{itemize}
    \item A type $A$ is $(-2)$-truncated when it has a unique element.
    \item A type $A$ is $(n + 1)$-truncated when for any two elements $x, y : A$, the type $x =_A y$ of identifications is $n$-truncated.
    \[
      \type{isTruncated}_{n+1}(A) \jeq \dprod{x,\, y : A} \type{isTruncated}_n(x =_A y)
    \]
  \end{itemize}
  We say that $A$ is a \emph{proposition} when it is $-1$-truncated, a \emph{set} when it is $0$-truncated, and a \emph{groupoid} when it is $1$-truncated.
  \begin{align*}
    \type{isProp}(A) &\jeq \dprod{x,y : A} \exists!(x =_A y) \\
    \type{isSet}(A) &\jeq \dprod{x,y : A} \type{isProp}(x =_A y)
  \end{align*}
\end{defn}

Types like $\exists! A$ and $\type{isEquiv}(f)$ are propositions according to the above definition. Furthermore, if $P$ and $Q$ are propositions, then $P \times Q$ and $P \to Q$ are propositions, justifying our use of these types to represent conjunction and implication. Similarly, $\dprod{x : X} P(x)$ is a proposition whenever $P(x)$ is a proposition for all $x : X$. The type of natural numbers $\Nb$ is provably a set, and the type of vector spaces is provably a groupoid.

For any type $A$, we have a $n$-truncated type $\trunc{A}_n$ called the
$n$-truncation of $A$, which is equipped with a map $|\cdot|_n : A \to
\trunc{A}_n$ that is the initial map from $A$ to an $n$-truncated
type\footnote{The $n$-truncation may be constructed by \emph{localization}; see
  Section 2 of \cite{RSS:Modalities}.}. The truncation operation $A \mapsto \trunc{A}_n$ is
an example of a \emph{modality} \cite{RSS:Modalities}; it may be equivalently
described by the stable orthogonal factorization system which factors any
function as an $n$-connected function (one for which $\forall b : B.\exists!
\trunc{\fib_f(b)}_n$) and an $n$-truncated function (one with $n$-truncated
fibers). This is the guise in which modalities appear in higher category theory; see \cite{ABFJ:blakers.massey}.

The propositional truncation $\exists A \jeq \trunc{A}_{-1}$ gives a definition of the existential quantifier: to give an element of $\exists A$ is to prove that an element of $A$ exists without giving a specific example. The universal property of $|\cdot|_{-1} : A \to \exists A$ as the initial map into a proposition
\[
\begin{tikzcd}
	A & P \\
	{\exists A}
	\arrow["\forall", from=1-1, to=1-2]
	\arrow["{|\cdot|_{-1}}"', from=1-1, to=2-1]
	\arrow["{\exists!}"', dashed, from=2-1, to=1-2]
\end{tikzcd}
\]
expresses the following familiar proof principle: to prove a proposition $P$ assuming that there exists an element of $A$, we may prove $P$ under the assumption that we have a free variable $x : A$. We may then define the usual existential quantifier $\exists x : X,\, P(x)$ to be the propositional truncation $\exists\big(\dsum{x : X} P(x)\big)$ of the type of pairs --- proving $\exists x : X,\, P(x)$ is the same as showing that there exists a pair $(x, p)$ with $x : X$ and $p : P(x)$ proving $P$ of $x$.

With the definition of sets and the existential quantifier, we can now define the type $\B U(1)$ of Hermitian lines --- that is, one dimensional Hermitian vector spaces --- using pairs, functions, identity types, and the propositional truncation as above, as well as assuming the existence of the smooth real line $\Rb$ from which we may define $\Cb \jeq \Rb \oplus i \Rb$ as in \cref{assumption:SDG}.

\begin{equation}\label{defn:BU}
  \B U(1) \jeq \left\{\begin{aligned}
    (L : \Type) &\times \type{isSet}(L) \\
    &\times (+ : L \times L \to L) \times (0 : L) \\
    &\times (\cdot : \Cb \times L \to L) \\
    &\times (\term{assoc} : \dprod{a,b,c : L} (a + b) + c = a + (b + c)) \\
    &\times (\term{zero} : \dprod{a : L} a + 0 = a) \\
    &\times (\term{comm} : \dprod{a, b : L} a + b = b + a) \\
    &\times (\term{unit} : \dprod{a : L} 1 \cdot a = a) \\
    &\times (\term{scalarassoc} : \dprod{u,v :\Cb}\dprod{a : L} u \cdot (v \cdot a) = (uv) \cdot a) \\
    &\times (\term{scalarzero} : \dprod{a : L} 0 \cdot a = 0)\\
    &\times (\term{dim} : \exists (\dsum{b : L} \neg(b = 0) \times (\dprod{a :  L} \dsum{c : \Cb} (a = c \cdot b)))) \\
    &\times (\langle -, - \rangle : L \times L \to \Cb) \\
    &\times (\term{leftlinear} : \dprod{c, d : \Cb}\dprod{a, b, z : L} (\langle ca + db, z \rangle = c\langle a, z\rangle + d \langle b, z \rangle ))\\
    &\times (\term{sesqui} : \dprod{a, b : L} \langle a, b \rangle = \overline{\langle b, a \rangle})\\
    &\times (\term{positive} : \dprod{a : L} \langle a, a\rangle \geq 0)\\
    &\times (\term{definite} : \dprod{a : L} (\dprod{b : L} \langle a, b\rangle = 0) \to (a = 0))
  \end{aligned}\right.
\end{equation}
The type $(L_1, \cdots) =_{\B U(1)} (L_2, \cdot)$ of identifications between two Hermitian lines is equivalent to the type of unitary isomorphisms betwen them. In particular, $(\Cb, \cdots) =_{\B U(1)} (\Cb, \cdot)$ is equivalent to $U(1)$, justifying the name $\B U(1)$ for the type of Hermitian lines --- it deloops the group $U(1)$. 

When interpreted in the $\infty$-topos of smooth stacks, this definition of $\B U(1)$ becomes the moduli stack of Hermitian line bundles described in \cref{sec:convenient}. In general, to construct the moduli stack of $X$-bundles in homotopy type theory, we just write down the definition of an $X$. This is one way that homotopy type theory makes working with stacks remarkably concrete.

\subsection{The homotopy type of a stack and crisp type theory}
\label{sec:intro.modal}

You might now be wondering: where is the ``homotopy'' in homotopy type theory?
Identifications between mathematical objects are, generally speaking, the
morphisms of higher groupoids. Grothendieck's \emph{homotopy hypothesis} posits
that higher groupoids are equivalently the homotopy types of topological spaces.
In other words, we can model identifications between mathematical objects as
homotopies (in the geometric realization of the nerve of the higher groupoid of
such objects). This is the sense in which ``homotopy'' is used in
the term ``homotopy type theory'', rather than in the literal sense of continuous
deformation between continuous functions. More formally, homotopy type theory admits a model in the Kan model structure on simplicial sets which interprets every type as a Kan complex. We emphasize that despite modelling homotopy theory, Kan complexes are \emph{discrete} combinatorial objects.

 But, since homotopy type theory can be
interpreted in higher toposes of stacks on general sites, we can also talk
about continuous deformation of maps so long as our site consists of suitably
continuous objects. In particular, in a topos of smooth stacks, we have the usual notion of homotopy $h : \Rb \times X \to Y$ between two maps $f = h(0, -)$ and $g = h(1, -)$, where $\Rb$ is the smooth manifold of reals considered as a smooth stack representably. 

These two interpretations of ``homotopy'' are \emph{orthogonal} in the $\infty$-topos of smooth stacks in a precise sense. In particular, if $X$ is a (geometrically discrete) homotopy type considered as a constant stack, then any map $p : \Rb \to X$ is constant. Even more, a stack $X$ is constant if and only if the inclusion of constants $X \to X^{\Rb}$ is an equivalence of stacks. On the other hand, $\Rb$ is a \emph{set} ($0$-truncated) as in \cref{defn:truncation} --- it is a sheaf valued in sets with no higher homotopies. In order to internalize this opposition between discrete and continuous objects in type theory, we will use Shulman's \emph{real cohesive homotopy type theory} \cite{Shulman:Real.Cohesion}. 

In type theory (and constructive mathematics generally), the proposition that all functions $\Rb \to \Rb$
are continuous is undecided --- there are models of HoTT in which every function $\Rb \to \Rb$ is continuous (and,
of course, familiar models where there are discontinuous functions $\Rb \to \Rb$). In smooth stacks, every function $\Rb \to \Rb$ is even smooth as a consequence of the Yoneda lemma.
Since in type theory such a function $f : \Rb \to \Rb$ is defined by giving
the image $f(x)$ of a free variable $x : \Rb$, we see that in a pure
constructive setting, the dependence of terms on their free variables confers a
liminal sort of continuity. This is a very powerful observation which, in its
various guises, lets us avoid the menial checking of continuity, smoothness,
regularity, and so on for functions in various models of
homotopy type theory. It extends far beyond real valued functions; for example,
the assignment of a vector space $V_p$ to a point $p$ in a manifold $M$,
constructively, gives a vector bundle $\dsum{p : M} V_p \to M$ over $M$ with all
its requirements of continuity or smoothness (depending on the model)\footnote{For local triviality, one must ask for a \emph{free} vector space --- ones for which there merely exists a basis. Without the axiom of choice, one cannot prove that every vector space admits a basis.}. Since
all sorts of continuity (open-set continuity, smoothness, regularity, analyticity) can be
captured in various models, Lawvere named the general notion ``cohesion'' in his
paper \cite{Lawvere:Axiomatic.Cohesion}, whose generalization to
$\infty$-categories in \cite{Schreiber:Differential.Cohomology} inspired the type
theory of \cite{Shulman:Real.Cohesion}.

Using a free variable in type theory confers a liminal sort of continuity in that variable, but not every dependence on a variable \emph{should} be continuous (smooth, analytic, etc.).  To allow for
discontinuous dependencies, then, we must mark our free variables as varying
cohesively or not. For this reason, Shulman introduces \emph{crisp}
variables, which are free variables in which elements depend discontinuously:
\[
a :: A.
\]
For emphasis, we will say that a variable
which is not crisp is \emph{cohesive}.
Any variable appearing in the type of a crisp variable must also be crisp. To
ensure this, we separate our context --- our list of typed free variables --- into two
chambers: one for the crisp variables, and one for the cohesive variables:
\[
  x :: X \,\mid\, y : Y(x),\, z : Z(x, y) \yields a : A
\]

By ensuring that the crisp variables appear before the cohesive variables, we
enforce the rule that the type of a crisp variable must only involve crisp
variables, while the type of a cohesive variable may involve both crisp and
cohesive variables.

When all the variables in an expression are crisp, we say that
that expression is crisp. A crisp expression has a context with an empty cohesive
chamber:
\[
\Delta \,\mid\, \cdot \yields a : A
\]

Constants appearing in
an empty context --- expressions with no free variables like $0 : \Nb$ or $\Nb : \Type$ ---  are therefore always crisp.\footnote{Note that as these are
  expressions and not free variables, we don't need to use the special syntax $a ::
  A$. The double colon introduces a crisp \emph{free} variable.} This means that
one cannot give a closed form example of a term which is \emph{not} crisp; all
terms with no free variables are crisp. 

We may only substitute crisp expressions in for
crisp variables. If we could substitute cohesive variables in for crisp variables,
then any discontinuous dependence would be continuous.
 The full rules for crisp type theory can be found in Section 2 of \cite{Shulman:Real.Cohesion}.

Given this notion of discontinuous dependence of terms on their crisp free variables,
we can now describe an operation on types which removes the cohesion amongst their
points. Given a \emph{crisp} type $X$, we have a type $\flat X$ whose points
are, in a sense, the \emph{crisp} points of $X$. Since it is free variables that
may be crisp, we express this idea by allowing ourselves to assume that a
(cohesive) variable $x : \flat X$ is of the form $u^{\flat}$ for a \emph{crisp}
$u :: X$. Here are the rules for $\flat$.

\begin{itemize}
  \item (Formation) If $X$ is a crisp type (depending only on crisp variables),
    then we may form $\flat X$:

    \begin{mathpar}
      \inferrule{\Delta \,\mid\, \cdot \yields X : \Type}{\Delta \,\mid\, \Gamma \yields
        \flat X : \Type}
    \end{mathpar}
  \item (Introduction) If $x$ is a crisp element of $X$, then we may introduce
    $x^{\flat} : \flat X$:

    \begin{mathpar}
      \inferrule{\Delta \,\mid\, \cdot \yields x : X}{\Delta \,\mid\, \Gamma \yields
      x^{\flat} : \flat X}
     \end{mathpar}

   \item (Elimination) When using a free variable $x : \flat X$ to construct an
     element, we may assume that $x$ is of the form $u^{\flat}$ for a
     \emph{crisp} variable $u :: X$. More formally, whenever we have type family $C : \flat X \to \Type$, an $x : \flat X$, and an element $f(u) : C(u^{\flat})$ depending on a \emph{crisp} $u :: X$, we get an element
     $$( \mbox{let $u^{\flat} := x$ in $f(u)$} ) : C(x).$$
     We refer to this method of proof as ``$\flat$-induction''. In type theoretic notation:

       \begin{mathpar}
\inferrule{\Delta \,\mid\, \Gamma, x : \flat X \yields C(x) : \Type \and \Delta, u :: A
  \,\mid\, \Gamma \yields f(u) : C(u^{\flat})}{\Delta \,\mid\, x : \flat X \yields \mbox{let $u^{\flat} := x$ in $f(u)$} : C(x)}
       \end{mathpar}

   \item (Computation) If $x \equiv v^{\flat}$, then $(\mbox{let $u^{\flat} := x$ in $f(u)$})
     \equiv f(v)$.
\end{itemize}

To interpret crisp type theory into an $\infty$-topos $\Ea$, we must now remember the global sections geometric morphism $\gamma : \Ea \to \Sa$ to the $\infty$-topos of homotopy types. When $\Ea$ is $\infty$-connected --- that is, when $\gamma^{\ast}$ is fully faithful --- we may interpret $\flat X$ as the constant stack $\gamma^{\ast} \gamma_{\ast} X$ at the global sections of $X$. In particular, if $\Ea$ is the $\infty$-topos of smooth stacks and $M$ is a manifold, then $\flat M$ will be the constant sheaf at the (discrete) set of points of $M$. If $\B U(1)$ is the classifying stack of Hermitian line bundles, then $\flat \B U(1)$ will be the constant stack at the (nerve of the) groupoid of Hermitian lines. 

We have an inclusion $(-)_{\flat} : \flat X \to X$ given by $x_{\flat} :\equiv
(\mbox{let $u^{\flat} := x$ in $u$})$. Since we are thinking of a dependence on a
crisp variable as a \emph{discontinuous} dependence, if this map $(-)_{\flat} :
\flat X \to X$
is an equivalence then every \emph{discontinuous} dependence on $x :: X$
underlies a \emph{continuous} dependence on $x$ --- that is, any discontinous
function on $X$ is already continuous. This leads us to the following definition:
\begin{defn}
  A crisp type $X :: \Type$ is \emph{crisply discrete} if the counit
  $(-)_{\flat} : \flat X \to X$ is an equivalence.\footnote{See Remark 6.13 of
    \cite{Shulman:Real.Cohesion} for a discussion on some of the subtleties in
    the notion of crisp discreteness.}
\end{defn}

Note that this definition is only sensible for crisp types, since we may only
form $\flat X$ for crisp $X :: \Type$. We would also like a notion of
discreteness which applies to any type. In the $\infty$-topos of smooth stacks,
we can give such a definition: a type $X$ is discrete if the inclusion of
constant functions $X \to (\Rb \to X)$ is an equivalence. Such types are said to
be $\Rb$\emph{-local}, and by the theory of localization in HoTT (see Section 2
of \cite{RSS:Modalities}), we can always define a modality $X \mapsto \shape X$
where $\shape X$ is $\Rb$-local and there is a unit $(-)^{\shape} : X \to \shape
X$ which is initial amongst maps into $\Rb$-local types.

Localization by $\Rb$ is known as the \emph{shape} modality; it sends every
manifold to the constant stack at the homotopy type of that manifold (see
Proposition 4.3.29 of \cite{Schreiber:Differential.Cohomology} for a proof of
this fact). This generalizes to higher types as well; 
for example $\shape \B U(1) \simeq \B^2 \Zb$, where the former is the module stack of Hermitian line bundles and the latter is the constant stack at a double delooping of the integers.

However, for the abstract portion of this paper (\cref{sec:modal.fracture}), we will only need the existence of the
modality $\shape$ and not that it is given by localization at $\Rb$. We will
nevertheless refer to the $\shape$-modal types as discrete. To make sure that
the two notions of discreteness coincide, we assume the following axiom:

\begin{axiom}[Unity of Opposites]\label{ax:unity_of_opposites}
For any \emph{crisp} type $X$, the counit $(-)^{\flat} : \flat X \to X$ is an
equivalence if and only if the unit $(-)^{\shape} : X \to \shape X$ is an equivalence.
\end{axiom}

This axiom implies that $\shape$ is left adjoint to $\flat$, at least for crisp
maps. In \cite{Shulman:Real.Cohesion}, Shulman assumes an axiom C0 which lets him
define the $\shape$ modality as a localization and prove our Unity of Opposites
axiom. The two axioms have roughly the same strength, though C0 is slightly
stronger since it assumes that $\shape$ is an \emph{accessible} modality (see \cite[\S~2.3]{RSS:Modalities}).

\begin{thm}\label{thm:adjointness_of_opposites}
  Let $X$ and $Y$ be crisp types. Then
  $$\flat(X \to \flat Y) = \flat(\shape X \to Y).$$
\end{thm}
\begin{proof}
This is Theorem 9.15 of \cite{Shulman:Real.Cohesion}. Note that Axiom C0 is only
used via our Unity of Opposites axiom.
\end{proof}

Semantically, \cref{ax:unity_of_opposites} is the assumption that $\Ea$ is \emph{locally $\infty$-connected}: that is, the inclusion $\gamma^{\ast} : \Sa \to \Ea$ of constant stacks has a left adjoint $\Pi_{\infty} : \Ea \to \Sa$. In the case of smooth stacks, $\Pi_{\infty}$ sends a manifold to its homotopy type --- its fundamental $\infty$-groupoid. We then interpret $\shape X$ as $\gamma^{\ast} \Pi_{\infty} X$, the constant stack at the fundamental $\infty$-groupoid of $X$.

We may also define the \emph{truncated shape modalities} $\shape_n$ to have as
modal types those types which are both $n$-truncated and $\shape$-modal. In particular, $\shape_0 X$ is the set of connected components of $X$, and $\shape_1 X$ is the fundamental groupoid of $X$. It is not
known whether $\shape_n X = \trunc{\shape X}_n$ for general (cohesively varying) $X$, but it is true
for \emph{crisp} $X$ (see Proposition 4.5 of \cite{Jaz:Good.Fibrations}).

\begin{rmk}\label{rmk:flat_name}
Theorem \ref{thm:adjointness_of_opposites} justifies the use of the symbol
``$\flat$'' in flat type theory. If we think of $\shape X$ as the homotopy type
of $X$, then the adjointness of $\shape$ with $\flat$ tells us that $\flat \B G$
modulates principal $G$-bundles with a \emph{homotopy invariant parallel
  transport} --- that is, bundles with \emph{flat} connection. This terminology
is due to Schreiber in \cite{Schreiber:Differential.Cohomology}. 
In particular, $\flat \B U(1)$ represents Hermitian line bundles with flat connection; adjointness states that any such bundle on a manifold $M$ is determined by its homotopy invariant parallel transport, understood as a representation of the fundamental groupoid $\shape_1 M$
\end{rmk}

\begin{rmk}
Since this paper was first drafted, there have been many developments in modal homotopy type theory. In particular, there are now modal type theories with semantics given by certain structured 2-functors from a 2-category of modes into the 2-category of categories \cite{Shulman:MATT}. As described in that paper, Shulman's multimodal adjoint type theory subsumes Shulman's crisp type theory used in this paper, with some inessential changes to syntax. However, the semantics of univalent universes in modal type theory remains, to the author's knowledge, not fully settled. 
\end{rmk}

\begin{rmk}
The proof assistant Agda 2.6 \cite{Agda:cohesive} has a compiler flag, ``\begin{verbatim}
  --cohesion
\end{verbatim}'' which enables the use of the flat modality. It can be very helpful to use Agda to get comfortable with crisp variables. While this paper has not been formalized in Agda, there should be no obstruction to doing so.
\end{rmk}

With this background in order, we can now proceed to the novel content of this paper. 

\section{The Modal Fracture Hexagon}\label{sec:modal.fracture}

In this section, we will construct the modal fracture hexagon of a higher group.

A higher group $G$ is a type equipped with a $0$-connected delooping $\B G$. An
ordinary group $G$ may be considered as a higher group by taking
$\B G$ to be the type of $G$-torsors and identifying $G$ with the group of automorhpisms of
$G$ considered as a $G$-torsor. We may also use other types for our deloopings; for example, the type $\B U(1)$ of Hermitian lines (\ref{defn:BU}) deloops the circle group $U(1)$, and the function sending a Hermitian line to its unit circle gives an equivalence of this type with the type of $U(1)$ torsors.

The theory of higher groups is expressed in terms of their deloopings: for example, a
homomorphism $G \to H$ is defined to be a pointed map $\B G \pto \B H$. See
\cite{Buchholtz-vanDoorn-Rijke:Higher.Groups} for a development of the
elementary theory of higher groups in homotopy type theory.

The modal fracture hexagon associated to a (crisp) higher group $G$ will factor
$G$ into its \emph{universal $\infty$-cover} $\inftycover{G}$ (\cref{defn:universal.infty.cover}) and its
\emph{infinitesimal remainder} $\lie{g}$ (\cref{defn:infinitesimal.remainder}). We will now introduce
$\inftycover{G}$ and $\lie{g}$ and prove some lemmas about them which will set
the stage for the modal fracture hexagon.

\subsection{Preliminaries}

In this section, we will prepare a few necessary lemmas concerning the $\flat$ comodality and $\shape$ modality in Shulman's flat type theory
\cite{Shulman:Real.Cohesion}. A first time reader may content
themselves with the the statements of the lemmas, as the proofs are mere technicalities.

\begin{lem}\label{lem:flat_fiber_sequenes}
  The comodality $\flat$ preserves fiber sequences. Let $f :: X \to Y$ be a
  crisp map and $y :: Y$ a crisp point. Then we have and equivalence $\flat
  \fib_f(y) = \fib_{\flat f}(y^{\flat})$ such that
\[
\begin{tikzcd}[row sep = small]
\flat \fib_f(y) \ar[dr, "(-)_{\flat}"] \ar[dd, equals] & \\
& \fib_f(y) \\
\fib_{\flat f}(y^{\flat}) \ar[ur, "\delta"']
\end{tikzcd}
\]
commutes. In particular, $\flat \fib_f(y) \to \flat X \to \flat Y$ is a fiber sequence and
that the naturality squares give a map of fiber sequences:
\[
  \begin{tikzcd}
   \flat \fib_f(y) \ar[r, "\delta"] \ar[d] & \fib_f(y) \ar[d] \\
   \flat X \ar[r] \ar[d] & X \ar[d] \\
   \flat Y \ar[r] & Y
  \end{tikzcd}
\]
\end{lem}
\begin{proof}
  We begin by constructing the equivalence:
  \begin{align*}
    \flat \fib_f(y) &\equiv \flat \left( \dsum{x : X} (f(x) = y) \right) \\
                    &= \dsum{u : \flat X} (\mbox{let $x^{\flat} := u$ in $\flat(f(x) = y)$}) &\mbox{\cite[Lemma.~6.8]{Shulman:Real.Cohesion}} \\
                    &= \dsum{u : \flat X} (\mbox{let $x^{\flat} := u$ in $f(x)^{\flat} = y^{\flat}$}) &\mbox{\cite[Theorem.~6.1]{Shulman:Real.Cohesion}}  \\
                    &\equiv \dsum{u : \flat X} (\mbox{let $x^{\flat} := u$ in $\flat f(x^{\flat}) = y^{\flat}$}) \\
                    &= \dsum{u : \flat X} (\flat f (u) = y^{\flat}) &\mbox{\cite[Lemma.~4.4]{Shulman:Real.Cohesion}} \\
    &\equiv \fib_{\flat f}(y^{\flat}).
  \end{align*} 
We will need to understand what this equivalence does on elements $(x,
p)^{\flat}$ for $(x, p) :: \fib_f(y)$. The first equivalence in the composite
sends $(x, p)^{\flat}$ to $(x^{\flat}, p^{\flat})$, and no other equivalence
affects the first component, so the first component of the result will be
$x^{\flat}$. The second equivalence will send $p^{\flat}$ to $\ap_{\flat}\,
(-)^{\flat}\, p$, where $\ap_{\flat}$ is the crisp application function. The
next equivalence is given by reflexivity, since $\flat f(x^{\flat}) \equiv
f(x)^{\flat}$. In total, then, this equivalence acts as
\[
(x, p)^{\flat} \mapsto (x^{\flat}, \ap_{\flat} (-)^{\flat} p).
\]

 Now, to show the triangle commutes, it will suffice to show that it commutes
 for $(x, p)^{\flat}$ where $(x, p) :: \fib_f(y)$. This is to say, we need to
 show that sending $(x, p)^{\flat}$ through the above equivalence and then into
 $\fib_f(y)$ yields $(x, p)$. The map $\delta$ from $\fib_{\flat f}(y^{\flat})$ to
 $\fib_f(y)$ sends $(u, q)$ to $(u_{\flat}, \square_{\flat}(u) \bullet \ap\,
 (-)_{\flat}\, q)$, where $\square_{\flat}(u) : f(u_{\flat}) = \flat f(u)_{\flat}$
 is the naturality square. So, the round trip $\flat \fib_f(y) \to \fib_{\flat
   f}(y^{\flat}) \to \fib_f(y)$ acts as
 \[
(x, p)^{\flat} \mapsto (x^{\flat},\, \ap_{\flat}\, (-)^{\flat}\, p) \mapsto
(x^{\flat}{}_{\flat},\, \square_{\flat}(x^{\flat}) \bullet \ap\,
 (-)_{\flat}\, (\ap_{\flat}\, (-)^{\flat}\, p) ).
 \]
Now, $x^{\flat}{}_{\flat} \equiv x$, so it remains to show that $\square_{\flat}(x^{\flat}) \bullet \ap\,
 (-)_{\flat}\, (\ap_{\flat}\, (-)^{\flat}\, p)  = p$. However, the naturality
 square is defined by
 $\square_{\flat}(x^{\flat}) \equiv \refl_{f(x)} : f(x^{\flat}{}_{\flat}) = \flat
 f(x^{\flat})_{\flat}$, so it only remains to show that the two applications
 cancel. This can easily be shown by a crisp path induction.
\end{proof}

\begin{lem}\label{lem:flat_etale}
Let $f :: X \to Y$ be a crisp map between crisp types. The following are
equivalent:
\begin{enumerate}
  \item For every crisp $y :: Y$, $\fib_{f}(y)$ is discrete. 
  \item The naturality square
    \[
  \begin{tikzcd}
   \flat X \ar[r, "(-)_{\flat}"] \ar[d, "\flat f"'] &
   X \ar[d, "\pi"] \\
   \flat Y \ar[r, "(-)_{\flat}"'] & Y
  \end{tikzcd}
    \]
    is a pullback.
\end{enumerate}
\end{lem}
\begin{proof}
  We note that the naturality square being a pullback is equivalent to the
  induced map
  \[
\fib_{\flat f}(u) \to \fib_f(u_{\flat})
\]
begin an equivalence for all $u : \flat Y$. By the universal property of $\flat
Y$, we may assume that $u$ is of the form $y^{\flat}$ for a crisp $y :: Y$. By
\cref{lem:flat_fiber_sequenes}, we have that
\[
\begin{tikzcd}[row sep = small]
\flat \fib_f(y) \ar[dr, "(-)_{\flat}"] \ar[dd, equals] & \\
& \fib_f(y) \\
\fib_{\flat f}(y^{\flat}) \ar[ur]
\end{tikzcd}
\]
commutes. Therefore, the naturality square is a pullback if and only if for all
crisp $y :: Y$, we have that $(-)_{\flat} : \flat \fib_f(y) \to \fib_f(y)$ is an
equivalence; but this is precisely what it means for $\fib_f(y)$ to be discrete.
\end{proof}

\begin{lem}\label{lem:flat.id.types}
 Let $X$ be a crisp type, and let $a,\, b :: X$ be crisp elements. Then there is
 an equivalence $e : (a^{\flat} = b^{\flat}) \simeq \flat (a = b)$ together with
 a commutation of the following triangle:
\[\begin{tikzcd}
	{(a^{\flat} = b^{\flat})} \\
	& {(a = b)} \\
	{\flat (a = b)}
	\arrow[from=1-1, to=3-1, equals, "e"']
	\arrow["{\ap\, (-)_{\flat}}", from=1-1, to=2-2]
	\arrow["{(-)_{\flat}}"', from=3-1, to=2-2]
\end{tikzcd}\]
\end{lem}
\begin{proof}
 For the construction of the equivalence $e$ we refer to
 \cite[Theorem.~6.1]{Shulman:Real.Cohesion}. For the commutativity, we use function
 extensionality to work from $u : \flat (a = b)$ seeking $e\inv(u) = u_{\flat}$
 and proceed by
 $\flat$-induction and then identity induction in which case both sides reduce
 to $\refl$.
\end{proof}

\begin{lem}\label{lem:flat.preserves.groups}
Let $G$ be a crisp higher group; that is, suppose that $\B G$ is a crisp,
$0$-connected type and its base point $\pt :: \B G$ is also crisp. Then $\flat
G$ is also a higher group and we may take
$$\B \flat G \equiv \flat \B G$$
pointed at $\pt^{\flat}$. Furthermore, the counit $(-)_{\flat} : \flat G \to G$ is a
homomorphism delooped by the counit $(-)_{\flat} : \flat \B G \to \B G$.
\end{lem}
\begin{proof}
We need to show that $\flat \B G$ deloops $\flat G$ via an equivalence $e :
\Omega \flat \B G = \flat G$, that it is $0$-connected,
and that looping the counit $(-)_{\flat} : \flat \B G \to \B G$ corresponds to
the counit $(-)_{\flat} : \flat G \to G$ along the equivalence $e$.

For the equivalence $e : \Omega \flat \B G = \flat G$, we may take the
equivalence $(\pt^{\flat} = \pt^{\flat}) = \flat (\pt = \pt)$ of
\cref{lem:flat.id.types}. The commutation of the triangle
\[\begin{tikzcd}
	{(\pt^{\flat} = \pt^{\flat})} \\
	& {(\pt = \pt)} \\
	{\flat (\pt = \pt)}
	\arrow[from=1-1, to=3-1, equals, "e"']
	\arrow["{\ap\, (-)_{\flat}}", from=1-1, to=2-2]
	\arrow["{(-)_{\flat}}"', from=3-1, to=2-2]
\end{tikzcd}\]
shows that $(-)_{\flat} : \flat \B G \to \B G$ deloops $(-)_{\flat}
: \flat G \to G$. 

To show that $\flat \B G$ is connected,
we rely on \cite[Corollary.~6.7]{Shulman:Real.Cohesion} which says that $\trunc{\flat
\B G}_0 = \flat \trunc{\B G}_0$, which is $\ast$ by the hypothesis that $\B G$
is $0$-connected.
\end{proof}

We end with a useful lemma: $\flat$ preserves long exact sequences of groups.
\begin{lem}\label{lem:flat.preseres.exact.sequences}
The comodality $\flat$ preserves crisp short and long exact sequences of groups.
\end{lem}
\begin{proof}
  A sequence
  $$0 \to K \to G \to H \to 0$$
  of groups is short exact if and only if its delooping
  $$\B K \pto \B G \pto \B H$$
  is a fiber sequence. But $\flat$ preserves crisp fiber sequences by
  \cref{lem:flat_fiber_sequenes} and by \cref{lem:flat.higher.groups} we have
  that the fiber sequence
$$\flat \B K \pto \flat \B G \pto \flat \B H$$
deloops the sequence 
$\flat K \to \flat G \to \flat H$, 
so this sequence is also short exact.

Now, a complex of groups
$$\cdots \to A_{n-1} \xto{d} A_n \xto{d} A_{n+1} \to \cdots$$
satisfying $d \circ d = 0$ is long exact if and only if the sequences 
$$0 \to K_n \to A_n \to K_{n + 1} \to 0$$
are short exact, where $K_n :\equiv \ker (A_n \to A_{n+1})$. Now, we have a
complex 
$$\cdots \to \flat A_{n-1} \xto{\flat d} \flat A_{n} \xto{\flat d} A_{n + 1} \to
\cdots$$
by the functoriality of $\flat$. 
Since $\flat$ preserves short exact sequences, the sequences 
$$0 \to \flat K_n \to \flat A_n \to \flat A_{n + 1} \to 0$$
are short exact. Now, since $\flat$ preserves fibers we have that 
$$\flat K_n = \ker (\flat A_n \to \flat A_{n + 1}),$$
so that the $\flat$-ed complex is long exact.
\end{proof}

\subsection{The Universal $\infty$-Cover of a Higher Group}\label{sec:universal.cover}

An $\infty$-cover of a type $X$ is a generalization of the notion of cover from
a theory concerning $1$-types (the fundamental groupoid of $X$, with the
universal cover being simply connected) to arbitrary types (the homotopy type of
$X$, with the universal $\infty$-cover being contractible).

Recall that, classically, a covering map $\pi : \tilde{X} \to X$ satisfies a
\emph{unique path lifting property}; that is, every square of the following form
admits a unique filler:
\[
  \begin{tikzcd}
    \ast \ar[r, "\tilde{x}"] \ar[d, "0"'] & \tilde{X} \ar[d, "\pi"] \\
    \Rb \ar[r, "\gamma"'] \ar[ur, dashed, "\exists !"] & X 
  \end{tikzcd}
\]
This property can be extended to a unique lifting property against any map which
induces an equivalence fundamental groupoids. That is, whenever $f : A \to B$
induces an equivalence $\shape_1 f : \shape_1 A \to \shape_1 B$, every 
square of the following form admits a unique filler:
\[
  \begin{tikzcd}
    A \ar[r, "\tilde{\gamma}"] \ar[d, "f"'] & \tilde{X} \ar[d, "\pi"] \\
    B \ar[r, "\gamma"'] \ar[ur, dashed, "\exists !"] & X 
  \end{tikzcd}
\]

For any modality $!$, there is an orthogonal factorization system where
$!$-equivalences (those maps $f$ such that $!f$ is an equivalence) lift uniquely
against \emph{$!$-{\'e}tale maps} (\cite{Rijke:Thesis,Cherubini-Rijke:Modal.Descent}).
\begin{defn}
  A map $f : A \to B$ is $!$-{\'e}tale for a modality $!$ if the naturality
  square
  \[
    \begin{tikzcd}
A \ar[r, "(-)^{!}"] \ar[d, "f"'] & !A \ar[d, "!f"] \\
B \ar[r, "(-)^{!}"'] & !B
\end{tikzcd}
  \]
  is a pullback.
\end{defn}

We may single out the covering maps as the $\shape_1$-{\'e}tale maps whose
fibers are sets. For more on this point of view, see the last section of
\cite{Jaz:Good.Fibrations}. Here, however, we will be more concerned with $\shape$-{\'e}tale maps, which we
will call $\infty$-covers. This notion was called a ``modal covering'' in
\cite{Wellen:Covering}, and was referred to as an $\infty$-cover in the
setting of $\infty$-categories by Schreiber in \cite{Schreiber:Differential.Cohomology}.

\begin{defn}
  A map $\pi : E \to B$ is an \emph{$\infty$-cover} if the naturality square
  \[
    \begin{tikzcd}
      E \ar[d, "\pi"'] \ar[r, "(-)^{\shape}"] & \shape E \ar[d,"\shape \pi"] \\
      B \ar[r, "(-)^{\shape}"] & \shape B
    \end{tikzcd}
  \]
  is a pullback. That is, an $\infty$-cover is precisely a $\shape$-{\'e}tale
  map.

  A map $\pi : E \to B$ is an $n$-cover if it is $\shape_{n+1}$-{\'e}tale and
  its fibers are $n$-types. We call a $1$-cover just a cover, or a covering map.
\end{defn}

Theorem 6.1 of \cite{Jaz:Good.Fibrations} gives a useful way for proving that a
map is an $\infty$-cover. 
\begin{prop}[Theorem 6.1 of \cite{Jaz:Good.Fibrations}]\label{thm:Good.Fibrations.Etale}
Let $\pi : E \to B$ and suppose that there is a crisp, discrete type $F$ so that
for all $b : B$, $\trunc{\fib_{\pi}(b) = F}$. Then $\pi$ is an $\infty$-cover.
\end{prop}

\begin{ex}
  As an example of an $\infty$-cover, consider the exponential map $\Rb \to
  \Sb^1$ from the real line to the circle. The fibers of this map are all merely
  $\Zb$, so by Theorem \ref{thm:Good.Fibrations.Etale}, this map is an
  $\infty$-cover. Since $\Rb$ is contractible, it is in fact the
  \emph{universal} $\infty$-cover of the circle.
\end{ex}

Just as the universal cover of a space $X$ is any simply connected cover $\tilde{X}$, the
universal $\infty$-cover $\inftycover{X}$ of a type $X$ is any \emph{contractible} cover ---
contractible in the sense of being $\shape$-connected, meaning $\shape \inftycover{X} =
\ast$. Since units of a modality are modally connected, we may always construct a universal
$\infty$-cover by taking the fiber of the $\shape$-unit $(-)^{\shape} : X \to
\shape X$.
\begin{defn}\label{defn:universal.infty.cover}
  The \emph{universal} $\infty$-cover of a pointed type $X$ is defined to be the fiber
  of the $\shape$-unit:
  $$\inftycover{X} :\equiv \fib((-)^{\shape} : X \to \shape X).$$
  Since the units of modalities are modally connected, $\inftycover{X}$ is
  homotopically contractible:
  $$\shape \inftycover{X} = \ast.$$
\end{defn}

Let's take a bit to get an image of the universal $\infty$-cover of a type. The
universal $\infty$-cover of a type only differs from its universal cover in the
identifications between its points; in other words, it is a ``stacky'' version
of the universal cover.
\begin{prop}
Let $X$ be a crisp type. Then the map $\inftycover{X} \to \tilde{X}$ from the
universal $\infty$-cover of $X$ to its universal cover induced by the square
\[
  \begin{tikzcd}
  \inftycover{X} \ar[r, "\pi"] \ar[d, dashed] & X \ar[d, equals] \ar[r,
  "(-)^{\shape}"] & \shape X \ar[d] \\
  \tilde{X} \ar[r, "\pi"'] & X \ar[r,"(-)^{\shape_1}"'] & \shape_1 X 
  \end{tikzcd}
\]
is $\trunc{-}_0$-connected and $\shape$-modal. In particular, if $X$ is a set then
$\trunc{\inftycover{X}}_0 = \tilde{X}$. Furthermore, its fibers may be
identified with the loop space $\Omega ( (\shape X)\langle 1 \rangle )$ of the first stage of the Whitehead tower of the
shape of $X$. 
\end{prop}  
\begin{proof}
We begin by noting that the unique factorization $f : \shape X \to \shape_1 X$ of the $\shape_1$ unit
$(-)^{\shape_1} : X \to \shape_1 X$ through the $\shape$ unit is a $\trunc{-}_1$
unit. We note that $(-)^{\shape_1} : X \to \shape_1 X$ and $|(-)^{\shape}|_1 : X
\to \trunc{\shape X}_1$ have the same universal
property: any map from $X$ to a discrete $1$-type factors uniquely through them.
However, unless $X$ is crisp, we do not know that $\trunc{\shape X}_1$ is itself
discrete; in general, we can only conclude that there is a map $\trunc{\shape
  X}_1 \to \shape_1 X$. This is why we must assume that $X$ is crisp. By Proposition 4.5 of
\cite{Jaz:Good.Fibrations}, $\trunc{\shape X}_1$ is discrete and therefore the
map $\trunc{\shape X}_1 \to \shape_1 X$ is an equivalence. Since $f$ factors
uniquely through this map (since $\shape_1$ is a $1$-type), we see that $f$ is
equal to the $\trunc{-}_1$ unit of $\shape X$ and is therefore a $\trunc{-}_1$ unit. In particular, $f :
\shape X \to \shape_1 X$ is $1$-connected.

/ow we will show that the fibers of the induced map $\inftycover{X} \to
\tilde{X}$ are $0$-connected and discrete. Consider the following diagram:
\[
\begin{tikzcd}
\fib(p) \arrow[r] \arrow[d]               & \ast \arrow[r] \arrow[d]              & \fib((\pi p)^{\shape_1}) \arrow[d] \\
\inftycover{X} \arrow[d] \arrow[r, "\pi"] & X \arrow[d, equals] \arrow[r, "(-)^{\shape}"] & \shape X \arrow[d]                 \\
\tilde{X} \arrow[r, "\pi"']               & X \arrow[r, "(-)^{\shape_1}"']        & \shape_1 X                        
\end{tikzcd}
\]
All vertical sequences are fiber sequences, and the bottom two sequences are
fiber sequences; therefore, the top sequence is a fiber sequence, which tells us
that the fiber over any point $p : \tilde{X}$ is equivalent to $\Omega
\fib_{f}((\pi p)^{\shape_1})$. As we have shown that the fibers of $f$ are
$1$-connected, their loop spaces are $0$-connected. And as $f$ is a map between
discrete types, its fibers are discrete and so their loop spaces are discrete.
Finally, we note that the fiber of the $1$-truncation $|\cdot|_1 : \shape X \to
\trunc{\shape X}_1$ is the first stage $(\shape X)\langle 1 \rangle$ of the
Whitehead tower of $\shape X$.
\end{proof}

We are now ready to prove a simple sort of fracture theorem for any crisp, pointed type. This
square
will make up the left square of our modal fracture hexagon.
\begin{prop}\label{prop:left_fracture}
Let $X$ be a crisp, pointed type. Then the $\flat$ naturality square of the universal
$\infty$-cover $\pi : \inftycover{X} \to X$ is a pullback:
\[
  \begin{tikzcd}
   \flat \inftycover{X} \ar[r, "(-)_{\flat}"] \ar[d, "\flat \pi"'] &
   \inftycover{X} \ar[d, "\pi"] \\
   \flat X \ar[r, "(-)_{\flat}"'] & X
  \end{tikzcd}
\]
\end{prop}
\begin{proof}
  By \cref{lem:flat_etale}, it will suffice to show that over a crisp point $x
  :: X$, $\fib_{\pi}(x)$ is discrete. But since $\pi$ is, by definition, the
  fiber of $(-)^{\shape}$, we have that
  \[
\fib_{\pi}(x) = \Omega(\shape X, x^{\shape}).
\]
Since $\shape X$ is discrete by assumption, so is $\Omega(\shape X, x^{\shape})$.
\end{proof}

Although $\shape$ is \emph{not} a left exact modality --- it does not preserve
all pullbacks --- it does preserve pullbacks and fibers of
\emph{$\shape$-fibrations}. The theory of modal fibrations was developed in
\cite{Jaz:Good.Fibrations}. Included amongst the $\shape$-fibrations are the
$\shape$-{\'e}tale maps, and so $\shape$ preserves fiber sequences of
$\shape$-{\'e}tale maps.
\begin{lem}
  Let $f : X \to Y$ be an $\infty$-cover. Then for any $y : Y$, the sequence
  \[
\shape \fib_{f}(y) \to \shape X \xto{\shape f} \shape Y
\]
is a fiber sequence.
\end{lem}
\begin{proof}
  Since a $\shape$-{\'e}tale map $f$ is modal, its {\'e}tale and modal factors
  agree (they are equivalently $f$), so by Theorem 1.2 of \cite{Jaz:Good.Fibrations},
  $f$ is a $\shape$-fibration. The result then follows since $\shape$ preserves
  all fibers of $\shape$-fibrations (see also Theorem 1.2 of \cite{Jaz:Good.Fibrations}).
\end{proof}

Importantly, it is also true that the shape of a \emph{crisp} $n$-connected type
is also $n$-connected by Theorem 8.6 of \cite{Jaz:Good.Fibrations}. It follows
that $\shape \B G$ is a delooping of $\shape G$ for crisp higher groups $G$, and
that this can continue for higher deloopings.
\begin{prop}\label{prop:top.sequence}
  Let $G$ be a crisp higher group. Then its universal $\infty$-cover $\inftycover{G}$
  is a higher group and $\pi : \inftycover{G} \to G$ is a homomorphism.
  Futhermore, if $G$ is $k$-commutative, then so is $\inftycover{G}$.
\end{prop}
\begin{proof}
  We may define
  $$\B^i \inftycover{G} :\equiv \fib((-)^{\shape} : \B^i G \to \shape \B^i G ).$$
  This lets us extend the fiber sequence:
\begin{center}
    \begin{tikzcd}
\inftycover{G} \ar[r, "\pi"] & G \arrow[r, "(-)^{\shape}"] & \shape{G} \arrow[lld, out=-30, in=150] \\
 \B\inftycover{G} \ar[r, "\B\pi"]   & \B G  \ar[r, "(-)^{\shape}"]       & \shape\B G \arrow[lld, out=-30, in=150]
 \\
 \B^2 \inftycover{G} \arrow[r, "\B^2 \pi"]    & \B^2 G \ar[r, "(-)^{\shape}"]
 & \shape \B^2 G\arrow[lld, out=-30, in=150]
 \\
 \cdots & & 
\end{tikzcd}
\end{center}\qedhere
  
\end{proof}

\subsection{The Infinitesimal Remainder of a Higher Group}\label{sec:infinitesimal.remainder}

In this section, we will investigate the infinitesimal remainder $\theta : G \to
\lie{g}$ of higher group $G$. The infinitesimal remainder is what is left of a
higher group when all of its crisp points have been made equal. Having
trivialized all substanstial difference between points, we are left with the
infinitesimal differences that remain.

\begin{defn}
  Let $G$ be a higher group. Define its \emph{infinitesimal remainder} to be
  $$\lie{g} :\equiv \fib((-)_{\flat} : \flat \B G \to \B G).$$
  Then, continuing the fiber sequence, we have
\begin{center}
    \begin{tikzcd}
{\flat G} \ar[r, "(-)_{\flat}"] & G \arrow[r, "\theta"] & \lie{g} \arrow[lld, out=-30, in=150] \\
 \flat \B G \arrow[r]    & \B G         &
\end{tikzcd}
\end{center}
which defines the quotient map $\theta : G \to \lie{g}$.
\end{defn}

\begin{rmk}
By its construction, we can see that $\lie{g}$ modulates flat connections on
trivial principal $G$-bundles, with respect to the interpretation of $\flat \B
G$ given in Remark \ref{rmk:flat_name}. In the setting of differential geometry,
such flat connections on trivial principal $G$-bundles are given by closed
$\mathfrak{g}$-valued $1$-forms, where here $\mathfrak{g}$ is the Lie algebra of the Lie
group $G$. In this setting, $\theta$ is the Mauer-Cartan form on $G$. This is
why we adopt the name $\theta : G \to \lie{g}$ for the infinitesimal remainder
in general. This can in fact be \emph{proven} in the setting of synthetic
differential geometry with tiny infinitesimals satisfying a principle of
constancy using a purely modal argument. See \cref{rmk:future.work.sdg} for a
further discussion.
\end{rmk}

\begin{rmk}
The infinitesimal remainder $\lie{g}$ is defined as the \emph{de Rham coefficient object}
of $\B G$ in Definition 5.2.59 of \cite{Schreiber:Differential.Cohomology}.
Schreiber defined $\flat_{dR} X$ for any (crisp) pointed type $X$ as the fiber
of $(-)^{\flat} : \flat X \to X$, so that $\lie{g} \equiv \flat_{dR} \B G$. We
focus on the case that $X$ is $0$-connected --- of the form $\B G$ --- and so
only consider the infinitesimal remainder of a higher group $G$.
\end{rmk}

While the infinitesimal remainder exists for any (crisp) higher group, it is not
necessarily itself a higher group. However, if $G$ is braided, then $\lie{g}$
will be a higher group.
\begin{prop}
If $G$ is a crisp $k$-commutative higher group, then $\lie{g}$ is a
$(k-1)$-commutative higher group. In particular, if $G$ is a braided higher
group, then $\lie{g}$ is a higher group and the remainder map $\theta : G \to
\lie{g}$ is a homomorphism.
\end{prop}
\begin{proof}
  We may define \[
\B^{i}\lie{g} :\equiv \fib((-)_{\flat} : \flat \B^{i+1} G \to \B^{i+1} G),
\]
which lets us continue the fiber sequence:
\begin{center}
    \begin{tikzcd}
{\flat G} \ar[r, "(-)_{\flat}"] & G \arrow[r, "\theta"] & \lie{g} \arrow[lld, out=-30, in=150] \\
 \flat \B G \arrow[r]    & \B G  \ar[r, "\B\theta"]       & \B\lie{g}\arrow[lld, out=-30, in=150]
 \\
 \flat \B^2 G \arrow[r]    & \B^2 G \ar[r, "\B^2\theta"]      & \B^2\lie{g}\arrow[lld, out=-30, in=150]
 \\
 \cdots & & 
\end{tikzcd}
\end{center}\qedhere
\end{proof}

\begin{rmk}
We can see the delooping $\B \theta : \B G \to \B \lie{g}$ of the infinitesimal remainder $\theta : G
\to \lie{g}$ as taking the \emph{curvature} of a principal $G$-bundle, in that
$\B \theta$ is an obstruction to the flatness of that bundle since
$$\flat \B G \to \B G \to \B \lie{g}$$
is a fiber sequence.
\end{rmk}

As with any good construction, the infinitesimal remainder is functorial in its
higher group. This is defined easily since the infinitesimal remainder is
constructed as a fiber.
\begin{defn}\label{defn:infinitesimal.remainder}
Let $f :: G \to H$ be a crisp homomorphism of higher groups with delooping $\B f
: \B G \pto \B H$. Then we have a pushforward $f_{\ast} : \lie{g} \pto \lie{h}$
given by $(t, p) \mapsto (\flat\B f(t), (\ap\, \B f\, p) \bullet \pt_{\B f} )$. 
This is the unique map fitting into the following diagram:
\[
  \begin{tikzcd}
    \lie{g} \ar[d] \ar[r, "f_{\ast}"] & \lie{h} \ar[d] \\
    \flat \B G \ar[d] \ar[r, "\flat \B f"] & \flat \B H \ar[d] \\
    \B G \ar[r, "\B f"'] & \B H
  \end{tikzcd}
\]
If $G$ and $H$ are $k$-commutative and $f$ is a $k$-commutative homomorphism,
then $f_{\ast}$ admits a unique structure of a $(k-1)$-commutative homomorphism
by defining $\B^{k-1} f_{\ast}$ to be the map induced by $\flat \B^k f$ on the fiber.
\end{defn}

We record a useful lemma: the fibers of the quotient map $ \theta : G \to
\lie{g}$ are all are identifiable with $\flat G$.
\begin{lem}\label{lem:fibers.of.inf.remainder}
  Let $G$ be a crisp higher group. For $t : \lie{g}$, we have $\trunc{\fib_{\theta}(t) = \flat G}$.
\end{lem}
\begin{proof}
By definition, $t : \lie{g}$ is of the form $(T, p)$ for $T : \flat \B G$ and $p
: T_{\flat} = \pt_{\B G}$. Since $\flat \B G$ is $0$-connected and we are trying
to prove a proposition, we may suppose that $q : T = \pt_{\flat \B G}$. We then
have that $t = (\pt_{\flat B G}, (-)_{\flat}{}_{\ast}(q) \bullet p)$, and
  therefore:
\begin{align*}
\fib_{\theta}(t) &\equiv \dsum{g : G} \left( (\pt_{\flat \B G}, g) = (\pt_{\flat \B G},(-)_{\flat}{}_{\ast}(q) \bullet p ) \right) \\
&=  \dsum{g : G} \dsum{a : \flat G} \left( a_{\flat} \bullet g  = (-)_{\flat}{}_{\ast}(q) \bullet p\right) \\
&=  \dsum{g : G} \dsum{a : \flat G} \left(  g  = a_{\flat}\inv \bullet(-)_{\flat}{}_{\ast}(q) \bullet p\right) \\
                 &= \flat G.\qedhere
\end{align*}

\end{proof}

The infinitesimal remainder is \emph{infinitesimal} in the sense that it has a
single crisp point. 
\begin{prop}\label{prop:inf.remainder.infinitesimal}
Let $G$ be a higher group. Then its infinitesimal remainder $\lie{g}$ is
infinitesimal in the sense that
$$\flat \lie{g} = \ast.$$
\end{prop}
\begin{proof}
  By \cref{lem:flat_fiber_sequenes}, $\flat$ preserves the fiber sequence
  $$\lie{g} \to \flat \B G \to \B G.$$
 But $\flat(-)_{\flat} : \flat \flat \B G \to \flat \B G$ is an equivalence by
 Theorem 6.18 of \cite{Shulman:Real.Cohesion}, so
 $\flat \lie{g}$ is contractible.
\end{proof}
Despite being infinitesimal, we will see that $\lie{g}$ has (in general) a highly non-trivial
homotopy type.

\begin{rmk}
The infinitesimal remainder $\lie{g}$ is of special interest when $G$ is a Lie
group, since in this case the vanishing of the cohomology groups
$H^{\ast}(\lie{g}; \Zb/p )$ for all primes $p$ is equivalent to the
Friedlander-Milnor conjecture. The fact that this conjecture remains unproven is
a testament to the intricacy of the homotopy type of the infinitesimal space $\lie{g}$.
\end{rmk}

We note that $\lie{g}$ itself represents an obstruction to the discreteness of $G$.
\begin{prop}\label{prop:discrete.iff.remainder.contractible}
A crisp higher group $G$ is discrete if and only if its infinitesimal remainder
$\lie{g}$ is contractible.
\end{prop}
\begin{proof}
If $G$ is discrete, then $(-)_{\flat} : \flat G \to G$ is an equivalence and so
$(-)_{\flat} : \flat \B G \to \B G$ is an equivalence: this implies that
$\lie{g} = \ast$. On the other hand, if $\lie{g} = \ast$ then $(-)_{\flat} :
\flat \B G \to \B G$ is an equivalence and so its action on loops is an equivalence.
\end{proof}

Using \cref{thm:Good.Fibrations.Etale}, we can quickly show that $\theta : G \to
\lie{g}$ is an $\infty$-cover. This gives us the right hand pullback square in our modal
fracture hexagon.
\begin{prop}\label{prop:inf.remainder.infty.cover}
Let $G$ be a crisp $\infty$-group. Then the infinitesimal remainder $\theta : G
\to \lie{g}$ is an $\infty$-cover. In particular, the $\shape$-naturality
square:
\[
\begin{tikzcd}
  G \ar[r, "\theta"] \ar[d, "(-)^{\shape}"'] & \lie{g} \ar[d, "(-)^{\shape}"] \\
  \shape G \ar[r, "\shape \theta"'] & \shape \lie{g}
\end{tikzcd}
\]
is a pullback. If, furthermore, $G$ is (crisply) an $n$-type, then $\theta$ is an $(n + 1)$-cover.
\end{prop}
\begin{proof}
By \cref{thm:Good.Fibrations.Etale}, to show that $\theta : G \to \lie{g}$ is an
$\infty$-cover (resp. an $(n + 1)$-cover) it suffices to show that the fibers are
merely equivalent to a crisply discrete type (resp. a crisply discrete $n$-type). But by
\cref{lem:fibers.of.inf.remainder}, the fibers of $\theta : G \to \lie{g}$ are
all merely equivalent to $\flat G$ which is crisply discrete (and, by Theorem
6.6 of \cite{Shulman:Real.Cohesion}, if $G$ is (crisply) an $n$-type then so is $\flat G$).
\end{proof}

There is a sense in which the infinitesimal remainder of a higher group behaves
like its Lie algebra. Just as the Lie algebra of a Lie group is the same as the Lie algebra of its
universal cover, we can show that the infinitesimal remainder of a higher group
is the same as that of its universal $\infty$-cover.
\begin{prop}\label{prop:exact.sequence.inf.remainder}
  Let $G \xto{\phi} H \xto{\psi} K$ be a crisp exact sequence of higher groups. Then
  \begin{enumerate}
  \item $K$ is discrete if and only if $ \phi_{\ast} : \lie{g} \to \lie{h}$ is
    an equivalence.
  \item $G$ is discrete if and only if $\psi_{\ast} : \lie{h} \to \lie{k}$ is an
    equivalence.
  \end{enumerate}
\end{prop}
\begin{proof}
We consider the following diagram in which each horizontal and vertical sequence
is a fiber sequence:
\[
  \begin{tikzcd}
    \lie{g} \ar[d] \ar[r] & \lie{h}\ar[d] \ar[r] & \lie{k} \ar[d] \\
    \flat \B G \ar[d] \ar[r]& \flat \B H \ar[d] \ar[r]& \flat \B K \ar[d] \\
    \B G \ar[r]& \B H\ar[r] & \B K
  \end{tikzcd}
\]
If $K$ is discrete, then $\lie{k} = \ast$ and so $\lie{g} = \lie{h}$. On the
other hand, if $\lie{g} = \lie{h}$, then the bottom left square of the above
diagram is a pullback. Therefore, the it induces an equivalence on the fibers of
the horizontal maps:
\[
  \begin{tikzcd}
    \flat K \ar[d, "\sim"] \ar[r] & \flat \B G \ar[d] \ar[r] & \flat \B H \ar[d] \\
    K \ar[r] & \B G \ar[r] & \B H
  \end{tikzcd}
\]
This shows that $K$ is discrete.

If $\psi_{\ast} : \lie{h} \to \lie{k}$ is an equivalence, then its fiber
$\lie{g}$ is contractible. Therefore, $G$ is discrete. On the other hand, if $G$ is
discrete, then the bottom right square is a pullback, and therefore the induced
map on vertical fibers is an equivalence. This map is $\psi_{\ast} : \lie{h} \to
\lie{k}$.
\end{proof}

\begin{cor}\label{cor:inf_remainder_infty_cover}
The universal $\infty$-cover $\pi : \inftycover{G} \to G$ induces an equivalence
$\inftycover{\lie{g}} = \lie{g}$ fitting into the following commutative
diagram:
\[
  \begin{tikzcd}
    \inftycover{ \lie{g} } \ar[d] \ar[r, equals] & \lie{g}\ar[d] \ar[r] & \ast \ar[d] \\
    \flat \B \inftycover{G} \ar[d] \ar[r]& \flat \B G \ar[d] \ar[r]& \flat
    \shape \B G \ar[d, equals] \\
    \B \inftycover{G} \ar[r]& \B G\ar[r] & \shape \B G 
  \end{tikzcd}
\]

In particular, this gives us a long
fiber sequence
\[
    \begin{tikzcd}
\flat\inftycover{G} \ar[r, "(-)_{\flat}"] & \inftycover{G} \arrow[r, "\theta"] & \lie{g} \arrow[lld, out=-30, in=150] \\
 \flat\B\inftycover{G} \ar[r, "\B\pi"]   & \B \inftycover{G} 
\end{tikzcd}
\]
which forms the top fiber sequence of the modal fracture hexagon.
\end{cor}

\subsection{The Modal Fracture Hexagon}

We have seen the two main fiber sequences
\[
    \begin{tikzcd}
\inftycover{G} \ar[r, "\pi"] & G \arrow[r, "(-)^{\shape}"] & \shape{G} \arrow[lld, out=-30, in=150] \\
 \B\inftycover{G} \ar[r, "\B\pi"]   & \B G  \ar[r, "(-)^{\shape}"]       & \shape\B G 
\end{tikzcd}
\quad\mbox{and}\quad
\begin{tikzcd}
{\flat G} \ar[r, "(-)_{\flat}"] & G \arrow[r, "\theta"] & \lie{g} \arrow[lld, out=-30, in=150] \\
\flat \B G \arrow[r]    & \B G         &
\end{tikzcd}
\]
associated to a higher group $G$. Now, when we apply $\flat$ to the left sequence and $\shape$ to the right
sequence, we find the sequences 
\[
    \begin{tikzcd}
\flat\inftycover{G} \ar[r, "\flat\pi"] & \flat G \arrow[r, "(-)^{\shape} \circ (-)_{\flat}"] & \shape{G} \arrow[lld, out=-30, in=150] \\
 \flat\B\inftycover{G} \ar[r, "\flat\B\pi"]   & \flat\B G  \ar[r,
 "(-)^{\shape}\circ (-)_{\flat}"]       & \shape\B G 
\end{tikzcd}
\quad\mbox{and}\quad
\begin{tikzcd}
{\flat G} \ar[r, "(-)^{\shape} \circ (-)_{\flat}"] & \shape G \arrow[r, "\shape\theta"] & \shape\lie{g} \arrow[lld, out=-30, in=150] \\
\flat \B G \arrow[r]    & \shape \B G         &
\end{tikzcd}
\]
which are \emph{the same sequence, just shifted over}. This gives us the bottom exact sequence of
our modal fracture hexagon, reading the sequence on the left. But it also proves that $\shape \lie{g} = \flat \B
\inftycover{G}$, which gives us the top exact sequence of our modal fracture
hexagon by \cref{cor:inf_remainder_infty_cover}.

Of course, we need to be able to apply $\shape$ freely to fiber sequences to
fulfil this argument. But $\shape$ is not left exact, and so does not preserve fiber
sequences in general. Luckily, Theorem 6.1 of \cite{Jaz:Good.Fibrations} gives us a
trick for showing that a map is a $\shape$-fibration which allows us to prove
this general lemma.
\begin{lem}\label{lem:classifying.map.fibration}
Let $G$ be a crisp higher group (that is, $\B G$ is a crisply pointed
$0$-connected type). Then any crisp map $f :: X \to \B G$ is a $\shape$-fibration.
\end{lem}
\begin{proof}
Since $\B G$ is $0$-connected, all the fibers of $f$ are merely equivalent to
the fiber $\fib_f(\pt)$ over the basepoint. Therefore, their homotopy types are merely
equivalent to $\shape \fib_f(\pt)$, which is a crisp, discrete type. It follows
by Theorem 6.1 of \cite{Jaz:Good.Fibrations} that $f$ is a $\shape$-fibration.
\end{proof}

This means that we can freely apply $\shape$ to crisp fiber sequences of
$0$-connected types. This concludes our proof of the main theorem.

\begin{thm}\label{thm.main_theorem}
  For a crisp $\infty$-group $G$, there is a \emph{modal
    fracture hexagon}:
  \[
    \begin{tikzcd}
      & \inftycover{G} \ar[dr, color=blue, "\pi"] \ar[rr] & & \lie{g}
      \ar[dr,  "(-)^{\shape}"] \\
      \flat \inftycover{G} \ar[ur, "(-)_{\flat}"] \ar[dr,  "\flat \pi"'] & &
       G \ar[dr,color=blue, "(-)^{\shape}"'] \ar[ur, color=red, " \theta"] & & \shape \lie{g} \\
      & \flat  G \ar[rr] \ar[ur, color=red, "(-)_{\flat}"'] & & \shape  G \ar[ur,
     "\shape  \theta"'] &
    \end{tikzcd}
  \]
  where
  \begin{itemize}
  \item $\theta : G \to \lie{g}$ is the \emph{infinitesimal remainder} of $G$, the quotient $G\sslash \flat G$, and 
  \item $\pi : \inftycover{G} \to G$ is the \emph{universal (contractible) $\infty$-cover} of $G$.
  \end{itemize}
  Moreover, 
  \begin{enumerate}
  \item The middle diagonal sequences are fiber sequences.
  \item The top and bottom sequences are fiber sequences.
  \item Both squares are pullbacks.
    \end{enumerate}
  Furthermore, the homotopy type of $\lie{g}$ is a delooping of $\flat\inftycover{G}$:
    $$\shape \lie{g} = \flat \B \inftycover{G}.$$
Therefore, if $G$ is $k$-commutative for $k \geq
  1$ (that is, admits futher deloopings $\B^{k + 1} G$), then we may continue the
  modal fracture hexagon on to $\B^{k}G$.
\end{thm}
\begin{proof}
  We assemble the various components of the proof here.
  \begin{enumerate}
    \item The middle diagonal sequences are fiber sequences by definition (see
      \cref{defn:universal.infty.cover} and \cref{defn:infinitesimal.remainder}).
    \item The top sequence was shown to be a fiber sequence in
      \cref{cor:inf_remainder_infty_cover}. We showed that the bottom sequence
      is a fiber sequence at the beginning of this subsection.
    \item The left square was shown to be a pullback in \cref{prop:left_fracture}, and the right
      sequence in \cref{prop:inf.remainder.infty.cover}.
  \end{enumerate}
  Finally, we calculated the homotopy type of $\lie{g}$ at the beginning of this subsection.
\end{proof}
\section{Ordinary Differential Cohomology}\label{sec:diff.coh}

In this section, we will use modal fracture to construct ordinary differential
cohomology in cohesive homotopy type theory. We will recover a differential
hexagon for ordinary differential cohomology which very closely resembles the
classical hexagon; however, as de Rham's theorem does not hold for all types, we
will not recover the classical hexagon exactly. For more discussion of these
subtleties, see \cref{sec:subtleties}.

In \cite{Bunke-Nikolaus-Volke:Differential.Cohomology}, Bunke, Nikolaus, and V{\:o}kel show that
differential cohomology theories can be understood as spectra in the
$\infty$-topos of sheaves on a site of manifolds. Schreiber notes in Proposition
4.4.9 of \cite{Schreiber:Differential.Cohomology} that the simpler site
consisting of Euclidean spaces and smooth maps between them yields the same topos
of sheaves, and proves in Proposition 4.4.8 that this $\infty$-topos is
cohesive. This topos, and the similar $\infty$-Dubuc topos (called the
$\infty$-Cahiers topos in Remark 4.5.6 and
$\textbf{SynthDiff}\infty\textbf{Grpd}$ in Definition 4.5.7 of \emph{ibid.}),
will be our intended model for cohesive homotopy type theory in this section.

The theme of this paper is that the main feature of differential cohomology
--- the differential cohomology hexagon --- is not of a particularly
differential character, but arises from the more basic opposition between an
adjoint modality $\shape$ and comodality $\flat$. As we saw in the previous
section, in the presence of these (co)modalities, any higher group may be
fractured in a manner resembling the differential cohomology hexagon.

We will take a similarly general view in constructing ordinary differential
cohomology. The key idea in ordinary differential cohomology is the equipping of
differential form data to integral cohomology. We will therefore focus on
cohomology theories (in particular, $\infty$-commutative higher groups or
connective spectra) which arise by equipping an existing cohomology theory with
extra data representing the cocycles. Our exposition will focus on ordinary
differential cohomology, but this extra generality will enable us to define
combinatorial analogues of ordinary differential cohomology as well (see
\cref{sec:combinatorial.analogues}).

\subsection{Assumptions and Preliminaries}

We will need to assume something extra in our type theory to give it a
differential geometric character, since there are many models of cohesion which
have nothing to do with differential geometry. We will make the following
assumptions in this section.
\begin{assumption}\label{assumption:SDG}
  We assume the following in this section. All assumptions are crisp.
  \begin{itemize}
    \item (Smooth reals): We assume the existence of the \emph{smooth reals}
      $\Rb$, which externally will be the smooth set represented by the manifold
      $\Rb$. We will not be entirely specific here about what axioms $\Rb$
      must satisfy (see \cref{rmk:future.work.sdg}), and will instead rely on
      ``obvious'' facts about real numbers. For our current purposes, it should
      more than suffice to take $\Rb$ to be a non-trivial ring of characteristic $0$ and to
      define $U(1)$ to be the abelian group $\{z : \Cb \mid z
      \bar{z} = 1\}$ of units in the \emph{smooth} complex
      numbers, which are defined as $\Cb \equiv \Rb[i]/(i^2 + 1)$$\Rb/\Zb$. Note: the smooth reals
      will \emph{not} be the Dedekind reals.
     \item (Axiom $\Rb\flat$): Following \cite{Shulman:Real.Cohesion}, we use
       axiom $\Rb\flat$ with $\Rb$ the smooth reals. That is, for a crisp type
       $X$, the $\flat$ counit $(-)_{\flat} : \flat X \to X$ is an equivalence if and only if $X$
       is $\Rb$-local, i.e. the
       inclusion of constant maps $X \to (\Rb \to X)$ is an equivalence.
     \item (Form classifiers): We now come to the assumption which has a special differential
       geometric character. We will assume the existence of a \emph{contractible
         and infinitesimal resolution} of $U(1)$
         \begin{equation}\label{eqn.contractible.resolution}
           0 \to \flat U(1) \to U(1) \xto{d} \Lambda^1 \xto{d} \Lambda^2 \xto{d} \cdots
         \end{equation}
         given by the \emph{differential $k$-form classifiers} $\Lambda^k$. That is:
         \begin{itemize}
         \item The $\Lambda^k$ are crisp $\Rb$-vector spaces.
         \item The maps $d :: \Lambda^k \to \Lambda^{k+1}$ are crisply $\flat
           \Rb$-linear (not $\Rb$-linear!), and the sequence
           \cref{eqn.contractible.resolution} is crisply long exact.
         \item The $\Lambda^k$ are infinitesimal: $\flat \Lambda^k = \ast$. Therefore
           the closed $k$-form classifiers $\Lambda^k_{\text{cl}} :\equiv \ker(d : \Lambda^k \to \Lambda^{k +
             1})$ are also infinitesimal.
         \end{itemize}

         Externally, the form classifier $\Lambda^k$ will be the sheaf of
         $k$-forms. We will not assume anything in particular to make these
         $\Lambda^k$ actually classify forms (but see
         \cref{remk:future.work.sdg}); we will only assume what we need to
         construct ordinary differential cohomology. To say that $\Lambda^k$ is
         an $\Rb$-vector space means externally that the sheaf of $k$-forms is a
         $\Ca^{\infty}$-module, and the existence of this long exact sequence is
         essentially de Rham's theorem for $\Rb^n$.
    \end{itemize}
\end{assumption}

\begin{rmk}\label{rmk:future.work.sdg}
  For reasons of space, we will not justify this assumption in this paper. In forthcoming work, we will show how one can construct the form classifiers
  and their long exact sequence
\begin{equation}\label{eqn:long.exact.form.class}
0 \to \flat \Rb \to \Rb \xto{d} \Lambda^1 \xto{d} \Lambda^2 \to \cdots
\end{equation}

 \noindent from the axioms of synthetic diffential geometry
  with tiny infinitesimals and a principle of constancy. Synthetic differential
  geometry is an axiomatic system for working with nilpotent infinitesimals put
  forward first by Lawvere \cite{Lawvere:Categorical.Dynamics} and developed by Bunge, Dubuc, Kock, Wraith,
  and others. It admits a model in sheaves on infinitesimally extended
  Euclidean spaces, known as the Dubuc topos (or Cahiers topos)
  \cite{Dubuc:Topos}; for a review of models see
  \cite{Moerdijk-Reyes:Models.SDG}. The Dubuc topos is cohesive, and is our
  intended model for this section.

  It was noted by Lawvere \cite{Lawvere:SDG} that the exceptional projectivity enjoyed by the
  infinitesimal interval $\Db = \{\epsilon : \Rb \mid \epsilon^2 = 0\}$ was
  equivalent to the existence of an (external) right adjoint to the exponential
  functor $X \mapsto X^{\Db}$. We will follow Yetter
  \cite{Yetter:Tiny} in calling objects $T$ for which the functor
  $X \mapsto X^{T}$ admits a right adjoint \emph{tiny objects}. Lawvere and Kock showed how one could use this
  ``amazing'' right adjoint to construct the form classifiers $\Lambda^k$ (see
  Section I.20 of
  \cite{Kock:SDG} for a construction of $\Lambda^1$).

  However, working with the form classifiers was difficult in synthetic
  differential geometry since the adjoint which defines them only exists
  externally. This may be remedied by using Shulman's Cohesive HoTT, where the
  $\sharp$ modality allows for an internalization of the external. This allows
  us to give a fully internal theory of the form classifiers. We will, however,
  post-pone a discussion of this internal theory of tiny objects to future work.

  The principle of constancy says that if the differential of a function $f :
  \Rb \to \Rb$ vanishes uniformly, then $f$ is constant. This extra principle has
  been long considered in synthetic differential geometry (see, for example, the
  second chapter of \cite{Bunge-Gabo-SanLuis:Synthetic.Differential.Topology}), but when
  combined with real cohesion it implies the exactness of the sequence
  $$0 \to \flat \Rb \to \Rb \xto{d} \Lambda^1$$
  and so begins the theory of differential cohomology we will see shortly. The
  interaction with cohesion is non-trivial in many ways for synthetic
  differential geomtry. For example, the principle of constancy in the presence
  of real cohesion implies the existence of primitives, and the exponential
  functions $\exp(-) : \Rb \to \Rb^+$ and $\exp(2 \pi i -) : \Rb \to U(1)$
  (where $U(1) := \{z : \Cb \mid z\overline{z} = 1\}$). 
\end{rmk}

\begin{rmk}
  Externally, the smooth reals $\Rb$ correspond to the sheaf of smooth real valued
  functions, represented by the manifold $\Rb$, $U(1)$ corresponds to the sheaf of smooth $U(1)$-valued functions,
  and $\Lambda^k$ is the sheaf sending a manifold to its set of differential
  $k$-forms.
\end{rmk}

We have assumed the existence of a crisp long exact sequence of abelian groups
in which each of the $\Lambda^i$ are real vector spaces (but $d : \Lambda^i \to
\Lambda^{i + 1}$ are not $\Rb$-linear). These are to be the \emph{differential form
  classifiers}, which externally are the sheaves of $\Rb$-valued $n$-forms on
manifolds (or infinitesimally extended manifolds). 

If we define
$$\Lambda^n_{\text{cl}} :\equiv \ker(d : \Lambda^n \to \Lambda^{n + 1})$$
to be the closed $n$-form classifier, then we can reorganize the long exact sequence (\ref{eqn:long.exact.form.class})
into a series of short exact sequences of abelian groups
$$0 \to \Lambda^n_{\text{cl}} \to \Lambda^n \xto{d} \Lambda^{n+1}_{\text{cl}}
\to 0.$$
The first of these short exact sequences is
$$0 \to \flat \Rb \to \Rb \xto{d} \Lambda^1_{\text{cl}} \to 0$$
which we may extend into a long fiber sequence
\[
\begin{tikzcd}
{\flat \Rb} \ar[r, "(-)_{\flat}"] & \Rb \arrow[r, "d"] & \Lambda^1_{\text{cl}} \arrow[lld, out=-30, in=150] \\
\flat \B \Rb \arrow[r]    & \B \Rb     \arrow[r]    & \cdots
\end{tikzcd}
\]
This shows that $\Lambda^1_{\text{cl}}$ is the infinitesimal remainder of the
additive Lie group $\Rb$. Since $\Rb$ has contractible shape by definition, we
see that $\Rb \xto{d} \Lambda^1_{\text{cl}}$ is the universal $\infty$-cover of
$\Lambda^1_{\text{cl}}$. This gives us the following theorem, a form of de
Rham's theorem in smooth cohesion.
\begin{lem}\label{lem.form_classifier_contr}
The $n$-form classifiers $\Lambda^n$ are contractible.
\end{lem}
\begin{proof}
This follows immediately from the assumption that they are real vector spaces.
By Lemma 6.9 of \cite{Jaz:Good.Fibrations}, to show that $\Lambda^n$ is contractible
it suffices to give for every $\omega : \Lambda^n$ a path $\gamma : \Rb \to
\Lambda^n$ from $\omega$ to $0$. We can of course define
$$\gamma(t) :\equiv t\omega$$
which gives our desired contraction.
\end{proof}

\begin{thm}\label{thm:form.classifiers.contractible}
Let $\Lambda^n_{\text{cl}} :\equiv \ker(d : \Lambda^n \to \Lambda^{n + 1})$ be
the closed $n$-form classifier. Then
$$\shape \Lambda^n_{\text{cl}} = \flat \B^n \Rb.$$
\end{thm}
\begin{proof}
Since $\Lambda^1_{\text{cl}}$ is the infinitesimal remainder of $\Rb$, this
follows from \cref{thm.main_theorem}:
$$\shape \Lambda^1_{\text{cl}} = \flat \B \Rb.$$
We then proceed by induction. We have a short exact sequence of abelian groups
$$0 \to \Lambda^n_{\text{cl}} \to \Lambda^n \xto{d} \Lambda^{n+1}_{\text{cl}} \to 0.$$
We note that since $d : \Lambda^n \to  \Lambda^{n+1}_{\text{cl}}$ is an abelian
group homomorphism, all of its fibers are identifiable with the crisp type
$\Lambda^n_{\text{cl}}$ and therefore, by the ``good fibrations'' trick (Theorem
6.1 of \cite{Jaz:Good.Fibrations}), it is a $\shape$-fibration. Therefore, we get a
fiber sequence
$$\shape \Lambda^n_{\text{cl}} \to \shape \Lambda^n \to \shape \Lambda^{n+1}_{\text{cl}}.$$
Now, since $\Lambda^n$ is contractible by \cref{lem.form_classifier_contr}, we
see that
$$\Omega \shape \Lambda^{n+1}_{\text{cl}} = \shape \Lambda^n_{\text{cl}}$$
By inductive hypothesis, $\shape \Lambda^n_{\text{cl}} = \flat B^n \Rb$, so all
that remains is to show that $\shape \Lambda^{n+1}_{\text{cl}}$ is
$n$-connected. We will do this by showing that for any $u : \shape
\Lambda^{n+1}_{\text{cl}}$, the loop space $\Omega(\shape
\lambda^{n+1}_{\text{cl}}, u)$ is $(n-1)$-connected. By Corollary 9.12 of \cite{Shulman:Real.Cohesion}, the $\shape$-unit
$(-)^{\shape} : \Lambda^{n+1}_{\text{cl}} \to \shape \Lambda^{n+1}_{\text{cl}}$
is surjective, so there exists an $\omega : \Lambda^{n+1}_{\text{cl}}$ with $u =
\omega^{\shape}$. We then have a fiber sequence

$$\fib_{d}(\omega) \to \Lambda^{n} \xto{d} \Lambda^{n+1}_{\text{cl}}$$
which, since $\Lambda^{n}\xto{d}\Lambda^{n + 1}$ is a $\shape$-fibration descends to a fiber sequence
\begin{equation}\label{eqn.fiber_seq}
\shape\fib_{d}(\omega) \to \shape \Lambda^{n} \xto{\shape} \shape \Lambda^{n +
1}_{\text{cl}}.
\end{equation}
Since $d$ is surjective, there is a $\alpha : \Lambda^n$ with $d\alpha =
\omega$, and we may therefore contract $\shape \Lambda^{n}$onto
$\alpha^{\shape}$. This lets us equate the sequence (\ref{eqn.fiber_seq})
with the sequence
$$\Omega(\shape\Lambda^{n+1}_{\text{cl}}, \omega^{\shape}) \to \ast
\xto{\omega^{\shape}} \shape\Lambda^{n+1}_{\text{cl}}.$$
But as $d : \Lambda^n \to \Lambda^{n+1}_{\text{cl}}$ is an abelian group
homomophism, its fibers are all identifiable with its kernel
$\Lambda^{n}_{\text{cl}}$; this means that $\Omega(\shape
\Lambda^{n+1}_{\text{cl}}, u)$ is identifiable with $\shape
\Lambda^n_{\text{cl}}$, which by inductive hypothesis is $(n-1)$-connected.
\end{proof}

We may understand this theorem as a form of de Rham theorem in smooth cohesive
homotopy type theory. We may think of the unit $(-)^{\shape} :
\Lambda^n_{\text{cl}} \to \flat B^n \Rb$ as giving the
de Rham class of a closed $n$-form. That this map is the $\shape$-unit says that
this is the universal discrete cohomological invariant of closed $n$-forms.
Explicitly, if $E_{\bullet}$ is a loop spectrum, then
$H^k(\Lambda^n_{\text{cl}}; E_{\bullet}) :\equiv \trunc{\Lambda^n_{\text{cl}}
  \to E_k}_0$. Therefore, if the $E_k$ are discrete, then any cohomology class
$c : \Lambda^n_{\text{cl}} \to E_k$ factors through the de Rham class
$(-)^{\shape} : \Lambda^n_{\text{cl}} \to \flat \B^n \Rb$. In this sense, every
discrete cohomological invariant of closed $n$-forms is in fact an invariant of
their de Rham class in discrete real cohomology.

\subsection{Circle $k$-Gerbes with Connection}

We can now go about defining ordinary differential cohomology. We understand
ordinary differential cohomology as equipping integral cohomology with
differential form data. Hopkins and Singer define (Definition 2.4 of
\cite{Hopkins-Singer:Differential.Cocycles}) a \emph{differential cocycle} of degree $k + 1$
on $X$
to be a triple $(c, h, \omega)$ consisting of an underlying cocycle $c \in
Z^{k+1}(X, \Zb)$ in integral cohomology, a \emph{curvature} form $\omega \in \Lambda^{k+1}_{\text{cl}}(X)$, and a \emph{monodromy} term $h \in C^{k}(X,
\Rb)$ satisfying the equation $dh = \omega - c$.

We will follow their lead, at least in spirit. In true homotopy type theoretic
fashion, we will define the classifying types first and then derive the
cohomology theory by truncation.
\begin{defn}\label{def:classifier.connections}
We define the classifier $\B^k_{\nabla}U(1)$ of degree $(k + 1)$ classes in
ordinary differential cohomology to be the pullback:
\[\begin{tikzcd}
	{\B^k_{\nabla}U(1)} & {\Lambda^{k+1}_{\text{cl}}} \\
	{\B^{k+1} \Zb} & {\flat\B^{k+1}\Rb}
	\arrow[from=1-1, to=1-2, "F_{-}"]
	\arrow[from=1-1, to=2-1]
	\arrow[from=2-1, to=2-2]
	\arrow[from=1-2, to=2-2, "(-)^{\shape}"]
	\arrow["\lrcorner"{anchor=center, pos=0.125}, draw=none, from=1-1, to=2-2]
\end{tikzcd}\]
\end{defn}

Therefore, a cocycle $\tilde{c} : X \to \B^k_{\nabla}U(1)$ in differential cohomology
will consist of an underlying cocycle $c : X \to \B^{k+1}\Zb$, a curvature form $\omega : X
\to \Lambda^{k+1}_{\text{cl}}$, together with an identification $h : c = \omega$ in
$X \to \flat \B^{k+1}\Rb$. Since $h$ lands in types identifiable with $\Omega
\flat \B^{k+1} \Rb$, which equals $ \flat \B^k \Rb$, we may consider it as the monodromy term in
discrete real cohomology. We will now set about justifying this terminology.

We may note immediately from this definition that the map $\B^k_{\nabla} U(1)
\to \B^{k+1}\Zb$ (which we may think of as taking the underlying class in
ordinary cohomology) is the $\shape$-unit. This means that the underlying
cocycle is the universal discrete cohomological invariant of a differential cocycle.
\begin{lem}\label{lem:diff.coh.shape.naturality}
The pullback square 
\[\begin{tikzcd}
	{\B^k_{\nabla}U(1)} & {\Lambda^{k+1}_{\text{cl}}} \\
	{\B^{k+1} \Zb} & {\flat\B^{k+1}\Rb}
	\arrow[from=1-1, to=1-2]
	\arrow[from=1-1, to=2-1]
	\arrow[from=2-1, to=2-2]
	\arrow[from=1-2, to=2-2, "(-)^{\shape}"]
	\arrow["\lrcorner"{anchor=center, pos=0.125}, draw=none, from=1-1, to=2-2]
\end{tikzcd}\]
is a $\shape$-naturality square. That is, $\B^k_{\nabla} U(1) \to \B^{k+1}\Zb$
is a $\shape$-unit.
\end{lem}
\begin{proof}
Since $\B^{k}_{\nabla}U(1) \to \B^{k+1}\Zb$ is a map into a $\shape$-modal type,
to show that it is a $\shape$ unit it suffices to show that it is
$\shape$-connected. Since we have a pullback square, the fibers of
$\B^k_{\nabla}U(1) \to \B^{k+1}\Zb$ are the same as those of
$(-)^{\shape} : \Lambda^{k+1}_{\text{cl}} \to \flat\B^{k+1}\Rb$. But as this map is a
$\shape$-unit, its fibers are $\shape$-connected.
\end{proof}

The reason for our change of index --- defining $\B^k_{\nabla} U(1)$ to represent
degree $(k+1)$ classes --- is because we would like to think of
$\B^k_{\nabla}U(1)$ as more directly classifying connections on $k$-gerbes with
band $U(1)$. To reify this idea, let's give the map $\B^k_{\nabla} U(1) \to
\B^k U(1)$ which we think of as taking the underlying $k$-gerbe.

\begin{con}
We construct a map $\B^k_{\nabla} U(1) \to \B^kU(1)$ which makes the following
triangle commute:
\[\begin{tikzcd}
	{\B^k_{\nabla}U(1)} && {\B^k U(1)} \\
	& {\B^{k+1}\Zb}
	\arrow[from=1-1, to=1-3]
	\arrow[from=1-1, to=2-2, "{(-)^{\shape}}"']
	\arrow["{(-)^{\shape}}", from=1-3, to=2-2]
\end{tikzcd}\]
\end{con}
\begin{proof}[Construction]
 Since $\Rb \to U(1)$ is the universal $\infty$-cover of $U(1)$, by
 \cref{cor:inf_remainder_infty_cover}, $U(1)$ has the same infinitesimal
 remainder as $\Rb$, which is $\Lambda^1_{\text{cl}}$.
  Therefore, by modal fracture \cref{thm.main_theorem}, we have a pullback square
  \[\begin{tikzcd}
	{\B^kU(1)} & {\B^k \Lambda^1_{\text{cl}}} \\
	{\B^{k+1}\Zb} & {\flat \B^{k+1}\Rb}
	\arrow["{(-)^{\shape}}"', from=1-1, to=2-1]
	\arrow[from=1-1, to=1-2]
	\arrow[from=1-2, to=2-2]
	\arrow[from=2-1, to=2-2]
	\arrow["\lrcorner"{anchor=center, pos=0.125}, draw=none, from=1-1, to=2-2]
\end{tikzcd}\]

Now, since we have a series of short exact sequences
$$0 \to \Lambda^n_{\text{cl}} \to \Lambda^{n} \xto{d}
\Lambda^{n+1}_{\text{cl}} \to 0$$
we have long fiber sequences
\[
\begin{tikzcd}
\Lambda^n_{\text{cl}} \ar[r] & \Lambda^n \arrow[r, "d"] & \Lambda^{n + 1}_{\text{cl}} \arrow[lld, out=-30, in=150] \\
\B \Lambda^n_{\text{cl}} \arrow[r]    & \B \Lambda^n     \arrow[r]    & \cdots
\end{tikzcd}
\]
for each $n$. In particular, we have maps $\B^{n}\Lambda^{m+1}_{\text{cl}} \to
\B^{n+1}\Lambda^{m}_{\text{cl}}$ for all $n$ and $m$. Taking repeated pullbacks
along these maps gives us a diagram
\begin{equation}\label{eqn.big_diff_coh_diagram}
\begin{tikzcd}
	{\B^k_{\nabla}U(1)} & {\Lambda^{k+1}_{\text{cl}}} \\
	\bullet & {\B \Lambda^{k}_{\text{cl}}} \\
	\vdots & \vdots \\
	\bullet & {\B^{k-1}\Lambda^2_{\text{cl}}} \\
	{\B^kU(1)} & {\B^k \Lambda^1_{\text{cl}}} \\
	{\B^{k+1}\Zb} & {\flat \B^{k+1}\Rb}
	\arrow[from=5-1, to=6-1]
	\arrow[from=5-1, to=5-2]
	\arrow[from=5-2, to=6-2]
	\arrow[from=6-1, to=6-2]
	\arrow["\lrcorner"{anchor=center, pos=0.125}, draw=none, from=5-1, to=6-2]
	\arrow[from=4-1, to=5-1]
	\arrow[from=4-1, to=4-2]
	\arrow[from=4-2, to=5-2]
	\arrow["\lrcorner"{anchor=center, pos=0.125}, draw=none, from=4-1, to=5-2]
	\arrow[from=1-1, to=2-1]
	\arrow[from=1-1, to=1-2]
	\arrow[from=1-2, to=2-2]
	\arrow[from=2-1, to=2-2]
	\arrow["\lrcorner"{anchor=center, pos=0.125}, draw=none, from=1-1, to=2-2]
	\arrow[from=2-2, to=3-2]
	\arrow[from=3-2, to=4-2]
	\arrow[from=2-1, to=3-1]
	\arrow[shift right=5,bend right = 30, dashed, from=1-1, to=5-1]
	\arrow[from=3-1, to=4-1]
\end{tikzcd}
\end{equation}
The dashed composite in this diagram is what we were seeking to construct.
\end{proof}

\begin{rmk}
  Diagram \ref{eqn.big_diff_coh_diagram} shows us that the following square is a
  pullback:
\[\begin{tikzcd}
	{\B^k_{\nabla}U(1)} & {\Lambda^{k+1}_{\text{cl}}} \\
	{\B^{k} U(1)} & {\B^{k} \Lambda^1_{\text{cl}}}
	\arrow[from=1-1, to=1-2]
	\arrow[from=1-1, to=2-1]
	\arrow[from=2-1, to=2-2]
	\arrow[from=1-2, to=2-2, "(-)^{\shape}"]
	\arrow["\lrcorner"{anchor=center, pos=0.125}, draw=none, from=1-1, to=2-2]
\end{tikzcd}\]
If we note that $\B^{k} \Lambda^1_{\text{cl}}$ is $\B^k \lie{u}(1)$, we get an
alternate definition of $\B^k_{\nabla}U(1)$ by this pullback. This shows that
our definition agrees with Schreiber's Definition 4.4.93 in \cite{Schreiber:Differential.Cohomology}.
\end{rmk}

We can now see that the map $\B^k_{\nabla}U(1) \to \Lambda^{k+1}_{\text{cl}}$
takes the
the \emph{curvature $(k+1)$-form}. We can justify this by showing that the fiber
of this map is $\flat\B^k U(1)$; in other words, a circle $k$-gerbe with
connection is flat if and only if its curvature vanishes.
\begin{lem}\label{lem.curvature}
The map $F_{(-)} : \B^k_{\nabla}U(1) \to \Lambda^{k+1}_{\text{cl}}$ has fiber $\flat \B^k
U(1)$. Since this map gives an obstruction to flatness, we refer to it as the
\emph{curvature} $(k+1)$-form.
\end{lem}
\begin{proof}
By considering the top part from the diagram (\ref{eqn.big_diff_coh_diagram}), we
find a pullback square
\[\begin{tikzcd}
	{\B^k_{\nabla}U(1)} & {\Lambda^{k+1}_{\text{cl}}} \\
	{\B^kU(1)} & {\B^{k}\Lambda^1_{\text{cl}}} \\
	{}
	\arrow[from=1-1, to=1-2]
	\arrow[from=1-2, to=2-2]
	\arrow[from=1-1, to=2-1]
	\arrow[from=2-1, to=2-2]
	\arrow["\lrcorner"{anchor=center, pos=0.125}, draw=none, from=1-1, to=2-2]
\end{tikzcd}\]
For this reason, we get an equivalence on fibers:
\[\begin{tikzcd}
	\bullet & {\B^k_{\nabla}U(1)} & {\Lambda^{k+1}_{\text{cl}}} \\
	{\flat\B^kU(1)} & {\B^kU(1)} & {\B^{k}\Lambda^1_{\text{cl}}} \\
	& {}
	\arrow[from=1-2, to=1-3]
	\arrow[from=1-3, to=2-3]
	\arrow[from=1-2, to=2-2]
	\arrow[from=2-2, to=2-3]
	\arrow["\lrcorner"{anchor=center, pos=0.125}, draw=none, from=1-2, to=2-3]
	\arrow[from=2-1, to=2-2]
	\arrow["{\rotatebox[origin=c]{90}{$\sim$}}"', from=1-1, to=2-1]
	\arrow[from=1-1, to=1-2] 
\end{tikzcd}\qedhere \]
\end{proof}

As a corollary, we may characterize the curvature $F_{(-)} : \B^k_{\nabla}U(1) \to
\Lambda^{k+1}_{\text{cl}}$ modally.
\begin{cor}
The curvature $F_{(-)} : \B^k_{\nabla}U(1) \to
\Lambda^{k+1}_{\text{cl}}$ is a unit for the $(k-1)$-truncation modality. In
particular,
$$\trunc{\B^k_{\nabla}U(1)}_j = \Lambda^{k+1}_{\text{cl}}$$
for any $0 \leq j < k$.
\end{cor}
\begin{proof}
As $\Lambda^{k+1}_{\text{cl}}$ is $0$-truncated and so $(k-1)$ truncated, it will suffice to show that
$F_{(-)}$ is $(k-1)$-connected. But by \cref{lem.curvature}, the fiber of $F_{(-)}$
over any point $\omega : \Lambda^{k+1}_{\text{cl}}$ is identifiable with $\flat
B^{k}U(1)$, which is $(k-1)$-connected. 
\end{proof}

Though our notation may have suggested that the $\B^k_{\nabla}U(1)$ form a loop
spectrum, they do not. Indeed, $\Omega \B^k_{\nabla}U(1) = \flat \B^{k-1}U(1)$,
as can be seen by taking loops of the pullback square defining
$\B^k_{\nabla}U(1)$ and noting that $\Lambda^{n+1}_{\text{cl}}$ is a set
($0$-type). In total,
\[
\pi_{\ast}\B^k_{\nabla}U(1) = \begin{cases} \Lambda^{k+1}_{\text{cl}}
  &\mbox{if $\ast = 0$} \\
  \flat U(1) &\mbox{if $\ast = k$} \\
  0 &\mbox{otherwise.} \end{cases}
\]

Nevertheless, each $\B^k_{\nabla}U(1)$ is an infinite loop space in
its own right.

\begin{defn}
For $n,\, k \geq 0$, define $\B^n \B^k_{\nabla}U(1)$ to be the following pullback:
\[\begin{tikzcd}
	{\B^n\B^k_{\nabla}U(1)} & {\B^n\Lambda^{k+1}_{\text{cl}}} \\
	{\B^{n + k + 1} \Zb} & {\flat \B^{n + k + 1} \Rb} \\
	{}
	\arrow[from=1-1, to=1-2]
	\arrow[from=1-2, to=2-2]
	\arrow[from=1-1, to=2-1]
	\arrow[from=2-1, to=2-2]
	\arrow["\lrcorner"{anchor=center, pos=0.125}, draw=none, from=1-1, to=2-2]
\end{tikzcd}\]
\end{defn}

It is immediate from this definition and the commutation of taking loops with
taking pullbacks that $\Omega\B^{n+1}\B^k_{\nabla}U(1) = \B^n\B^k_{\nabla}U(1)$.
We have already seen these deloopings before in Diagram (\ref{eqn.big_diff_coh_diagram}):
\[\begin{tikzcd}
	{\B^k_{\nabla}U(1)} & {\Lambda^{k+1}_{\text{cl}}} \\
	{\B\B^{k-1}_{\nabla}U(1)} & {\B \Lambda^{k}_{\text{cl}}} \\
	\vdots & \vdots \\
	{\B^{k-1}\B_{\nabla}U(1)} & {\B^{k-1}\Lambda^2_{\text{cl}}} \\
	{\B^kU(1)} & {\B^k \Lambda^1_{\text{cl}}} \\
	{\B^{k+1}\Zb} & {\flat \B^{k+1}\Rb}
	\arrow[from=5-1, to=6-1]
	\arrow[from=5-1, to=5-2]
	\arrow[from=5-2, to=6-2]
	\arrow[from=6-1, to=6-2]
	\arrow["\lrcorner"{anchor=center, pos=0.125}, draw=none, from=5-1, to=6-2]
	\arrow[from=4-1, to=5-1]
	\arrow[from=4-1, to=4-2]
	\arrow[from=4-2, to=5-2]
	\arrow["\lrcorner"{anchor=center, pos=0.125}, draw=none, from=4-1, to=5-2]
	\arrow[from=1-1, to=2-1]
	\arrow[from=1-1, to=1-2]
	\arrow[from=1-2, to=2-2]
	\arrow[from=2-1, to=2-2]
	\arrow["\lrcorner"{anchor=center, pos=0.125}, draw=none, from=1-1, to=2-2]
	\arrow[from=2-2, to=3-2]
	\arrow[from=3-2, to=4-2]
	\arrow[from=2-1, to=3-1]
	\arrow[from=3-1, to=4-1]
\end{tikzcd}\]
These maps along the left hand side give us maps of loop spectra
$$\B^{\bullet} \B^k_{\nabla} U(1) \to \Sigma \B^{\bullet}\B^{k-1}_{\nabla} U(1)$$
We will see in \cref{sec:subtleties} that for $\bullet = 0$, these maps give
obstructions to de Rham's theorem for general types.

Each $\B^k_{\nabla} U(1)$ is a higher group itself. We may
therefore ask: what is it's infinitesimal remainder?
\begin{lem}\label{lem:inf.remainder.curvature}
The infinitesimal remainder of $\B^k_{\nabla}U(1)$ is the curvature $F_{(-)} :
\B^k_{\nabla}U(1) \to \Lambda^{k+1}_{\text{cl}}$.
\end{lem}
\begin{proof}
Consider the following diagram:
\[\begin{tikzcd}
	{\flat_{\text{dR}}\B\B^k_{\nabla}U(1)} & {\Lambda^{k+1}_{\text{cl}}} \\
	\ast & \ast & {\flat\B\B^k_{\nabla}U(1)} & {\flat\B\Lambda^{k+1}_{\text{cl}}} \\
	&& {\B^{k+2}\Zb} & {\flat\B^{k+2}\Rb} & {\B\B^k_{\nabla}U(1)} & {\B\Lambda^{k+1}_{\text{cl}}} \\
	&&&& {\B^{k+2}\Zb} & {\flat\B^{k+2}\Rb}
	\arrow[from=3-5, to=3-6]
	\arrow[from=3-6, to=4-6]
	\arrow[from=4-5, to=4-6]
	\arrow[from=3-4, to=4-6, equals]
	\arrow[from=2-4, to=3-6]
	\arrow[from=3-3, to=4-5, equals]
	\arrow[from=2-4, to=3-4]
	\arrow[from=2-3, to=2-4]
	\arrow[from=3-3, to=3-4]
	\arrow[from=2-2, to=3-4]
	\arrow[from=2-1, to=3-3]
	\arrow[from=1-1, to=2-1]
	\arrow[from=1-2, to=2-2]
	\arrow["\sim", from=1-1, to=1-2]
	\arrow[from=2-1, to=2-2]
	\arrow[from=1-2, to=2-4]
	\arrow[from=1-1, to=2-3, crossing over]
	\arrow[from=2-3, to=3-3, crossing over]
	\arrow[from=2-3, to=3-5, crossing over]
	\arrow[from=3-5, to=4-5, crossing over]
	\arrow["\lrcorner"{anchor=center, pos=0.125}, draw=none, from=2-3, to=3-4]
	\arrow["\lrcorner"{anchor=center, pos=0.125}, draw=none, from=3-5, to=4-6]
	\arrow["\lrcorner"{anchor=center, pos=0.125}, draw=none, from=1-1, to=2-2]
\end{tikzcd}\]
The diagonal sequences in this diagram are fiber sequences of the
$\flat$-counits which define the infinitesimal remainders. Now, since $\flat \B
\Lambda^{k+1}_{\text{cl}} = \ast$ since the form classifiers are infinitesimal, we find that the fiber of the
$\flat$-counit $(-)_{\flat} : \flat \B \Lambda^{k+1}_{\text{cl}} \to \B
\Lambda^{k+1}_{\text{cl}}$ is $\Omega \B \Lambda^{k+1}_{\text{cl}}$, which is
$\Lambda^{k+1}_{\text{cl}}$.

Now, the frontmost square is a crisp pullback and $\flat$ is left exact, so the
middle square is also a pullback. Then, since the diagonal sequences are fiber
sequences, the back square is also a pullback. But this shows that the
infinitesimal remainder of $\B^{k}_{\nabla}U(1)$ is $\Lambda^{k+1}_{\text{cl}}$.

If we take one more fiber, we can continue the diagram to give us the following diagram:
\[\begin{tikzcd}
	{\B^{k}_{\nabla}U(1)} & {\Lambda^{k+1}_{\text{cl}}} \\
	{\B^{k+1}\Zb} & {\flat\B^{k+1}\Rb} & {\flat_{\text{dR}}\B\B^k_{\nabla}U(1)} & {\Lambda^{k+1}_{\text{cl}}} \\
	&& \ast & \ast & {\flat\B\B^k_{\nabla}U(1)} & {\flat\B\Lambda^{k+1}_{\text{cl}}} \\
	&&&& {\B^{k+2}\Zb} & {\flat\B^{k+2}\Rb} \\
	\\
	&& {}
	\arrow[from=2-4, to=3-4]
	\arrow[from=2-3, to=2-4, "\sim" near start]
	\arrow[from=3-3, to=3-4]
	\arrow[from=3-3, to=4-5]
	\arrow[from=4-5, to=4-6]
	\arrow[from=3-4, to=4-6]
	\arrow[from=3-6, to=4-6]
	\arrow[from=2-4, to=3-6]
	\arrow[from=3-5, to=3-6]
	\arrow[from=1-2, to=2-4, equals]
	\arrow[from=2-2, to=3-4]
	\arrow[from=2-1, to=3-3]
	\arrow[from=1-1, to=2-1]
	\arrow[from=1-2, to=2-2]
	\arrow[from=1-1, to=1-2, "F_{(-)}"]
	\arrow[from=2-1, to=2-2]
	\arrow[from=1-1, to=2-3, crossing over, "\theta" near end]
	\arrow[from=2-3, to=3-3, crossing over]
	\arrow[from=2-3, to=3-5, crossing over]
	\arrow[from=3-5, to=4-5, crossing over]
	\arrow["\lrcorner"{anchor=center, pos=0.125}, draw=none, from=1-1, to=2-2]
	\arrow["\lrcorner"{anchor=center, pos=0.125}, draw=none, from=3-5, to=4-6]
	\arrow["\lrcorner"{anchor=center, pos=0.125}, draw=none, from=2-3, to=3-4]
\end{tikzcd}\]
This shows that the infinitesimal remainder $\theta$ is equal, modulo our
constructed equivalence, to the curvature $F_{(-)}$.
\end{proof}

Now that we know the infinitesimal remainder of $\B^{k}_{\nabla}U(1)$, we are
almost ready to understand its modal fracture hexagon. But first, we must
understand its universal $\infty$-cover. We will show that the universal
$\infty$-cover of $\B^{k}_{\nabla}U(1)$ is an analogous type $\B^{k}_{\nabla}\Rb$.
\begin{defn}
For $n, k \geq 0$, define $\B^n\B^k_{\nabla} \Rb$ to be the universal
$\infty$-cover of $\B^n\Lambda^{k+1}_{\text{cl}}$:
\[\begin{tikzcd}
	{\B^n\B^k_{\nabla}\Rb} & {\B^n\Lambda^{k+1}_{\text{cl}}} \\
	\ast & {\flat \B^{n + k + 1}\Rb}
	\arrow[from=1-1, to=2-1]
	\arrow[from=2-1, to=2-2]
	\arrow[from=1-1, to=1-2, "F_{(-)}"]
	\arrow[from=1-2, to=2-2]
	\arrow["\lrcorner"{anchor=center, pos=0.125}, draw=none, from=1-1, to=2-2]
\end{tikzcd}\]
We refer to the cohomology theories $\B^k_{\nabla} \Rb$ as \emph{pure
  differential cohomology}.
\end{defn}

Just as we may think of $\B^k_{\nabla}U(1)$ as classifying circle $k$-gerbes
with connection, we may think of $\B^k_{\nabla}\Rb$ as classifying affine
$k$-gerbes with connection. We can now show that $\B^k_{\nabla} \Rb$ is the universal $\infty$-cover of
$\B^k_{\nabla}U(1)$.
\begin{prop}\label{prop:universal.infty.cover.diff.coh}
The map $(\omega,\, p) \mapsto (\pt_{\B^{k+1}\Zb},\, \omega,\, \lambda\_\,.\,p) :
\B^k_{\nabla}\Rb \to \B^k_{\nabla} U(1)$ is the universal $\infty$-cover of $\B^k_{\nabla}U(1)$.
\end{prop}
\begin{proof}
Consider the following cube: 
\[\begin{tikzcd}
	{\B^k_{\nabla}\Rb} & {\Lambda^{k+1}_{\text{cl}}} \\
	\ast & {\flat \B^{k + 1}\Rb} & {\B^k_{\nabla} U(1)} & {\Lambda^{k+1}_{\text{cl}}} \\
	&& {\B^{k + 1}\Zb} & {\flat \B^{k + 1}\Rb}
	\arrow[from=1-1, to=2-1]
	\arrow[from=2-1, to=2-2]
	\arrow[from=1-1, to=1-2]
	\arrow[from=1-2, to=2-2]
	\arrow[from=1-2, to=2-4, equals]
	\arrow[from=2-1, to=3-3]
	\arrow[from=2-2, to=3-4, equals]
	\arrow[from=2-3, to=2-4]
	\arrow[from=2-4, to=3-4]
	\arrow[from=3-3, to=3-4]
	\arrow[from=1-1, to=2-3, crossing over, dashed]
	\arrow[from=2-3, to=3-3, crossing over]
	\arrow["\lrcorner"{anchor=center, pos=0.125}, draw=none, from=2-3, to=3-4]
	\arrow["\lrcorner"{anchor=center, pos=0.125}, draw=none, from=1-1, to=2-2]
\end{tikzcd}\]
In this cube, the front and back spaces are pullbacks by definition, and the
right face is a pullback because its top and bottom sides are identities. Therefore, the left face is a pullback. Since
$\B^k_{\nabla}U(1) \to \B^{k+1} \Zb$ is a $\shape$-unit by
\cref{lem:diff.coh.shape.naturality}, this shows that the dashed map is the
fiber of a $\shape$-unit, and therefore the universal $\infty$-cover.
\end{proof}
\begin{rmk}
  The fiber sequence
  $$\B^k_{\nabla} \Rb \to \B^k_{\nabla} U(1) \to \B^{k + 1} \Zb$$
  expresses the informal identity
  $$\mbox{ordinary differential cohomology} = \mbox{pure differential
    cohomology} + \mbox{ordinary cohomology}.$$
\end{rmk}

We are now ready to assemble what we have learned into the modal facture hexagon
of $\B^k_{\nabla}U(1)$:
\begin{equation}\label{eqn:differential.hexagon}
    \begin{tikzcd}
      & \B^k_{\nabla}\Rb \ar[dr, "\pi"] \ar[rr] & & \Lambda^{k+1}_{\text{cl}} 
      \ar[dr,  "(-)^{\shape}"] \\
      \flat \B^k\Rb \ar[ur, "(-)_{\flat}"] \ar[dr] & &
       \B^k_{\nabla}U(1) \ar[dr, "(-)^{\shape}"'] \ar[ur, "F_{(-)}"] & & \flat\B^{k+1}\Rb \\
      & \flat \B^k U(1) \ar[rr, "\beta"'] \ar[ur, "(-)_{\flat}"'] & & \B^{k+1}\Zb \ar[ur] &
    \end{tikzcd}
    \end{equation}

  \subsection{Descending to Cohomology and the Character Diagram}\label{sec:subtleties}

  In this section, we will discuss how the modal fracture hexagon (\ref{eqn:differential.hexagon}) descends to cohomology. In general, if $E_{\bullet}$ is a loop spectrum, then the we may define the cohomology groups of a type valued in $E_{\bullet}$ to be the $0$-truncated types of maps:
  $$H^k(X; E_{\bullet}) :\equiv \trunc{X \to E_k}_{0}.$$

  However, these abelian groups are not discrete --- externally, they are
  (possible non-constant) sheaves of abelian groups. We will want the discrete
  (externally, constant) invariants. 
  With this in mind, we make the following definitions.
 \begin{defn}
  Let $X$ be a crisp type. We then make the following definitions:
  \begin{align*}
    H^n(X; \Zb) &:\equiv \trunc{\flat(X \to \B^n \Zb)}_0 \\
    H^n(X; \flat \Rb) &:\equiv \trunc{\flat(X \to \flat \B^n \Rb)}_0 \\
    H^n(X; \flat U(1)) &:\equiv \trunc{\flat( X \to \flat \B^n U(1) )}_0 \\
    H_{\nabla}^{n,k}(X; U(1)) &:\equiv \trunc{\flat( X \to \B^n\B^k_{\nabla}U(1) )}_0\\
    H_{\nabla}^{n,k}(X; \Rb) &:\equiv \trunc{\flat( X \to \B^n\B^k \Rb )}_0 \\
    \Lambda^k(X) &:\equiv \flat( X \to \Lambda^k ) \\
    \Lambda^k_{\text{cl}}(X) &:\equiv \flat( X \to \Lambda^k_{\text{cl}} ).
  \end{align*}
 \end{defn}

 \begin{rmk}
In full cohesion, it would be better to work with codiscrete cohomology groups,
rather than discrete cohomology groups. This way the definition could be given
for all types and not just crisp ones. But we will continue to use discrete
groups so that we do not need to work with the codiscrete modality $\sharp$ in this paper.
 \end{rmk}

We note that with these definitions we may reduce the calculation of
ordinary differential cohomology for discrete and homotopically contractible types.
\begin{prop}\label{prop:calc.diff.coh.disc.contr}
Let $X$ be a crisp type and let $k \geq 1$.
\begin{enumerate}
\item If $X$ is discrete (that is, $X = \shape X$), then $H^{n,k}_{\nabla}(X;
  U(1)) = H^{n + k}(X; \flat U(1))$.
\item If $X$ is homotopically contractible (that is, $\shape X = \ast$), then $H^{n,k}_{\nabla}(X; U(1)) = H^n(X; \Lambda^{k
  +1}_{\text{cl}})$.
\end{enumerate}
We may make similar calculations for pure differential cohomology:
\begin{enumerate}
  \item If $X$ is discrete, then $H^{n, k}(X; \Rb) = H^{n + k}_{\nabla}(X; \flat \Rb)$.
  \item If $X$ is homotopically contractible, then $H^{n, k}_{\nabla}(X; \Rb) =
    H^n(X;\Lambda^{k + 1}_{\text{cl}})$.
\end{enumerate}
\end{prop}
\begin{proof}
We will only prove the identities for ordinary differential cohomology; the
proofs for pure differential cohomology are identical. We take advantage of the adjointness between $\shape$ and $\flat$.
\begin{enumerate}
\item Suppose that $X$ is discrete. Then
  \begin{align*}
    H^{n,k}_{\nabla}(X; U(1)) &= \trunc{\flat(X \to \B^n\B^k_{\nabla}U(1))}_0 \\
&= \trunc{\flat(\shape X \to \B^n\B^k_{\nabla}U(1))}_0 \\
&= \trunc{\flat(X \to \flat \B^n\B^k_{\nabla}U(1))}_0 \\
&= \trunc{\flat(X \to \flat \B^{n + k}U(1))}_0 \\
&= H^{n + k}(X; \flat U(1))
  \end{align*}
\item Suppose that $X$ is homotopically contractible, and let $i : \B^{n + k +
    1}\Zb \to \flat \B^{n + k + 1} \Rb$ denote the (delooping of) the inclusion. Then 
  \begin{align*}
    H^{n,k}_{\nabla}(X; U(1)) &= \trunc{\flat(X \to \B^n\B^k_{\nabla}U(1))} \\
                              &= \flat\trunc{
                                \begin{aligned}
                                  (\omega : X \to \B^n\Lambda^{k+1}_{\text{cl}}) &\times (c : X \to \B^{n + k + 1}\Zb) \\
                                  &\times (h : ic = \omega^{\shape})
                                \end{aligned}
                                }_0\\
                              &= \flat\trunc{
                                \begin{aligned}
                                  (\omega : X \to \B^n\Lambda^{k+1}_{\text{cl}}) &\times (c : \B^{n + k + 1}\Zb) \\
                                  &\times (h : (x : X) \to ic = \omega(x)^{\shape})
                                \end{aligned}
                                }_0\\
    \intertext{Since $X$ is homotopically contractible, we have an equivalence $e : (X \to \flat \B^{n + k + 1} \Rb) \simeq \flat \B^{n + k + 1}$. We therefore have $e((-)^{\shape} \circ \omega) : \flat \B^{n + k + 1}$ and for all $x : X$ a witness $\omega(x)^{\shape} = e((-)^{\shape} \circ \omega)$. We may therefore continue:}
                              &= \flat\trunc{
                                \begin{aligned}
                                  (\omega : X \to \B^n\Lambda^{k+1}_{\text{cl}}) &\times (c : \B^{n + k + 1}\Zb) \\
                                  &\times (h : (x : X) \to ic = e((-)^{\shape} \circ \omega))
                                \end{aligned}
                                }_0\\
                              &= \flat\trunc{
                                \begin{aligned}
                                  (\omega : X \to \B^n\Lambda^{k+1}_{\text{cl}}) &\times (c : \B^{n + k + 1}\Zb) \\
                                  &\times (ic = e((-)^{\shape} \circ \omega))
                                \end{aligned}
                                }_0\\
    \intertext{Now, both $\B^{n + k + 1} \Zb$ and $(ic = e((-)^{\shape} \circ \omega))$ are $0$-connected. The latter because it is identifiable with $\Omega \flat \B^{n + k + 1} \Rb$, which is $\flat \B^{n + k} \Rb$ and so $0$-connected for $k \geq 1$ and any $n$. We may therefore continue:}
                              &= \flat \trunc{X \to \B^n\Lambda^{k + 1}_{\text{cl}}}\\
    &= H^n(X; \Lambda^{k + 1}_{\text{cl}}). \qedhere
  \end{align*}
\end{enumerate}
\end{proof}

\begin{rmk}
  We note here that since every type $X$ lives in the center of a fiber sequence
  $$\inftycover{X} \to X \to \shape X$$
  between a homotopically contractible type and a discrete type, we get a Serre
  spectral sequence converging to the $k^{\text{th}}$ ordinary differential
  cohomology of $X$ with $E_2$ page depending on it's $\flat U(1)$ cohomology
  and the $\Lambda^{k + 1}_{\text{cl}}$ valued cohomology of it's universal $\infty$-cover.
\end{rmk}

 Now, since the top, bottom, and diagonal sequences in the modal fracture
 hexagon (\ref{eqn:differential.hexagon}) of $\B^k_{\nabla}U(1)$ are fiber
 sequences, when we take $\flat$ and $0$-truncations we will get long exact sequences. With
 the above definitions, we get the following diagram:
\begin{equation}\label{eqn:differential.hexagon.coh}
    \begin{tikzcd}
      & H^{0,k}_{\nabla}(X; \Rb) \ar[dr] \ar[rr] & & \Lambda^{k+1}_{\text{cl}}(X)
      \ar[dr] \\
      H^k(X;\flat \Rb) \ar[ur] \ar[dr] & &
       H^{0,k}_{\nabla}(X;U(1)) \ar[dr] \ar[ur] & & H^{k+1}(X;\flat \Rb) \\
      & H^{k}(X;\flat U(1)) \ar[rr, "\beta"'] \ar[ur] & & H^{k+1}(X; \Zb)\ar[ur] &
    \end{tikzcd}
    \end{equation}
in which the top and bottom sequences are long exact, and the diagonal sequences
are exact in the middle. This looks
very much like the character diagram for ordinary differential  cohomology \cite{Simons-Sullivan:Axiomatic.Ordinary.Differential.Cohomology} except for
two differences:
\begin{enumerate}
  \item Where we have the pure cohomology $H^{0,k}_{\nabla}(X;\Rb)$, one would normally find
    $\Lambda^{k}(X)/\im(d)$, the abelian group which fits into an exact sequence
$$\Lambda^{k-1}(X) \xto{d} \Lambda^{k}(X) \to \Lambda^{k}(X)/\im(d) \to 0.$$
\item Where we have $H^{k+1}(X; \flat \Rb)$, which is ordinary (discrete)
  cohomology with real coefficients, one would normally find the de Rham
  cohomology $H^{k + 1}_{\text{dR}}(X)$. The de Rham cohomology is defined as
  closed forms mod exact forms, and so $H^{k + 1}_{\text{dR}}(X)$ is the abelian group
  fitting the following exact sequence:
 $$\Lambda^{k}(X) \xto{d} \Lambda^{k + 1}_{\text{cl}}(X) \to H^{k + 1}_{\text{dR}}(X) \to 0.$$
\end{enumerate}
Both of these discrepancies are instances of de Rham's theorem that the de Rham
cohomology of forms is the (discrete) ordinary cohomology with real
coefficients. Classically and externally, this holds for smooth manifolds. We
note that de Rham's theorem cannot hold for all types for rather trivial
reasons: the form classifiers are sets, and so $\Lambda^k(X)$ depends only on
the set trunctation of $X$ whereas $H^k(X; \flat \Rb)$ can depend on the
$k$-truncation of $X$.
\begin{prop}\label{prop:de.Rham.fails}
The de Rham theorem does not hold for the delooping $\flat \B \Rb$ of the
discrete additive
group of real numbers. Explicitly,
\begin{align*}
  H^{1}_{\text{dR}}(\flat \B \Rb) &= 0 \\
  H^1(\flat \B \Rb; \flat \Rb) &= \flat\text{Hom}(\flat \Rb, \flat \Rb) \neq 0
\end{align*}
\end{prop}
\begin{proof}
Since $\flat \B \Rb$ is 0-connected and the form classifiers are sets, every map $\flat \B \Rb \to
\Lambda^{k}$ is constant for all $k$. Therefore,
\[
H^1_{\text{dR}}(\flat \B \Rb) = \Lambda^1_{\text{cl}}(\flat \B \Rb)/
\Lambda^0(\flat \B \Rb) = 0
\]
On the other hand, $H^{1}(\flat \B \Rb; \flat \Rb) = \trunc{\flat(\flat \B \Rb
  \to \flat \B \Rb)}_0$ is the set of group homomorphisms from $\flat \Rb$ to
itself (modulo conjugacy, which makes no difference). The identity is not
conjugate to $0$, and so this group is not trivial.
\end{proof}

We can, however, make explicit the obstruction
to de Rham's theorem lying in the first (cohomological) degree pure differential
cohomology groups $H^{1, k}_{\nabla}(X;\Rb)$.\footnote{A similar obstruction was also found independently through more classical means in \cite{minichiello2024diffeological}.} We begin first by trying to construct an exact sequence
$$\Lambda^k(X) \xto{d} \Lambda^{k + 1}(X) \to H^{0, k}_{\nabla}(X; \Rb) \to 0.$$
Recall Diagram \ref{eqn.big_diff_coh_diagram}. There is a similar diagram for
pure differential cohomology:

\[\begin{tikzcd}
	{\B^k_{\nabla}\Rb} & {\Lambda^{k+1}_{\text{cl}}} \\
	{\B\B^{k-1}_{\nabla}\Rb} & {\B \Lambda^{k}_{\text{cl}}} \\
	\vdots & \vdots \\
	{\B^{k-1}\B_{\nabla}\Rb} & {\B^{k-1}\Lambda^2_{\text{cl}}} \\ 
	{\B^k\Rb} & {\B^k \Lambda^1_{\text{cl}}} \\
	{\ast} & {\flat \B^{k+1}\Rb}
	\arrow[from=5-1, to=6-1]
	\arrow[from=5-1, to=5-2]
	\arrow[from=5-2, to=6-2]
	\arrow[from=6-1, to=6-2]
	\arrow["\lrcorner"{anchor=center, pos=0.125}, draw=none, from=5-1, to=6-2]
	\arrow[from=4-1, to=5-1]
	\arrow[from=4-1, to=4-2]
	\arrow[from=4-2, to=5-2]
	\arrow["\lrcorner"{anchor=center, pos=0.125}, draw=none, from=4-1, to=5-2]
	\arrow[from=1-1, to=2-1]
	\arrow[from=1-1, to=1-2]
	\arrow[from=1-2, to=2-2]
	\arrow[from=2-1, to=2-2]
	\arrow["\lrcorner"{anchor=center, pos=0.125}, draw=none, from=1-1, to=2-2]
	\arrow[from=2-2, to=3-2]
	\arrow[from=3-2, to=4-2]
	\arrow[from=2-1, to=3-1]
	\arrow[from=3-1, to=4-1]
\end{tikzcd}\]

If we focus at the top, we see that we have a pullback square which induces an
equivalence on fibers:
\begin{equation}\label{eqn:obstruction.diagram}
\begin{tikzcd}
	{\bullet} & {\Lambda^k} \\
	{\B^k_{\nabla} \Rb} & {\Lambda^{k + 1}_{\text{cl}} }\\
	{\B\B^{k-1}_{\nabla} \Rb} & {\B\Lambda^k_{\text{cl}}}
	\arrow[from=1-1, to=2-1]
	\arrow[from=2-1, to=3-1]
	\arrow[""{name=0, anchor=center, inner sep=0}, from=3-1, to=3-2]
	\arrow[from=2-1, to=2-2]
	\arrow[from=2-2, to=3-2]
	\arrow[from=1-2, to=2-2, "d"]
	\arrow["\sim", from=1-1, to=1-2]
	\arrow["\lrcorner"{anchor=center, pos=0.125}, draw=none, from=2-1, to=0]
\end{tikzcd}
\end{equation}
This gives us a fiber sequence
$$\Lambda^k \to \B^k_{\nabla} \Rb \to \B\B^{k-1}_{\nabla}\Rb,$$
which we may deloop as much as we like. Noting that $\Omega \B^k_{\nabla} \Rb =
\flat \B^{k -1} \Rb$, we therefore have a long exact sequence:
\[
0 \to H^{k-1}(X; \flat \Rb) \to H^{0,k-1}_{\nabla}(X;\Rb) \to \Lambda^k(X) \to
H^{0,k}_{\nabla}(X; \Rb) \to H^{1,k-1}(X ; \Rb) \to \cdots
\]
From this, we see that the surjectivity of the map $\Lambda^k(X) \to H^{0, k}(X;
\Rb)$ is determined by the vanishing of the map $H^{0,k}_{\nabla}(X; \Rb) \to
H^{1,k-1}(X; \Rb)$. Furthermore, the version of Diagram
\ref{eqn:obstruction.diagram} for $k-1$ shows us
that $d : \Lambda^{k-1}(X) \to \Lambda^k(X)$ factors through
$H^{0,k}_{\nabla}(X; \Rb)$. This means that for the kernel of $\Lambda^k(X) \to
H^{0, k}_{\nabla}(X; \Rb)$ to be the image of $d : \Lambda^{k-1}(X) \to \Lambda^k(X)$, we
need for $\Lambda^{k-1}(X) \to H^{0, k-1}_{\nabla}(X; \Rb)$ to be surjective;
this is controlled by the vanishing of $H^{0, k-1}_{\nabla}(X; \Rb) \to H^{1,
  k-2}_{\nabla}(X ; \Rb)$. In general, we see the obstructions to having exact
sequences
$$\Lambda^{k-1}(X) \to \Lambda^k(X) \to H^{0,k}_{\nabla}(X;\Rb)$$
lie in $H^{1, k-1}_{\nabla}(X; \Rb)$ and $H^{1, k-2}_{\nabla}(X; \Rb)$.

First cohomological degree pure differential cohomology groups also control
obstructions to de Rham's theorem for general types $X$. By definition we have a fiber sequence $\B_{\nabla}^k \Rb \to \Lambda^{k +
  1}_{\text{cl}} \to \flat \B^{k + 1} \Rb$ which may be delooped arbitrarily. We
therefore get exact sequences
$$0 \to H^{k}(X; \flat \Rb) \to H^{0,k}_{\nabla}(X; \Rb) \to \Lambda^{k +
  1}_{\text{cl}}(X) \to H^{k + 1}(X; \flat \Rb) \to H^{1,k}_{\nabla}(X; \Rb)\cdots$$

This exact sequence shows us that the surjectivity of the map $\Lambda^{k +
  1}_{\text{cl}}(X) \to H^{k + 1}(X; \flat \Rb)$ is controlled by the vanishing
of the map $H^{k + 1}(X; \flat \Rb) \to H^{1,k}_{\nabla}(X; \Rb)$. Furthermore,
in order for the kernel of $\Lambda^{k + 1}_{\text{cl}}(X) \to H^{k + 1}(X;
\flat \Rb)$ to be $d : \Lambda^k(X) \to \Lambda^{k + 1}_{\text{cl}}$, we need
for $\Lambda^k(X) \to H^{0, k}_{\nabla}(X; \Rb)$ to be surjective. As we saw
above, for the map $\Lambda^k(X) \to H^{0,k}_{\nabla}(X; \Rb)$ to be surjective,
we must have that $H^{0,k}_{\nabla}(X; \Rb) \to H^{1, k-1}_{\nabla}(X; \Rb)$ vanishes.

Remembering the classical, external differential cohomology hexagon, we are led
to the following conjecture:
\begin{conjecture}
Let $X$ be a crisp smooth manifold. Then $H^{1, k}_{\nabla}(X; \Rb)$ vanishes
for all $k$.
\end{conjecture}

 \subsection{Abstract Ordinary Differential Cohomology}\label{sec:abstract.ordinary.diff.coh}
  In the above sections, we constructed ordinary differential cohomology from the
  assumption of a long exact sequence of form classifiers. Apart from the concrete
  differential geometric input of the form classifiers, the construction was
  entirely abstract. In this section, we will describe the abstract ordinary
  differential cohomology theory from an axiomatic perspective. 

 The role of the form classifiers will be played by a \emph{contractible and
   infinitesimal resolution} of a crisp abelian group $C$. 

 \begin{defn}\label{defn:cir}
  Let $C$ be a crisp abelian group. A \emph{contractible
    and infinitesimal resolution} (CIR) of $C$ is a crisp long exact sequence
  \[
0 \to \flat C \to C \xto{d} C_1 \xto{d} C_2 \xto{d} \cdots
  \]
  where the $C_n$ are homotopically contractible --- $\shape C_n = \ast$ --- and where the kernels $Z_n :\equiv \ker(d : C_n
  \to C_{n +1})$ are infinitesimal --- $\flat Z_n = \ast$. We may think of $C_n$ as the
  abelian group of $n$-cochains, and $Z_n$ as the abelian group of $n$-cocycles.
 \end{defn}

 \begin{rmk}
   In an $\infty$-topos of sheaves of homotopy types, an abelian group $C$ in the
   empty context (which would therefore be crisp) is a sheaf of abelian groups.
   In this setting, we can understand a contractible and infinitesimal
   resolution of $C$ as presenting a cohomology theory on the site. The $C_n$
   are the sheaves of $n$-cochains, and the $Z_n$ the sheaves of $n$-cocycles.
   To suppose that the chain complex $d : C_n \to C_{n + 1}$ is exact is to say
   that representables have vanishing cohomology. To say that $Z_n$ is
   infinitesimal for $n > 0$ is to say that there is a unique $n$-cocycle on the
   terminal sheaf, namely $0$. To say that the $C_n$ are contractible may be
   understood as saying that for any two objects of the site, there is a
   homotopically unique \emph{concordance} between any $n$-cochains on them. 
 \end{rmk}

\begin{rmk}
It's likely that the generality could be pushed even further by taking $C$ to be
a spectrum and giving the following definition of a contractible and
infinitesimal resolution of
$C$:
\begin{itemize}
  \item Two sequences $C_n$ and $Z_n$ of spectra, $n \geq 0$, with $C_0 = C$ and
    $Z_0 = \flat C$. We may think of $C_n$ as the spectrum of $n$-cochains, and
    $Z_n$ as the spectrum of $n$-cocyles.
  \item Fiber sequences $Z_n \xto{i} C_n \xto{d} Z_{n+1}$ in which all maps $d$ are
    $\shape$-fibrations, and where $i_0 : Z_0 \to C_0$ is $(-)_{\flat} : \flat C
    \to C$.
  \item The $C_n$ are contractible, and the $Z_n$ are infinitesimal.
  \end{itemize}
This definition re-expresses the long exact sequence $C_{n-1} \xto{d} C_n
\xto{d} C_{n+1}$ in terms of the short exact sequences
\[0 \to Z_n \to C_n
\xto{d} Z_{n+1} \to 0\]
where $Z_n \equiv \ker(d : C_n \to C_{n+1})$. As we have no concrete examples at
this level of generality in mind, we leave the details of this
generalization to future work.
\end{rmk}

 For the rest of this section, we fix a crisp abelian group $C$ and an
 contractible and infinitesimal resolution of it. We can then prove analogues of
 the lemmas in the above sections. We begin by an analogue of \cref{thm:form.classifiers.contractible}.

 \begin{lem}
  Let $C$ be a crisp abelian group and $C_{\bullet}$ a contractible and
   infinitesimal resolution of $C$. Then $d : C \to Z_1$ is the infinitesimal
   remainder of $C$.
 \end{lem}
 \begin{proof}
   By hypothesis, we have a short exact sequence
   $$0 \to \flat C \to C \xto{d} Z_1 \to 0.$$
   We therefore have a long fiber sequence
   $$C \xto{d} Z_1 \to \flat \B C \to \B C$$
   which exhibits $d : C \to Z_1$ as the infinitesimal remainder of $C$. 
 \end{proof}

 \begin{thm}\label{thm:shape.of.cocycle.classifier}
   For $C$ a crisp abelian group and $C_{\bullet}$ a contractible and
   infinitesimal resolution of $C$, we have
   \[
   \shape Z_n = \flat \B^n \inftycover{C} .
\]
 \end{thm}
 \begin{proof}
   The same as the proof of \cref{thm:form.classifiers.contractible}.
 \end{proof}

 We now define analogues of the ordinary differential geometry classifiers
 $\B^n\B^k_{\nabla} U(1)$.

 \begin{defn}\label{defn:abstract.diff.coh}
   For $n, k \geq 0$, define $\B^n D_k$ to be the following pullback:
\[\begin{tikzcd}
	{\B^nD_k} & {\B^nZ_{k+1}} \\
	{\shape\B^{n + k} C} & {\flat\B^{n + k+1}\inftycover{C}}
	\arrow[from=1-1, to=1-2, "\B^n F_{-}"]
	\arrow[from=1-1, to=2-1]
	\arrow[from=2-1, to=2-2]
	\arrow[from=1-2, to=2-2, "(-)^{\shape}"]
	\arrow["\lrcorner"{anchor=center, pos=0.125}, draw=none, from=1-1, to=2-2]
\end{tikzcd}\]
We refer to $F_{(-)} : D_k \to Z_{k + 1}$ as the \emph{curvature}.
 \end{defn}
 
We begin by noting that $D_0$ is simply $C$. This went without saying before; we refrained from defining $\B^0_{\nabla}U(1)$, but if we had it would have
been $U(1)$.
 \begin{lem}
   As abelian groups, $D_0 = C$.
 \end{lem}
 \begin{proof}
   The defining pullback of $D_0$ is
   
\[\begin{tikzcd}
	{D_0} & {Z_1} \\
	{\shape C} & {\flat\B^{1}\inftycover{C}}
	\arrow[from=1-1, to=1-2, "F_{-}"]
	\arrow[from=1-1, to=2-1]
	\arrow[from=2-1, to=2-2]
	\arrow[from=1-2, to=2-2, "(-)^{\shape}"]
	\arrow["\lrcorner"{anchor=center, pos=0.125}, draw=none, from=1-1, to=2-2]
\end{tikzcd}\]
But $Z_1$ is the infinitesimal remainder of $C$, so the right square in the
modal fracture hexagon of $C$ shows that $C$ is the pullback of the same diagram.
 \end{proof}

We can prove an analogue of \cref{lem:diff.coh.shape.naturality}.
 \begin{lem}
   The defining diagram 
\[\begin{tikzcd}
	{\B^nD_k} & {\B^nZ_{k+1}} \\
	{\shape\B^{n + k} C} & {\flat\B^{n + k+1}\inftycover{C}}
	\arrow[from=1-1, to=1-2, "\B^n F_{-}"]
	\arrow[from=1-1, to=2-1]
	\arrow[from=2-1, to=2-2]
	\arrow[from=1-2, to=2-2, "(-)^{\shape}"]
	\arrow["\lrcorner"{anchor=center, pos=0.125}, draw=none, from=1-1, to=2-2]
\end{tikzcd}\]
is a $\shape$-naturality square. In particular, $\shape D_k = \shape \B^k C$.
 \end{lem}
 \begin{proof}
   Since $\shape \B^{n + k} C$ is discrete, it suffices to show that the fibers
   of $\B^n D_k \to \shape B^{n + k} C$ are $\shape$-connected. But they are
   equivalent to the fibers of $(-)^{\shape} : \B^n Z_{k + 1} \to \flat \B^{n + k +
     1} \inftycover{C}$, which are $\shape$-contractible.
 \end{proof}

 As a corollary, we can deduce an analogue of \cref{prop:universal.infty.cover.diff.coh}.
\begin{cor}
  The curvature $F_{(-)} : D_k \to Z_{k + 1}$ induces an equivalence on
  universal $\infty$-covers: $\inftycover{D}_k = \inftycover{Z}_{k + 1}$.
\end{cor}
\begin{proof}
  The defining pullback
\[\begin{tikzcd}
	{D_k} & {Z_{k+1}} \\
	{\shape\B^{k} C} & {\flat\B^{k+1}\inftycover{C}}
	\arrow[from=1-1, to=1-2, "F_{-}"]
	\arrow[from=1-1, to=2-1]
	\arrow[from=2-1, to=2-2]
	\arrow[from=1-2, to=2-2, "(-)^{\shape}"]
	\arrow["\lrcorner"{anchor=center, pos=0.125}, draw=none, from=1-1, to=2-2]
\end{tikzcd}\]
induces an equivalence on the fibers of the vertical maps. Since these maps are
$\shape$-units, the fibers are by definition the respective universal $\infty$-covers.
\end{proof}

 As with $\B^k_{\nabla}U(1)$, we may see $D_k$ as equipping $k$-gerbes with band
 $C$ with cocycle data coming from $Z_{k + 1}$. We have the analogue of Diagram \ref{eqn.big_diff_coh_diagram}:
 \begin{equation}\label{eqn:big.diff.coh.diagram.abstract}
\begin{tikzcd}
	{D_k} & {Z_{k + 1}} \\
	{\B D_k} & {\B Z_k} \\
	\vdots & \vdots \\
	{\B^{k-1} D_1} & {\B^{k-1}Z_2} \\
	{\B^k C} & {\B^k Z_1} \\
	{\shape\B^{k}C} & {\flat \B^{k+1}\Rb}
	\arrow[from=5-1, to=6-1]
	\arrow[from=5-1, to=5-2]
	\arrow[from=5-2, to=6-2]
	\arrow[from=6-1, to=6-2]
	\arrow["\lrcorner"{anchor=center, pos=0.125}, draw=none, from=5-1, to=6-2]
	\arrow[from=4-1, to=5-1]
	\arrow[from=4-1, to=4-2]
	\arrow[from=4-2, to=5-2]
	\arrow["\lrcorner"{anchor=center, pos=0.125}, draw=none, from=4-1, to=5-2]
	\arrow[from=1-1, to=2-1]
	\arrow[from=1-1, to=1-2]
	\arrow[from=1-2, to=2-2]
	\arrow[from=2-1, to=2-2]
	\arrow["\lrcorner"{anchor=center, pos=0.125}, draw=none, from=1-1, to=2-2]
	\arrow[from=2-2, to=3-2]
	\arrow[from=3-2, to=4-2]
	\arrow[from=2-1, to=3-1]
	\arrow[from=3-1, to=4-1]
\end{tikzcd}
\end{equation}

By applying $\flat$ to this diagram and recalling that the $Z_i$ (and therefore
their deloopings) are infinitesimal, we see that $\flat D_k = \flat \B^k C$. We
can use a composite square from this diagram to prove an analogue of \cref{lem.curvature}.

\begin{lem}
The fiber of the curvature $F_{(-)} : D_k \to Z_{k+1}$ is $\flat \B^k C$. We are
therefore justified in seeing the curvature as an obstruction to the flatness of
the underlying gerbe.
\end{lem}
\begin{proof}
The defining pullback

\[\begin{tikzcd}
	{D_k} & {Z_{k+1}} \\
	{\shape\B^{k} C} & {\flat\B^{k+1}\inftycover{C}}
	\arrow[from=1-1, to=1-2, "F_{-}"]
	\arrow[from=1-1, to=2-1]
	\arrow[from=2-1, to=2-2]
	\arrow[from=1-2, to=2-2, "(-)^{\shape}"]
	\arrow["\lrcorner"{anchor=center, pos=0.125}, draw=none, from=1-1, to=2-2]
\end{tikzcd}\]
induces an equivalence on the fibers of the horizontal maps. But the fiber of
$\shape \B^k C \to \flat \B^{k + 1} \inftycover{C}$ is $\flat \B^k C$.
\end{proof}
\begin{cor}
The curvature $F_{(-)} : D_k \to Z_{k + 1}$ is a unit for the $( k - 1
)$-truncation modality.
\end{cor}

Finally, we record an analogue to \cref{lem:inf.remainder.curvature}. 
\begin{lem}
  The infinitesimal remainder of $D_k$ is the curvature $F_{(-)} : D_k \to Z_{k
    + 1}$.
\end{lem}
\begin{proof}
Exactly as \cref{lem:inf.remainder.curvature}.
\end{proof}

We may put these results together to find the modal fracture hexagon of $D_k$.

\begin{thm}
The modal fracture hexagon of $D_k$ (\cref{defn:abstract.diff.coh}) is:
\begin{equation}\label{eqn:abstract.hexagon}
    \begin{tikzcd}
      & \inftycover{Z_{k+1}} \ar[dr, "\pi"] \ar[rr] & & Z_{k+1} 
      \ar[dr,  "(-)^{\shape}"] \\
      \flat \B^k \inftycover{C} \ar[ur, "(-)_{\flat}"] \ar[dr] & &
       D_k \ar[dr, "(-)^{\shape}"'] \ar[ur, "F_{(-)}"] & & \flat\B^{k+1}\inftycover{C} \\
      & \flat \B^k C \ar[rr] \ar[ur, "(-)_{\flat}"'] & & \shape\B^{k}C \ar[ur] &
    \end{tikzcd}
    \end{equation}
  \end{thm}

\subsection{A combinatorial analogue of differential cohomology}\label{sec:combinatorial.analogues}

Our arguments in the preceeding sections have been abstract and modal in
character. This abstraction means that we can apply these arguments in settings
other than differential geometry. In this subsection, we will sketch a
combinatorial analogue of differential cohomology taking place in the cohesive
$\infty$-topos of symmetric simplicial homotopy types. We will
mix internal and external reasoning in sketching the setup. We will give a fuller ---
and properly internal ---
exploration of symmetric simplicial cohesion in future work.

A symmetric simplicial homotopy type $S$ is an $\infty$-functor $X : \textbf{Fin}_{>0} \to
\textbf{H}$ from the category of non-empty finite sets into the
$\infty$-category of homotopy types. These are the unordered analogue of
simplicial homotopy types.

The $\infty$-topos of symmetric simplicial homotopy types is cohesive. The modalities operate on a symmetric simplicial
homotopy type $X$ in the following ways:
\begin{itemize}
  \item $\flat X$ is the discrete ($0$-skeletal) inclusion of $X_0 \equiv X([0])$ the
    homotopy type of $0$-simplices in $X$.
  \item $\shape X$ is the discrete inclusion of the \emph{geometric realization}
    (or colimit) of $X$.
  \item $\sharp X$ is the codiscrete ($0$-coskeletal) inclusion of $X_0$.
\end{itemize}

We have thus far avoided using the codiscrete modality $\sharp$ in this paper,
but it plays a crucial role in this section. This is because the $n$-simplex
$\Delta[n]$ may be defined to be the codiscrete reflection of the $(n +
1)$-element set $[n] \equiv \{0, \ldots, n\}$.
$$\Delta[n] :\equiv \sharp[n].$$
We may therefore axiomatize symmetric simplical cohesion internally with the
following axiom:
\begin{axiom}[Symmetric Simplicial Cohesion]
A crisp type $X$ is crisply discrete if and only if it is $\sharp[n]$-local for
all $n$.
\end{axiom}

We may therefore define $\shape = \text{Loc}_{\{\sharp[n]\mid n : \Nb\}}$ to be
the localization at the simplices, and the Symmetric Simplicial Cohesion axiom
will ensure that $\shape$ is adjoint to $\flat$ as required by the Unity of
Opposites axiom.

In his paper \cite{Lawvere:Codiscrete.Cohesion}, Lawvere points out that the simplices $\Delta[n]$ are
\emph{tiny}, much like the infinitesimal disks in synthetic differential
geometry. $\Delta[n]$ being tiny means the functor $(-)^{\Delta[n]}$ admits an
\emph{external right} adjoint. We may refer to this adjoint as
$(-)^{\frac{1}{\Delta[n]}}$, following Lawvere. If $C$ is a crisp codiscrete abelian
group, then the external adjointness shows that maps $X \to C^{\frac{1}{\Delta[n]}}$
correspond to maps $X^{\Delta[n]} \to C$, which, since $C$ is codiscrete,
correspond to maps $\flat X^{\Delta_n} \to C$; but $\flat X^{\Delta_n} =
X^{\Delta_n}([0]) = X([n])$, so such maps ultimate correspond to maps $X([n]) \to C$ --- that is, to $C$-valued $n$-cochains on the
symmetric simplicial homotopy type $X$! In total, $C^{\frac{1}{\Delta[n]}}$
classifies $n$-cochains, much in the way that $\Lambda^n$ classifies
differential $n$-forms. We note that $C^{\frac{1}{\Delta[n]}}$ inherits the
(crisp) algebraic structure of $C$ since $(-)^{\frac{1}{\Delta[n]}}$ is a right
adjoint.

If furthermore $C$ is a ring, then the
$C^{\frac{1}{\Delta[n]}}$ will be modules and since codiscretes are contractible
in this topos (by Theorem 10.2 of \cite{Shulman:Real.Cohesion}, noting that it
satisfies Shulman's Axiom C2), we see that the
$C^{\frac{1}{\Delta[n]}}$ are contractible. We may use the face inclusions
$\Delta[n] \to \Delta[n + 1]$ to give maps $C^{\frac{1}{\Delta[n]}} \to
C^{\frac{1}{\Delta[n + 1]}}$, and taking their alternating sum gives us a chain
complex 
$$C \xto{d} C^{\frac{1}{\Delta[1]}} \xto{d} C^{\frac{1}{\Delta[2]}} \xto{d}
\cdots$$

Reasoning externally, we can see that this sequence will be exact since the
$C$-valued cohomology of the $n$-simplices is trivial. Furthermore, since $\flat C_k =
\flat C$ by adjointness ($\flat (\ast \to C_k) = \flat (\ast^{\Delta_k} \to
C)$), we see that the $Z_k$ are infinitesimal: $$\flat Z_k = \ker(\flat d : \flat C_k \to \flat C_{k + 1}) = \ast.$$
For this reason, we may
make the following assumption in the setting of symetric simplicial cohesion,
mirroring \cref{assumption:form.classifiers} of the existence of form
classifiers in synthetic differential cohesion. 

\begin{assumption}\label{assumption.cochain.classifiers}
  Let $C$ be a codiscrete ring, and define $C_n :\equiv C^{\frac{1}{\Delta[n]}}$. Then the sequence
  $$0 \to \flat C \to C \xto{d} C_1 \xto{d} C_2 \xto{d} \cdots$$
  forms a contractible and infinitesimal resolution of $C$.
\end{assumption}

We can now interpret the abstract language of \cref{sec:abstract.ordinary.diff.coh} into the more
concrete language of \cref{assumption.cochain.classifiers}:
\begin{itemize}
  \item We have begun with a codiscrete abelian group $C$. We note that since
    $C$ is codiscrete, it is homotopically contractible: $\shape C = \ast$. Therefore, $\inftycover{C} = C$.
  \item The abelian groups $C_k$ are the $k$-cochain classifiers.
  \item The kernels $Z_k :\equiv \ker(d : C_k \to C_{k + 1})$ classify
    $k$-cocycles. Applying \cref{thm:shape.of.cocycle.classifier} here
    shows us that
    $$\shape Z_k = \flat B^k C.$$
    From this, we see that cohomology valued in the discrete group $\flat C$ is the
    universal discrete cohomological invariant of $k$-cocycles value in $C$. This
    justifies a remark of Lawvere in \cite{Lawvere:Codiscrete.Cohesion} that the
    $Z_k$ have the homotopy type of the Eilenberg-MacLane space $K(\flat C, k)$.
  \item Since $C$ is contractible, we have that $D_k$ as defined in
    \cref{defn:abstract.diff.coh} is the universal $\infty$-cover $\inftycover{Z}_{k + 1}$ of
    $Z_{k + 1}$.
    We see that $D_k$ classifies $(k + 1)$-cocycles together with witnesses that
    their induced cohomology class vanishes in $\flat \B^{k+1} C$.
\end{itemize}

The $D_k$ in this setting have more in common with pure differential cohomology
$B_{\nabla}^k \Rb$ than with ordinary differential cohomology $\B_{\nabla}^k
U(1)$ on account of being contractible. We can remedy this by introducing some
new data. Suppose that we have an exact sequence
$$0 \to K \to C \to G \to 0$$
of crisp codiscrete abelian groups. We may then redefine $D_k$ to instead
be the following pullback: 

\[\begin{tikzcd}
	{D_k} & {Z_{k + 1}} \\
	{\flat \B^{k + 1} K} & {\flat \B^{k + 1} C}
	\arrow[from=1-1, to=2-1]
	\arrow[from=1-1, to=1-2]
	\arrow[from=1-2, to=2-2]
	\arrow[from=2-1, to=2-2]
	\arrow["\lrcorner"{anchor=center, pos=0.125}, draw=none, from=1-1, to=2-2]
\end{tikzcd}\]
 We will then have $\shape D_k = \flat
B^{k + 1} K$, $\inftycover{D}_k = \inftycover{Z}_{k + 1}$, and $\flat D_k =
\flat \B^k G$, giving us a modal fracture hexagon:
\begin{equation}\label{eqn:abstract.hexagon.simplicial}
    \begin{tikzcd}
      & \inftycover{Z}_{k+1} \ar[dr, "\pi"] \ar[rr] & & Z_{k+1} 
      \ar[dr,  "(-)^{\shape}"] \\
      \flat \B^k C \ar[ur, "(-)_{\flat}"] \ar[dr] & &
       D_k \ar[dr, "(-)^{\shape}"'] \ar[ur, "F_{(-)}"] & & \flat\B^{k+1}{C} \\
      & \flat \B^k G  \ar[rr, "\beta"'] \ar[ur, "(-)_{\flat}"'] & & \flat \B^{k+1} K \ar[ur] &
    \end{tikzcd}
    \end{equation}

    Taking the short exact sequence $0 \to K \to C \to G \to 0$ to be
    $$0 \to \sharp \Zb \to \sharp \Rb \to \sharp U(1) \to 0$$
    gives us a bona-fide combinatorial analogue of ordinary differential
    cohomology, fitting within a similar hexagon. However, instead of equipping
   the integral cohomology of manifolds with differential form data, we are
   equipping the integral cohomology of symmetric simplicial sets with real
   cocycle data.
   
  We intend to give this combinatorial analogue of ordinary differential
  cohomology a fully internal treatment in future work.

   \printbibliography

\end{document}